\documentclass[preprint,authoryear]{elsarticle}
\usepackage{epsfig}
\usepackage{amsmath}
\usepackage{amstext}
\usepackage{amsfonts}
\usepackage{amssymb}
\usepackage{eucal}
\usepackage{graphicx}
\usepackage{verbatim}
\usepackage{bm}
\usepackage{color}

\usepackage{tikz}
\usepackage{pgfplots}

\usetikzlibrary{arrows,petri,topaths}
\usepackage{tkz-berge}

\newcommand{\SN}{\hat{S}_{N}^{-1}}
\newcommand{\gammanrho}{\gamma_N(\rho)}
\newcommand{\qnrho}{\frac{\alpha(\rho)}{\gamma_N(\rho)}}
\newcommand{\Si}{\hat{S}_{(i)}^{-1}}
\newcommand{\Sj}{\hat{S}_{(j)}^{-1}}
\newcommand{\Sij}{\hat{S}_{(i,j)}^{-1}}
\newcommand{\Sjk}{\hat{S}_{(j,k)}^{-1}}
\newcommand{\Sijk}{\hat{S}_{(i,j,k)}^{-1}}
\newcommand{\Sijsquare}{\hat{S}_{(i,j)}^{-2}}

\newcommand{\RR}{{\mathbb{R}}}
\newcommand{\NN}{{\mathbb{N}}}
\newcommand{\ZZ}{{\mathbb{Z}}}
\newcommand{\CC}{{\mathbb{C}}}

\newcommand{\trans}{{\sf T}}

\newcommand{\first}{c}
\newcommand{\second}{d}

\newcommand{\asto}{\overset{\rm a.s.}{\longrightarrow}}
\newcommand{\lawto}{\overset{\mathcal L}{\longrightarrow}}

\newcommand{\EE}{{\rm E}}
\newcommand{\E}{{\rm E}}
\DeclareMathOperator{\tr}{tr}

\DeclareMathOperator{\argmin}{argmin}

\newproof{proof}{Proof}

\newcounter{ctheorem}
\newtheorem{theorem}[ctheorem]{Theorem}
\newcounter{cassumption}
\newtheorem{assumption}[cassumption]{Assumption}

\newcounter{cproposition}
\newtheorem{proposition}[cproposition]{Proposition}
\newcounter{ccorollary}
\newtheorem{corollary}[ccorollary]{Corollary}
\newcounter{clemma}
\newtheorem{lemma}[clemma]{Lemma}

\journal{Journal of Multivariate Analysis}

\begin{document}

\begin{frontmatter}

\title{Second order statistics of robust estimators of scatter. \\ Application to GLRT detection for elliptical signals\tnoteref{t1}}
\tnotetext[t1]{ Couillet's work is supported by the ERC MORE EC--120133. Pascal's work is supported by the DGA grant no. 2013.60.0011.00.470.75.01.}

\author[supelec]{Romain Couillet}
\ead{romain.couillet@supelec.fr}
\address[supelec]{Telecommunication department, Sup\'elec, Gif sur Yvette, France}

\author[kaust]{Abla Kammoun}
\ead{abla.kammoun@kaust.edu.sa}
\address[kaust]{King Abdullah's University of Science and Technology, Saudi Arabia}

\author[sondra]{Fr\'ed\'eric Pascal}
\ead{frederic.pascal@supelec.fr}
\address[sondra]{SONDRA Laboratory, Sup\'elec, Gif sur Yvette, France}

%\maketitle
\begin{abstract}
	A central limit theorem for bilinear forms of the type $a^*\hat{C}_N(\rho)^{-1}b$, where $a,b\in\CC^N$ are unit norm deterministic vectors and $\hat{C}_N(\rho)$ a robust-shrinkage estimator of scatter parametrized by $\rho$ and built upon $n$ independent elliptical vector observations, is presented. The fluctuations of $a^*\hat{C}_N(\rho)^{-1}b$ are found to be of order $N^{-\frac12}$ and to be the same as those of $a^*\hat{S}_N(\rho)^{-1}b$ for $\hat{S}_N(\rho)$ a matrix of a theoretical tractable form. This result is exploited in a classical signal detection problem to provide an improved detector which is both robust to elliptical data observations (e.g., impulsive noise) and optimized across the shrinkage parameter $\rho$.
\end{abstract}	


\end{frontmatter}

\section{Introduction}
\label{sec:intro}

As an aftermath of the growing interest for large dimensional data analysis in machine learning, in a recent series of articles \citep{COU13,COU13b,COU14,ZHA14,KAR13}, several estimators from the field of robust statistics (dating back to the seventies) started to be explored under the assumption of commensurably large sample ($n$) and population ($N$) dimensions. Robust estimators were originally designed to turn classical estimators into outlier- and impulsive noise-resilient estimators, which is of considerable importance in the recent big data paradigm. Among these estimation methods, robust regression was studied in \citep{KAR13} which reveals that, in the large $N,n$ regime, the difference in norm between estimated and true regression vectors (of size $N$) tends almost surely to a positive constant which depends on the nature of the data and of the robust regressor. In parallel, and of more interest to the present work, \citep{COU13,COU13b,COU14,ZHA14} studied the limiting behavior of several classes of robust estimators $\hat{C}_N$ of scatter (or covariance) matrices $C_N$ based on independent zero-mean elliptical observations $x_1,\ldots,x_n\in\CC^N$. Precisely, \citep{COU13} shows that, letting $N/n<1$ and $\hat{C}_N$ be the (almost sure) unique solution to
\begin{align*}
	\hat{C}_N &= \frac1n \sum_{i=1}^n u\left( \frac1Nx_i^*\hat{C}_N^{-1}x_i \right) x_ix_i^*
\end{align*}
under some appropriate conditions over the nonnegative function $u$ (corresponding to Maronna's M-estimator \citep{MAR76}), $\Vert\hat{C}_N-\hat{S}_N\Vert\asto 0$ in spectral norm as $N,n\to\infty$ with $N/n\to c\in(0,1)$, where $\hat{S}_N$ follows a standard random matrix model (such as studied in \citep{CHO95,HAC13}). In \citep{ZHA14}, the important scenario where $u(x)=1/x$ (referred to as Tyler's M-estimator) is treated. It is in particular shown for this model that for identity scatter matrices the spectrum of $\hat{C}_N$ converges weakly to the Mar\u{c}enko--Pastur law \citep{MAR67} in the large $N,n$ regime. Finally, for $N/n\to c\in(0,\infty)$, \citep{COU14} studied yet another robust estimation model defined, for each $\rho\in(\max\{0,1-n/N\},1]$, by $\hat{C}_N=\hat{C}_N(\rho)$, unique solution to
\begin{align}
	\label{eq:hatCN}
	\hat{C}_N(\rho) &= \frac1n \sum_{i=1}^n \frac{x_ix_i^*}{\frac1Nx_i^*\hat{C}_N^{-1}(\rho)x_i} + \rho I_N.
\end{align}
This estimator, proposed in \citep{PAS13}, corresponds to a hybrid robust-shrinkage estimator reminding Tyler's M-estimator of scale \citep{TYL87} and Ledoit--Wolf's shrinkage estimator \citep{LED04}. This estimator is particularly suited to scenarios where $N/n$ is not small, for which other estimators are badly conditioned if not undefined. For this model, it is shown in \citep{COU14} that $\sup_{\rho}\Vert\hat{C}_N(\rho)-\hat{S}_N(\rho)\Vert\asto 0$ where $\hat{S}_N(\rho)$ also follows a classical random matrix model. 

The aforementioned approximations $\hat{S}_N$ of the estimators $\hat{C}_N$, the structure of which is well understood (as opposed to $\hat{C}_N$ which is only defined implicitly), allow for both a good apprehension of the limiting behavior of $\hat{C}_N$ and more importantly for a better usage of $\hat{C}_N$ as an appropriate substitute for sample covariance matrices in various estimation problems in the large $N,n$ regime. The convergence in norm $\Vert\hat{C}_N-\hat{S}_N\Vert\asto 0$ is indeed sufficient in many cases to produce new consistent estimation methods based on $\hat{C}_N$ by simply replacing $\hat{C}_N$ by $\hat{S}_N$ in the problem defining equations. For example, the results of \citep{COU13b} led to the introduction of novel consistent estimators based on functionals of $\hat{C}_N$ (of the Maronna type) for power and direction-of-arrival estimation in array processing in the presence of impulsive noise or rare outliers \citep{COU14c}. Similarly, in \citep{COU14}, empirical methods were designed to estimate the parameter $\rho$ which minimizes the expected Frobenius norm $\tr [(\hat{C}_N(\rho)-C_N)^2]$, of interest for various outlier-prone applications dealing with non-small ratios $N/n$.\footnote{Other metrics may also be considered as in e.g.\@ \citep{YAN14} with $\rho$ chosen to minimize the return variance in a portfolio optimization problem.}

Nonetheless, when replacing $\hat{C}_N$ for $\hat{S}_N$ in deriving consistent estimates, if the convergence $\Vert\hat{C}_N-\hat{S}_N\Vert\asto 0$ helps in producing novel consistent estimates, this convergence (which comes with no particular speed) is in general not sufficient to assess the performance of the estimator for large but finite $N,n$. Indeed, when second order results such as central limit theorems need be established, say at rate $N^{-\frac12}$, to proceed similarly to the replacement of $\hat{C}_N$ by $\hat{S}_N$ in the analysis, one would ideally demand that $\Vert\hat{C}_N-\hat{S}_N\Vert=o(N^{-\frac12})$; but such a result, we believe, unfortunately does not hold. This constitutes a severe limitation in the exploitation of robust estimators as their performance as well as optimal fine-tuning often rely on second order performance. Concretely, for practical purposes in the array processing application of \citep{COU14c}, one may naturally ask which choice of the $u$ function is optimal to minimize the variance of (consistent) power and angle estimates. This question remains unanswered to this point for lack of better theoretical results.

\bigskip

The main purpose of the article is twofold. From a technical aspect, taking the robust shrinkage estimator $\hat{C}_N(\rho)$ defined by \eqref{eq:hatCN} as an example, we first show that, although the convergence $\Vert\hat{C}_N(\rho)-\hat{S}_N(\rho)\Vert\asto 0$ (from \citep[Theorem~1]{COU14}) may not be extensible to a rate $O(N^{1-\varepsilon})$, one has the bilinear form convergence $N^{1-\varepsilon} a^*(\hat{C}_N^k(\rho)-\hat{S}_N^k(\rho))b\asto 0$ for each $\varepsilon>0$, each $a,b\in\CC^N$ of unit norm, and each $k\in\ZZ$. This result implies that, if $\sqrt{N}a^*\hat{S}_N^k(\rho)b$ satisfies a central limit theorem, then so does $\sqrt{N}a^*\hat{C}_N^k(\rho)b$ with the same limiting variance. This result is of fundamental importance to any statistical application based on such quadratic forms. Our second contribution is to exploit this result for the specific problem of signal detection in impulsive noise environments via the generalized likelihood-ratio test, particularly suited for radar signals detection under elliptical noise \citep{CON95,PAS13}. In this context, we determine the shrinkage parameter $\rho$ which minimizes the probability of false detections and provide an empirical consistent estimate for this parameter, thus improving significantly over traditional sample covariance matrix-based estimators.

The remainder of the article introduces our main results in Section~\ref{sec:results} which are proved in Section~\ref{sec:proof}. Technical elements of proof are provided in the appendix.

{\it Notations:} In the remainder of the article, we shall denote $\lambda_1(X),\ldots,\lambda_n(X)$ the real eigenvalues of $n\times n$ Hermitian matrices $X$. The norm notation $\Vert \cdot\Vert$ being considered is the spectral norm for matrices and Euclidean norm for vectors. The symbol $\imath$ is the complex $\sqrt{-1}$.

\section{Main Results}
\label{sec:results}

Let $N,n\in\NN$, $c_N\triangleq N/n$, and $\rho \in (\max\{0,1-c_N^{-1}\},1]$. 
Let also $x_1,\ldots,x_n\in\CC^N$ be $n$ independent random vectors defined by the following assumptions.
\begin{assumption}[Data vectors]
	\label{ass:x}
For $i\in\{1,\ldots,n\}$, $x_i=\sqrt{\tau_i}A_Nw_i=\sqrt{\tau_i}z_i$, where
\begin{itemize}
	\item $w_i\in\CC^N$ is Gaussian with zero mean and covariance $I_N$, independent across $i$;
	\item $A_NA_N^*\triangleq C_N\in\CC^{N\times N}$ is such that $\nu_N\triangleq \frac1N\sum_{i=1}^N{\bm\delta}_{\lambda_i(C_N)}\to \nu$ weakly, $\limsup_N\Vert C_N\Vert <\infty$, and $\frac1N\tr C_N=1$;
	\item $\tau_i>0$ are random or deterministic scalars. 
\end{itemize}
\end{assumption}
Under Assumption~\ref{ass:x}, letting $\tau_i=\tilde{\tau}_i/\Vert w_i\Vert$ for some $\tilde{\tau}_i$ independent of $w_i$, $x_i$ belongs to the class of elliptically distributed random vectors. Note that the normalization $\frac1N\tr C_N=1$ is not a restricting constraint since the scalars $\tau_i$ may absorb any other normalization.

It has been well-established by the robust estimation theory that, even if the $\tau_i$ are independent, independent of the $w_i$, and that $\lim_n \frac1n \sum_{i=1}^n\tau_i=1$ a.s., the sample covariance matrix $\frac1n\sum_{i=1}^n x_ix_i^*$ is in general a poor estimate for $C_N$. Robust estimators of scatter were designed for this purpose \citep{MAR76,TYL87}. In addition, if $N/n$ is non trivial, a linear shrinkage of these robust estimators against the identity matrix often helps in regularizing the estimator as established in e.g., \citep{PAS13,CHE11}. The robust estimator of scatter considered in this work, which we denote $\hat{C}_N(\rho)$, is defined (originally in \citep{PAS13}) as the unique solution to
\begin{align*}
	\hat{C}_N(\rho) &= (1-\rho) \frac1n \sum_{i=1}^n \frac{x_ix_i^*}{\frac1Nx_i\hat{C}_N^{-1}(\rho)x_i} + \rho I_N.
\end{align*}

\subsection{Theoretical Results}

The asymptotic behavior of this estimator was studied recently in \citep{COU14} in the regime where $N,n\to\infty$ in such a way that $c_N\to c\in(0,\infty)$. We first recall the important results of this article, which shall lay down the main concepts and notations of the present work.
First define
\begin{align*}
	\hat{S}_N(\rho) &= \frac1{\gamma_N(\rho)}\frac{1-\rho}{1-(1-\rho)c_N} \frac1n\sum_{i=1}^n z_iz_i^* + \rho I_N
\end{align*}
where $\gamma_N(\rho)$ is the unique solution to
\begin{align*}
	1 &= \int \frac{t}{\gamma_N(\rho)\rho + (1-\rho)t}\nu_N(dt).
\end{align*}

For any $\kappa>0$ small, define $\mathcal R_\kappa\triangleq [\kappa+\max\{0,1-c^{-1}\},1]$. Then, from \cite[Theorem~1]{COU14}, as $N,n\to\infty$ with $c_N\to c\in(0,\infty)$,
\begin{align*}
	\sup_{\rho \in \mathcal R_\kappa} \left\Vert \hat{C}_N(\rho) - \hat{S}_N(\rho) \right\Vert \asto 0.
\end{align*}

A careful analysis of the proof of \cite[Theorem~1]{COU14} (which is performed in Section~\ref{sec:proof}) reveals that the above convergence can be refined as
\begin{align}
	\label{eq:cv12-eps}
	\sup_{\rho \in \mathcal R_\kappa} N^{\frac12-\varepsilon} \left\Vert \hat{C}_N(\rho) - \hat{S}_N(\rho) \right\Vert \asto 0
\end{align}
for each $\varepsilon>0$. This suggests that (well-behaved) functionals of $\hat{C}_N(\rho)$ fluctuating at a slower speed than $N^{-\frac12+\varepsilon}$ for some $\varepsilon>0$ follow the same statistics as the same functionals with $\hat{S}_N(\rho)$ in place of $\hat{C}_N(\rho)$. However, this result is quite weak as most limiting theorems (starting with the classical central limit theorems for independent scalar variables) deal with fluctuations of order $N^{-\frac12}$ and sometimes in random matrix theory of order $N^{-1}$. In our opinion, the convergence speed \eqref{eq:cv12-eps} cannot be improved to a rate $N^{-\frac12}$. Nonetheless, thanks to an averaging effect documented in Section~\ref{sec:proof}, the fluctuation of special forms of functionals of $\hat{C}_N(\rho)$ can be proved to be much slower. Although among these functionals we could have considered linear functionals of the eigenvalue distribution of $\hat{C}_N(\rho)$, our present concern (driven by more obvious applications) is rather on bilinear forms of the type $a^*\hat{C}_N^k(\rho)b$ for some $a,b\in\CC^N$ with $\Vert a\Vert=\Vert b\Vert=1$, $k\in\ZZ$.

Our first main result is the following.
\begin{theorem}[Fluctuation of bilinear forms]
	\label{th:bilin}
Let $a,b\in\CC^N$ with $\Vert a\Vert=\Vert b\Vert=1$. Then, as $N,n\to\infty$ with $c_N\to c\in(0,\infty)$, for any $\varepsilon>0$ and every $k\in\ZZ$,
\begin{align*}
	\sup_{\rho\in\mathcal R_\kappa} N^{1-\varepsilon} \left| a^*\hat{C}_N^{k}(\rho)b - a^*\hat{S}_N^{k}(\rho)b \right| &\asto 0.
\end{align*}
\end{theorem}

Some comments and remarks are in order. First, we recall that central limit theorems involving bilinear forms of the type $a^*\hat{S}_N^k(\rho)b$ are classical objects in random matrix theory (see e.g.\@ \citep{KAM09,MES08} for $k=-1$), particularly common in signal processing and wireless communications. These central limit theorems in general show fluctuations at speed $N^{-\frac12}$. This indicates, taking $\varepsilon<\frac12$ in Theorem~\ref{th:bilin} and using the fact that almost sure convergence implies weak convergence, that $a^*\hat{C}_N^k(\rho)b$ exhibits the same fluctuations as $a^*\hat{S}_N^k(\rho)b$, the latter being classical and tractable while the former is quite intricate at the onset, due to the implicit nature of $\hat{C}_N(\rho)$. 

Of practical interest to many applications in signal processing is the case where $k=-1$. In the next section, we present a classical generalized maximum likelihood signal detection in impulsive noise, for which we shall characterize the shrinkage parameter $\rho$ that meets minimum false alarm rates.

\subsection{Application to Signal Detection}

In this section, we consider the hypothesis testing scenario by which an $N$-sensor array receives a vector $y\in\CC^N$ according to the following hypotheses
\begin{align*}
	y &= \left\{ 
		\begin{array}{ll}
			x &,~\mathcal H_0 \\
			\alpha p + x&,~\mathcal H_1
		\end{array}
		\right.
\end{align*}
in which $\alpha>0$ is some unknown scaling factor constant while $p\in\CC^N$ is deterministic and known at the sensor array (which often corresponds to a steering vector arising from a specific known angle), and $x$ is an impulsive noise distributed as $x_1$ according to Assumption~\ref{ass:x}.
For convenience, we shall take $\Vert p\Vert=1$.

Under $\mathcal H_0$ (the null hypothesis), a noisy observation from an impulsive source is observed while under $\mathcal H_1$ both information and noise are collected at the array. The objective is to decide on $\mathcal H_1$ versus $\mathcal H_0$ upon the observation $y$ and prior pure-noise observations $x_1,\ldots,x_n$ distributed according to Assumption~\ref{ass:x}. When $\tau_1,\ldots,\tau_n$ and $C_N$ are unknown, the corresponding generalized likelihood ratio test, derived in \citep{CON95}, reads
\begin{align*}
	T_N(\rho) &\overset{\mathcal H_1}{\underset{\mathcal H_0}{\gtrless}} \Gamma
\end{align*}
for some detection threshold $\Gamma$ where
\begin{align*}
	T_N(\rho) \triangleq \frac{|y^*\hat{C}_N^{-1}(\rho)p|}{\sqrt{y^*\hat{C}_N^{-1}(\rho)y}\sqrt{p^*\hat{C}_N^{-1}(\rho)p}}.
\end{align*}
More precisely, \citep{CON95} derived the detector $T_N(0)$ only valid when $n\geq N$. The relaxed detector $T_N(\rho)$ allows for a better conditioning of the estimator, in particular for $n\simeq N$. In \citep{PAS13}, $T_N(\rho)$ is used explicitly in a space-time adaptive processing setting but only simulation results were provided. Alternative metrics for similar array processing problems involve the signal-to-noise ratio loss minimization rather than likelihood ratio tests; in \citep{ABR13a,ABR13b}, the authors exploit the estimators $\hat{C}_N(\rho)$ but restrict themselves to the less tractable finite dimensional analysis.

Our objective is to characterize the false alarm performance of the detector. That is, provided $\mathcal H_0$ is the actual scenario (i.e.\@ $y=x$), we shall evaluate $P(T_N(\rho)>\Gamma)$. Since it shall appear that, under $\mathcal H_0$, $T_N(\rho)\asto 0$ for every fixed $\Gamma>0$ and every $\rho$, by dominated convergence $P(T_N(\rho)>\Gamma)\to 0$ which does not say much about the actual test performance for large but finite $N,n$. To avoid such empty statements, we shall then consider the non-trivial case where $\Gamma=N^{-\frac12}\gamma$ for some fixed $\gamma>0$. In this case our objective is to characterize the false alarm probability
\begin{align*}
	P\left( T_N(\rho) > \frac{\gamma}{\sqrt{N}} \right).
\end{align*}

Before providing this result, we need some further reminders from \citep{COU14}. First define
\begin{align*}
	\underline{\hat{S}}_N(\rho) &\triangleq (1-\rho)\frac1n\sum_{i=1}^n z_iz_i^* + \rho I_N.
\end{align*}
Then, from \cite[Lemma~1]{COU14}, for each $\rho\in (\max\{0,1-c^{-1}\},1]$,
\begin{align*}
	\frac{\hat{S}_N(\rho)}{\rho+\frac1{\gamma_N(\rho)}\frac{1-\rho}{1-(1-\rho)c}} = \underline{\hat{S}}_N(\underline\rho)
\end{align*}
where
\begin{align*}
	 \underline\rho \triangleq \frac{{\rho}}{\rho+\frac1{\gamma_N(\rho)}\frac{1-\rho}{1-(1-\rho)c}}.
\end{align*}
Moreover, the mapping $\rho\mapsto \underline\rho$ is continuously increasing from $(\max\{0,1-c^{-1}\},1]$ onto $(0,1]$.

From classical random matrix considerations (see e.g.\@ \citep{SIL95}), letting $Z=[z_1,\ldots,z_n]\in\CC^{N\times n}$, the empirical spectral distribution\footnote{That is the normalized counting measure of the eigenvalues.} of $(1-\rho)\frac1nZ^*Z$ almost surely admits a weak limit $\mu$. The Stieltjes transform $m(z)\triangleq \int (t-z)^{-1}\mu(dt)$ of $\mu$ at $z\in\CC\setminus {\rm Supp}(\mu)$ is the unique complex solution with positive (resp.\@ negative) imaginary part if $\Im[z]>0$ (resp.\@ $\Im[z]<0$) and unique real positive solution if $\Im[z]=0$ and $\Re[z]<0$ to
\begin{align*}
	m(z) &= \left( -z + c\int \frac{(1-\rho)t}{1+(1-\rho)tm(z)} \nu(dt) \right)^{-1}.
\end{align*}
We denote $m'(z)$ the derivative of $m(z)$ with respect to $z$ (recall that the Stieltjes transform of a positively supported measure is analytic, hence continuously differentiable, away from the support of the measure).

With these definitions in place and with the help of Theorem~\ref{th:bilin}, we are now ready to introduce the main result of this section.
\begin{theorem}[Asymptotic detector performance]
	\label{th:T}
Under hypothesis $\mathcal H_0$, as $N,n\to\infty$ with $c_N\to c\in(0,\infty)$,
\begin{align*}
	\sup_{\rho\in\mathcal R_\kappa} \left| P\left( T_N(\rho) > \frac{\gamma}{\sqrt{N}} \right) - \exp  \left( - \frac{\gamma^2}{2\sigma_N^2(\underline\rho)} \right) \right| &\to 0
\end{align*}
where $\rho\mapsto\underline\rho$ is the aforementioned mapping and
\begin{align*}
	\sigma_N^2(\underline{\rho}) &\triangleq \frac12 \frac{ p^*C_NQ_N^2(\underline\rho)p}{  p^*Q_N(\underline\rho)p \cdot \frac1N\tr C_NQ_N(\underline\rho)\cdot \left(1-c (1-\rho)^2 m(-\underline{\rho})^2 \frac1N\tr C_N^2Q_N^2(\underline\rho)\right)  }
%	\\ \sigma_N^2(\underline{\rho}) &\triangleq \frac12 \frac{ p^*C_N\left( I_N + (1-\underline{\rho})m(-\underline{\rho})C_N \right)^{-2}p}{  p^*\left( I_N+(1-\underline{\rho}) m(-\underline{\rho})C_N \right)^{-1}p \int \frac{t \nu_N(dt)}{ 1+(1-\underline{\rho}) m(-\underline{\rho})t} \left(1-c\int \frac{m(-\underline{\rho})^2(1-\underline{\rho})^2t^2\nu_N(dt)}{(1+(1-\underline{\rho})tm(-\underline{\rho}))^2} \right)  }.
\end{align*}
with $Q_N(\underline\rho)\triangleq (I_N + (1-\underline{\rho})m(-\underline{\rho})C_N)^{-1}$.
%Moreover, for each integer $p\geq 1$,
%\begin{align*}
%	\EE\left[ N^{\frac{p}2} T(\rho)^p \right] - 2^{\frac{p}2}\sigma_N(\underline\rho)^p \Gamma\left(1+\frac{p}2\right) &\to 0
%\end{align*}
%with $\Gamma$ the Gamma function.
\end{theorem}
Otherwise stated, $\sqrt{N}T_N(\rho)$ is uniformly well approximated by a Rayleigh distributed random variable $R_N(\underline\rho)$ with parameter $\sigma_N(\underline\rho)$. % and all its moments are well approximated by the moments of $R$.
Simulation results are provided in Figure~\ref{fig:hist_detector_20} and Figure~\ref{fig:hist_detector_100} which corroborate the results of Theorem~\ref{th:T} for $N=20$ and $N=100$, respectively (for a single value of $\rho$ though). Comparatively, it is observed, as one would expect, that larger values for $N$ induce improved approximations in the tails of the approximating distribution.

\begin{figure}[h!]
  \centering
  \begin{tabular}{cc}
  \begin{tikzpicture}[font=\footnotesize,scale=.7]
    \renewcommand{\axisdefaulttryminticks}{4} 
    \tikzstyle{every major grid}+=[style=densely dashed]       
    \tikzstyle{every axis y label}+=[yshift=-10pt] 
    \tikzstyle{every axis x label}+=[yshift=5pt]
    \tikzstyle{every axis legend}+=[cells={anchor=west},fill=white,
        at={(0.98,0.98)}, anchor=north east, font=\scriptsize ]
    \begin{axis}[
      %ybar,
      xmin=0,
      ymin=0,
      xmax=4,
      bar width=3pt,
      grid=major,
      ymajorgrids=false,
      scaled ticks=true,
      %scale ticks above={4},
      %xlabel={Values taken},
      ylabel={Density}
      ]
      \addplot+[ybar,mark=none,color=black,fill=gray!40!white] coordinates{
(0.05,0.055)(0.15,0.138)(0.25,0.263)(0.35,0.324)(0.45,0.42)(0.55,0.516)(0.65,0.544)(0.75,0.626)(0.85,0.621)(0.95,0.631)(1.05,0.639)(1.15,0.618)(1.25,0.625)(1.35,0.588)(1.45,0.515)(1.55,0.512)(1.65,0.417)(1.75,0.395)(1.85,0.331)(1.95,0.3)(2.05,0.232)(2.15,0.186)(2.25,0.143)(2.35,0.093)(2.45,0.081)(2.55,0.065)(2.65,0.053)(2.75,0.031)(2.85,0.019)(2.95,0.01)(3.05,0.006)(3.15,0.002)(3.25,0.)(3.35,0.001)(3.45,0.)(3.55,0.)(3.65,0.)(3.75,0.)(3.85,0.)(3.95,0.)(4.05,0.)(4.15,0.)(4.25,0.)(4.35,0.)(4.45,0.)(4.55,0.)(4.65,0.)(4.75,0.)(4.85,0.)(4.95,0.)(5.05,0.)
      };
      \addplot[black,smooth,line width=1pt] plot coordinates{
	      (0.,0.)(0.1,0.109902)(0.2,0.21619)(0.3,0.315448)(0.4,0.40464)(0.5,0.481262)(0.6,0.543458)(0.7,0.590088)(0.8,0.620745)(0.9,0.635727)(1.,0.635965)(1.1,0.62292)(1.2,0.598449)(1.3,0.564673)(1.4,0.523829)(1.5,0.478146)(1.6,0.429733)(1.7,0.380485)(1.8,0.332025)(1.9,0.285669)(2.,0.24241)(2.1,0.202932)(2.2,0.167636)(2.3,0.136674)(2.4,0.109997)(2.5,0.087403)(2.6,0.068576)(2.7,0.053135)(2.8,0.040662)(2.9,0.030736)(3.,0.02295)(3.1,0.01693)(3.2,0.012338)(3.3,0.008885)(3.4,0.006321)(3.5,0.004445)(3.6,0.003088)(3.7,0.00212)(3.8,0.001439)(3.9,0.000965)(4.,0.00064)(4.1,0.000419)(4.2,0.000271)(4.3,0.000174)(4.4,0.00011)(4.5,0.000069)(4.6,0.000042)(4.7,0.000026)(4.8,0.000016)(4.9,0.000009)(5.,0.000006)	      
      };
      \legend{ {Empirical hist.\@ of $T_N(\rho)$},{Distribution of $R_N(\underline\rho)$} }
    \end{axis}
  \end{tikzpicture}
  &
  \begin{tikzpicture}[font=\footnotesize,scale=.7]
    \renewcommand{\axisdefaulttryminticks}{4} 
    \tikzstyle{every major grid}+=[style=densely dashed]       
    \tikzstyle{every axis y label}+=[yshift=-10pt] 
    \tikzstyle{every axis x label}+=[yshift=5pt]
    \tikzstyle{every axis legend}+=[cells={anchor=west},fill=white,
        at={(0.98,0.02)}, anchor=south east, font=\scriptsize ]
    \begin{axis}[
      %ybar,
      xmin=0,
      ymin=0,
      xmax=4,
      ymax=1,
      bar width=1.5pt,
      grid=major,
      ymajorgrids=false,
      scaled ticks=true,
      %scale ticks above={4},
      mark repeat=10,
      ylabel={Cumulative distribution}
      ]
      \addplot[black,mark=*] coordinates{
	      (0.,0.)(0.006969,0.01)(0.140717,0.02)(0.20486,0.03)(0.244764,0.04)(0.283584,0.05)(0.315349,0.06)(0.346699,0.07)(0.377842,0.08)(0.406186,0.09)(0.429656,0.1)(0.451618,0.11)(0.475542,0.12)(0.500291,0.13)(0.520232,0.14)(0.539641,0.15)(0.559926,0.16)(0.577158,0.17)(0.596035,0.18)(0.616435,0.19)(0.638308,0.2)(0.657572,0.21)(0.673424,0.22)(0.68921,0.23)(0.704637,0.24)(0.719824,0.25)(0.736868,0.26)(0.752971,0.27)(0.768967,0.28)(0.784846,0.29)(0.801679,0.3)(0.819511,0.31)(0.835377,0.32)(0.848875,0.33)(0.866909,0.34)(0.880432,0.35)(0.89928,0.36)(0.913679,0.37)(0.932125,0.38)(0.948459,0.39)(0.965796,0.4)(0.977908,0.41)(0.992774,0.42)(1.010217,0.43)(1.023342,0.44)(1.039632,0.45)(1.053273,0.46)(1.068361,0.47)(1.086139,0.48)(1.105446,0.49)(1.121136,0.5)(1.136039,0.51)(1.151077,0.52)(1.168546,0.53)(1.184198,0.54)(1.200874,0.55)(1.21731,0.56)(1.233584,0.57)(1.248894,0.58)(1.266999,0.59)(1.28289,0.6)(1.297469,0.61)(1.314827,0.62)(1.331183,0.63)(1.34982,0.64)(1.366348,0.65)(1.381957,0.66)(1.398737,0.67)(1.418269,0.68)(1.435907,0.69)(1.456034,0.7)(1.476974,0.71)(1.496225,0.72)(1.514232,0.73)(1.530731,0.74)(1.551749,0.75)(1.570442,0.76)(1.591612,0.77)(1.614273,0.78)(1.63925,0.79)(1.664769,0.8)(1.685157,0.81)(1.709925,0.82)(1.733259,0.83)(1.759881,0.84)(1.787953,0.85)(1.816399,0.86)(1.844413,0.87)(1.872439,0.88)(1.906016,0.89)(1.938662,0.9)(1.974716,0.91)(2.008711,0.92)(2.050453,0.93)(2.097168,0.94)(2.146087,0.95)(2.201662,0.96)(2.268444,0.97)(2.365899,0.98)(2.481182,0.99)(2.634444,1.)	      
      };
      \addplot[black,smooth] plot coordinates{
	      (0.,0.)(0.1,0.00551)(0.2,0.02186)(0.3,0.048514)(0.4,0.084613)(0.5,0.129022)(0.6,0.180384)(0.7,0.237194)(0.8,0.297868)(0.9,0.36082)(1.,0.424522)(1.1,0.487569)(1.2,0.548725)(1.3,0.606949)(1.4,0.661424)(1.5,0.711554)(1.6,0.756962)(1.7,0.797473)(1.8,0.833086)(1.9,0.863948)(2.,0.890323)(2.1,0.912556)(2.2,0.931049)(2.3,0.946228)(2.4,0.958527)(2.5,0.968364)(2.6,0.976133)(2.7,0.982192)(2.8,0.986859)(2.9,0.990409)(3.,0.993078)(3.1,0.995058)(3.2,0.996511)(3.3,0.997564)(3.4,0.998318)(3.5,0.998851)(3.6,0.999224)(3.7,0.999481)(3.8,0.999657)(3.9,0.999776)(4.,0.999855)(4.1,0.999908)(4.2,0.999942)(4.3,0.999963)(4.4,0.999977)(4.5,0.999986)(4.6,0.999992)(4.7,0.999995)(4.8,0.999997)(4.9,0.999998)(5.,0.999999)	      
      };
      \legend{ {Empirical dist.\@ of $T_N(\rho)$},{Distribution of $R_N(\underline\rho)$} }
    \end{axis}
  \end{tikzpicture}
  \end{tabular}
  \caption{Histogram distribution function of the $\sqrt{N}T_N(\rho)$ versus $R_N(\underline\rho)$, $N=20$, $p=N^{-\frac12}[1,\ldots,1]^\trans$, $[C_N]_{ij}=0.7^{|i-j|}$, $c_N=1/2$, $\rho=0.2$.}
  \label{fig:hist_detector_20}
\end{figure}
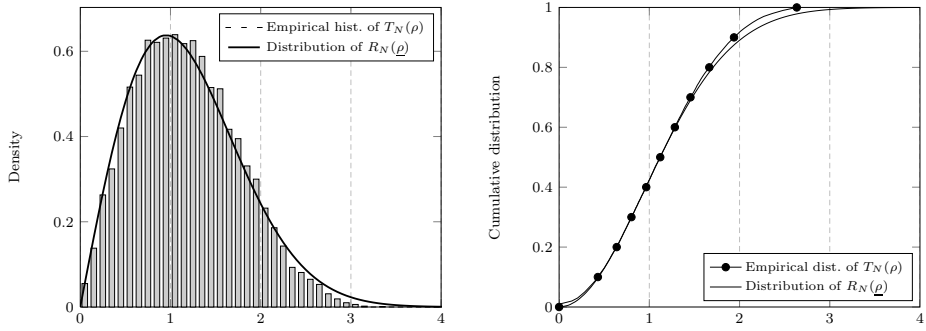

\begin{figure}[h!]
  \centering
  \begin{tabular}{cc}
  \begin{tikzpicture}[font=\footnotesize,scale=.7]
    \renewcommand{\axisdefaulttryminticks}{4} 
    \tikzstyle{every major grid}+=[style=densely dashed]       
    \tikzstyle{every axis y label}+=[yshift=-10pt] 
    \tikzstyle{every axis x label}+=[yshift=5pt]
    \tikzstyle{every axis legend}+=[cells={anchor=west},fill=white,
        at={(0.98,0.98)}, anchor=north east, font=\scriptsize ]
    \begin{axis}[
      %ybar,
      xmin=0,
      ymin=0,
      xmax=4,
      bar width=3pt,
      grid=major,
      ymajorgrids=false,
      scaled ticks=true,
      %scale ticks above={4},
      %xlabel={Values taken},
      ylabel={Density}
      ]
      \addplot+[ybar,mark=none,color=black,fill=gray!40!white] coordinates{
	      (0.05,0.062)(0.15,0.152)(0.25,0.292)(0.35,0.332)(0.45,0.412)(0.55,0.515)(0.65,0.576)(0.75,0.592)(0.85,0.614)(0.95,0.669)(1.05,0.638)(1.15,0.617)(1.25,0.615)(1.35,0.546)(1.45,0.532)(1.55,0.426)(1.65,0.391)(1.75,0.378)(1.85,0.331)(1.95,0.252)(2.05,0.225)(2.15,0.185)(2.25,0.144)(2.35,0.125)(2.45,0.098)(2.55,0.063)(2.65,0.066)(2.75,0.052)(2.85,0.032)(2.95,0.021)(3.05,0.017)(3.15,0.007)(3.25,0.008)(3.35,0.006)(3.45,0.003)(3.55,0.)(3.65,0.003)(3.75,0.002)(3.85,0.)(3.95,0.001)(4.05,0.)(4.15,0.)(4.25,0.)(4.35,0.)(4.45,0.)(4.55,0.)(4.65,0.)(4.75,0.)(4.85,0.)(4.95,0.)(5.05,0.)
      };
      \addplot[black,smooth,line width=1pt] plot coordinates{
	      (0.,0.)(0.1,0.109131)(0.2,0.214698)(0.3,0.313332)(0.4,0.402036)(0.5,0.478332)(0.6,0.540381)(0.7,0.587044)(0.8,0.617904)(0.9,0.633237)(1.,0.633945)(1.1,0.621449)(1.2,0.597572)(1.3,0.564395)(1.4,0.524123)(1.5,0.478956)(1.6,0.430981)(1.7,0.382081)(1.8,0.333873)(1.9,0.287674)(2.,0.244483)(2.1,0.204995)(2.2,0.169624)(2.3,0.138538)(2.4,0.111702)(2.5,0.088927)(2.6,0.069911)(2.7,0.054281)(2.8,0.041628)(2.9,0.031536)(3.,0.023602)(3.1,0.017452)(3.2,0.01275)(3.3,0.009204)(3.4,0.006566)(3.5,0.004629)(3.6,0.003225)(3.7,0.002221)(3.8,0.001511)(3.9,0.001017)(4.,0.000676)(4.1,0.000444)(4.2,0.000289)(4.3,0.000185)(4.4,0.000118)(4.5,0.000074)(4.6,0.000046)(4.7,0.000028)(4.8,0.000017)(4.9,0.00001)(5.,0.000006)
      };
      \legend{ {Empirical hist.\@ of $T_N(\rho)$},{Density of $R_N(\underline\rho)$} }
    \end{axis}
  \end{tikzpicture}
  &
  \begin{tikzpicture}[font=\footnotesize,scale=.7]
    \renewcommand{\axisdefaulttryminticks}{4} 
    \tikzstyle{every major grid}+=[style=densely dashed]       
    \tikzstyle{every axis y label}+=[yshift=-10pt] 
    \tikzstyle{every axis x label}+=[yshift=5pt]
    \tikzstyle{every axis legend}+=[cells={anchor=west},fill=white,
        at={(0.98,0.02)}, anchor=south east, font=\scriptsize ]
    \begin{axis}[
      %ybar,
      xmin=0,
      ymin=0,
      xmax=4,
      ymax=1,
      bar width=1.5pt,
      grid=major,
      ymajorgrids=false,
      scaled ticks=true,
      %scale ticks above={4},
      mark repeat=10,
      ylabel={Cumulative distribution}
      ]
      \addplot[black,mark=*] coordinates{
	      (0.,0.)(0.004147,0.01)(0.134196,0.02)(0.194839,0.03)(0.231715,0.04)(0.268037,0.05)(0.29753,0.06)(0.326162,0.07)(0.361765,0.08)(0.387635,0.09)(0.415288,0.1)(0.441611,0.11)(0.464966,0.12)(0.485413,0.13)(0.509661,0.14)(0.529458,0.15)(0.548931,0.16)(0.567224,0.17)(0.585663,0.18)(0.606538,0.19)(0.622153,0.2)(0.639809,0.21)(0.657092,0.22)(0.674893,0.23)(0.691094,0.24)(0.709848,0.25)(0.723726,0.26)(0.742494,0.27)(0.759849,0.28)(0.776713,0.29)(0.792786,0.3)(0.809041,0.31)(0.826186,0.32)(0.843159,0.33)(0.860687,0.34)(0.876359,0.35)(0.891876,0.36)(0.908243,0.37)(0.921541,0.38)(0.935376,0.39)(0.950716,0.4)(0.966973,0.41)(0.982105,0.42)(0.996863,0.43)(1.010958,0.44)(1.031895,0.45)(1.04803,0.46)(1.059607,0.47)(1.074379,0.48)(1.090857,0.49)(1.105972,0.5)(1.12501,0.51)(1.14239,0.52)(1.157615,0.53)(1.171878,0.54)(1.188458,0.55)(1.204927,0.56)(1.222216,0.57)(1.237269,0.58)(1.251976,0.59)(1.271192,0.6)(1.287183,0.61)(1.301946,0.62)(1.319419,0.63)(1.338576,0.64)(1.355493,0.65)(1.377815,0.66)(1.394742,0.67)(1.415819,0.68)(1.433655,0.69)(1.451259,0.7)(1.468387,0.71)(1.486074,0.72)(1.509745,0.73)(1.534089,0.74)(1.555325,0.75)(1.579848,0.76)(1.601953,0.77)(1.626237,0.78)(1.650685,0.79)(1.678268,0.8)(1.704434,0.81)(1.728554,0.82)(1.757243,0.83)(1.782361,0.84)(1.810967,0.85)(1.840125,0.86)(1.869583,0.87)(1.904995,0.88)(1.944813,0.89)(1.984082,0.9)(2.025297,0.91)(2.068876,0.92)(2.11431,0.93)(2.168671,0.94)(2.231536,0.95)(2.303031,0.96)(2.383405,0.97)(2.477321,0.98)(2.621018,0.99)(2.800403,1.)	     
      };
      \addplot[black,smooth] plot coordinates{
(0.,0.)(0.1,0.005472)(0.2,0.021707)(0.3,0.04818)(0.4,0.084042)(0.5,0.128172)(0.6,0.179233)(0.7,0.235735)(0.8,0.296114)(0.9,0.358798)(1.,0.422273)(1.1,0.485146)(1.2,0.546184)(1.3,0.60435)(1.4,0.658826)(1.5,0.709012)(1.6,0.754524)(1.7,0.795178)(1.8,0.830964)(1.9,0.86202)(2.,0.888599)(2.1,0.91104)(2.2,0.929735)(2.3,0.945108)(2.4,0.957585)(2.5,0.967584)(2.6,0.975496)(2.7,0.981679)(2.8,0.986451)(2.9,0.99009)(3.,0.99283)(3.1,0.99487)(3.2,0.996369)(3.3,0.997458)(3.4,0.99824)(3.5,0.998795)(3.6,0.999184)(3.7,0.999453)(3.8,0.999638)(3.9,0.999762)(4.,0.999846)(4.1,0.999901)(4.2,0.999937)(4.3,0.999961)(4.4,0.999976)(4.5,0.999985)(4.6,0.999991)(4.7,0.999995)(4.8,0.999997)(4.9,0.999998)(5.,0.999999)
      };
      \legend{ {Empirical dist.\@ of $T_N(\rho)$},{Distribution of $R_N(\underline\rho)$} }
    \end{axis}
  \end{tikzpicture}
  \end{tabular}
  \caption{Histogram distribution function of the $\sqrt{N}T_N(\rho)$ versus $R_N(\underline\rho)$, $N=100$, $p=N^{-\frac12}[1,\ldots,1]^\trans$, $[C_N]_{ij}=0.7^{|i-j|}$, $c_N=1/2$, $\rho=0.2$.}
  \label{fig:hist_detector_100}
\end{figure}
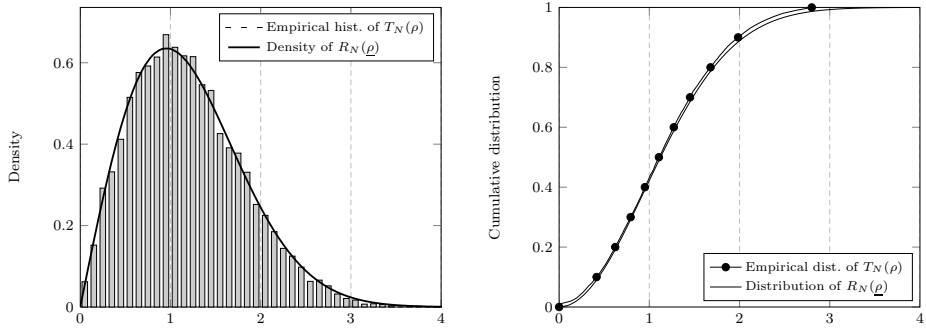

The result of Theorem~\ref{th:T} provides an analytical characterization of the performance of the GLRT for each $\rho$ which suggests in particular the existence of values for $\rho$ which minimize the false alarm probability for given $\gamma$. Note in passing that, independently of $\gamma$, minimizing the false alarm rate is asymptotically equivalent to minimizing $\sigma_N^2(\underline\rho)$ over $\rho$. However, the expression of $\sigma_N^2(\underline\rho)$ depends on the covariance matrix $C_N$ which is unknown to the array and therefore does not allow for an immediate online choice of an appropriate $\underline\rho$. To tackle this problem, the following proposition provides a consistent estimate for $\sigma_N^2(\underline\rho)$ based on $\hat{C}_N(\rho)$ and $p$.

\begin{proposition}[Empirical performance estimation]
	\label{prop:1}
	For $\rho\in(\max\{0,1-c_N^{-1}\},1)$ and $\underline\rho$ defined as above, let $\hat{\sigma}_N^2(\underline\rho)$ be given by
	\begin{align*}
		\hat{\sigma}_N^2(\underline\rho) &\triangleq \frac12 \frac{1-\underline\rho \cdot \frac{p^*\hat{C}_N^{-2}(\rho)p}{p^*\hat{C}_N^{-1}(\rho)p}\cdot \frac1N\tr \hat{C}_N(\rho)}{\left( 1-c + c\underline\rho \frac1N\tr\hat{C}_N^{-1}(\rho)\cdot \frac1N\tr \hat{C}_N(\rho) \right)\left( 1 - \underline\rho \frac1N\tr\hat{C}_N^{-1}(\rho)\cdot \frac1N\tr \hat{C}_N(\rho)\right)}.
	\end{align*}
	Also let $\hat{\sigma}_N^2(1)\triangleq \lim_{\underline\rho\uparrow 1}\hat{\sigma}_N^2(\underline\rho)$. Then we have
	\begin{align*}
		\sup_{\rho \in \mathcal R_\kappa} \left| \sigma_N^2(\underline\rho) - \hat{\sigma}_N^2(\underline\rho) \right| &\asto 0.
	\end{align*}
\end{proposition}

Since both the estimation of $\sigma_N^2(\underline\rho)$ in Proposition~\ref{prop:1} and the convergence in Theorem~\ref{th:T} are uniform over $\rho\in\mathcal R_\kappa$, we have the following result.
\begin{corollary}[Empirical performance optimum]
	\label{co:1}
	Let $\hat{\sigma}_N^2(\underline\rho)$ be defined as in Proposition~\ref{prop:1} and define $\hat{\rho}_N^*$ as any value satisfying
	\begin{align*}
		\hat{\rho}_N^* &\in \argmin_{ \rho\in \mathcal R_\kappa } \left\{ \hat{\sigma}_N^2(\underline\rho) \right\}
	\end{align*}
	(this set being in general a singleton). Then, for every $\gamma>0$,
	\begin{align*}
		P\left( \sqrt{N}T_N(\hat{\rho}_N^*) > \gamma \right) - \inf_{\rho\in \mathcal R_\kappa} \left\{ P\left( \sqrt{N}T_N(\rho) > \gamma \right) \right\} &\to 0.
	\end{align*}
\end{corollary}

This last result states that, for $N,n$ sufficiently large, it is increasingly close-to-optimal to use the detector $T_N(\hat{\rho}_N^*)$ in order to reach minimal false alarm probability. A practical graphical confirmation of this fact is provided in Figure~\ref{fig:FAR} where, in the same scenario as in Figures~\ref{fig:hist_detector_20}--\ref{fig:hist_detector_100}, the false alarm rates for various values of $\gamma$ are depicted. In this figure, the black dots correspond to the actual values taken by $P(\sqrt{N}T_N(\rho)>\gamma)$ empirically obtained out of $10^6$ Monte Carlo simulations. The plain curves are the approximating values $\exp(-\gamma^2/(2\sigma_N(\underline\rho)^2))$. Finally, the white dots with error bars correspond to the mean and standard deviations of $\exp(-\gamma^2/(2\hat\sigma_N(\underline\rho)^2))$ for each $\underline\rho$, respectively.
 It is first interesting to note that the estimates $\hat\sigma_N(\underline\rho)$ are quite accurate, especially so for $N$ large, with standard deviations sufficiently small to provide good estimates, already for small $N$, of the false alarm minimizing $\rho$. However, similar to Figures~\ref{fig:hist_detector_20}--\ref{fig:hist_detector_100}, we observe a particularly weak approximation in the (small) $N=20$ setting for large values of $\gamma$, corresponding to tail events, while for $N=100$, these values are better recovered. This behavior is obviously explained by the fact that $\gamma=3$ is not small compared to $\sqrt{N}$ when $N=20$.

 Nonetheless, from an error rate viewpoint, it is observed that errors of order $10^{-2}$ are rather well approximated for $N=100$. In Figure~\ref{fig:FAR2}, we consider this observation in depth by displaying $P(T_N(\hat{\rho}_N^*)>\Gamma)$ and its approximation $\exp(-N\Gamma^2/(2\hat\sigma_N^2(\underline\rho)))$ for $N=20$ and $N=100$, for various values of $\Gamma$. This figures shows that even errors of order $10^{-4}$ are well approximated for large $N$, while only errors of order $10^{-2}$ can be evaluated for small $N$.\footnote{Note that a comparison against alternative algorithms that would use no shrinkage (i.e., by setting $\rho=0$) or that would not implement a robust estimate is not provided here, being of little relevance. Indeed, a proper selection of $c_N$ to a large value or $C_N$ with condition number close to one would provide an arbitrarily large gain of shrinkage-based methods, while an arbitrarily heavy-tailed choice of the $\tau_i$ distribution would provide a huge performance gain for robust methods. It is therefore not possible to compare such methods on fair grounds.}

\begin{figure}[h!]
  \centering
  \begin{tabular}{cc}
  \begin{tikzpicture}[font=\footnotesize,scale=.7]
    \renewcommand{\axisdefaulttryminticks}{4} 
    \tikzstyle{every major grid}+=[style=densely dashed]       
    \tikzstyle{every axis y label}+=[yshift=-10pt] 
    \tikzstyle{every axis x label}+=[yshift=5pt]
    \tikzstyle{every axis legend}+=[cells={anchor=west},fill=white,
        at={(0.98,0.02)}, anchor=south east, font=\scriptsize ]
    \begin{semilogyaxis}[
      %ybar,
      xmin=0,
      xmax=1,
      ymax=1,
      grid=major,
      ymajorgrids=false,
      scaled ticks=true,
      %scale ticks above={4},
      xlabel={$\rho$},
      ylabel={$P(\sqrt{N}T_N(\rho)>\gamma)$}
      ]
      \addplot[black] plot coordinates{
	      (0.010000,0.132470)(0.020000,0.129832)(0.030000,0.127406)(0.040000,0.125179)(0.050000,0.123139)(0.060000,0.121277)(0.070000,0.119582)(0.080000,0.118045)(0.090000,0.116660)(0.100000,0.115417)(0.110000,0.114311)(0.120000,0.113336)(0.130000,0.112486)(0.140000,0.111756)(0.150000,0.111141)(0.160000,0.110639)(0.170000,0.110244)(0.180000,0.109954)(0.190000,0.109766)(0.200000,0.109677)(0.210000,0.109685)(0.220000,0.109788)(0.230000,0.109983)(0.240000,0.110270)(0.250000,0.110648)(0.260000,0.111114)(0.270000,0.111669)(0.280000,0.112311)(0.290000,0.113040)(0.300000,0.113856)(0.310000,0.114758)(0.320000,0.115747)(0.330000,0.116822)(0.340000,0.117984)(0.350000,0.119233)(0.360000,0.120570)(0.370000,0.121996)(0.380000,0.123511)(0.390000,0.125116)(0.400000,0.126811)(0.410000,0.128599)(0.420000,0.130480)(0.430000,0.132456)(0.440000,0.134528)(0.450000,0.136696)(0.460000,0.138964)(0.470000,0.141332)(0.480000,0.143802)(0.490000,0.146376)(0.500000,0.149055)(0.510000,0.151841)(0.520000,0.154736)(0.530000,0.157743)(0.540000,0.160862)(0.550000,0.164095)(0.560000,0.167446)(0.570000,0.170915)(0.580000,0.174504)(0.590000,0.178216)(0.600000,0.182052)(0.610000,0.186014)(0.620000,0.190104)(0.630000,0.194324)(0.640000,0.198675)(0.650000,0.203159)(0.660000,0.207776)(0.670000,0.212530)(0.680000,0.217420)(0.690000,0.222448)(0.700000,0.227615)(0.710000,0.232922)(0.720000,0.238368)(0.730000,0.243955)(0.740000,0.249682)(0.750000,0.255549)(0.760000,0.261556)(0.770000,0.267703)(0.780000,0.273987)(0.790000,0.280408)(0.800000,0.286965)(0.810000,0.293655)(0.820000,0.300476)(0.830000,0.307425)(0.840000,0.314500)(0.850000,0.321698)(0.860000,0.329014)(0.870000,0.336445)(0.880000,0.343987)(0.890000,0.351635)(0.900000,0.359384)(0.910000,0.367230)(0.920000,0.375167)(0.930000,0.383190)(0.940000,0.391293)(0.950000,0.399471)(0.960000,0.407718)(0.970000,0.416027)(0.980000,0.424394)(0.990000,0.432813)(1.000000,0.441279)
      };
      \addplot[black,only marks,mark=o,mark options={scale=0.75},error bars/.cd,y dir=both,y explicit, error bar style={mark size=1.5pt}] plot coordinates{
%	      (0.050000,0.122862)+-(0.004666,0.004666)(0.100000,0.115304)+-(0.007083,0.007083)(0.150000,0.111296)+-(0.008336,0.008336)(0.200000,0.110100)+-(0.009052,0.009052)(0.250000,0.111252)+-(0.009382,0.009382)(0.300000,0.114462)+-(0.009643,0.009643)(0.350000,0.119670)+-(0.009998,0.009998)(0.400000,0.126848)+-(0.010446,0.010446)(0.450000,0.136146)+-(0.011166,0.011166)(0.500000,0.147449)+-(0.012283,0.012283)(0.550000,0.161179)+-(0.013842,0.013842)(0.600000,0.177356)+-(0.015766,0.015766)(0.650000,0.196272)+-(0.017992,0.017992)(0.700000,0.218029)+-(0.020721,0.020721)(0.750000,0.242943)+-(0.023712,0.023712)(0.800000,0.270972)+-(0.026973,0.026973)(0.850000,0.301811)+-(0.030579,0.030579)(0.900000,0.334868)+-(0.034531,0.034531)(0.950000,0.369376)+-(0.038917,0.038917)
	      (0.050000,0.122858)+-(0.004658,0.004658)(0.100000,0.115290)+-(0.007070,0.007070)(0.150000,0.111293)+-(0.008363,0.008363)(0.200000,0.110090)+-(0.009027,0.009027)(0.250000,0.111235)+-(0.009386,0.009386)(0.300000,0.114445)+-(0.009656,0.009656)(0.350000,0.119667)+-(0.009976,0.009976)(0.400000,0.126863)+-(0.010424,0.010424)(0.450000,0.136131)+-(0.011170,0.011170)(0.500000,0.147467)+-(0.012326,0.012326)(0.550000,0.161167)+-(0.013813,0.013813)(0.600000,0.177344)+-(0.015758,0.015758)(0.650000,0.196270)+-(0.018054,0.018054)(0.700000,0.218062)+-(0.020671,0.020671)(0.750000,0.242979)+-(0.023689,0.023689)(0.800000,0.270952)+-(0.026959,0.026959)(0.850000,0.301740)+-(0.030677,0.030677)(0.900000,0.334893)+-(0.034650,0.034650)(0.950000,0.369402)+-(0.038847,0.038847)	      
      };
      \addplot[black,only marks,mark=*,mark options={scale=0.75}] plot coordinates{
%	      (0.050000,0.102590)(0.100000,0.095720)(0.150000,0.091240)(0.200000,0.091620)(0.250000,0.092830)(0.300000,0.097810)(0.350000,0.102860)(0.400000,0.110760)(0.450000,0.122320)(0.500000,0.136140)(0.550000,0.151220)(0.600000,0.172100)(0.650000,0.196570)(0.700000,0.227680)(0.750000,0.259180)(0.800000,0.295480)(0.850000,0.338640)(0.900000,0.386430)(0.950000,0.431240)(1.000000,0.478980)
	      (0.050000,0.102554)(0.100000,0.095524)(0.150000,0.092681)(0.200000,0.091376) % (0.200000,0.090928)
	      (0.250000,0.093394)(0.300000,0.097212)(0.350000,0.103109)(0.400000,0.111273)(0.450000,0.121610)(0.500000,0.134944)(0.550000,0.152715)(0.600000,0.172006)(0.650000,0.197016)(0.700000,0.226062)(0.750000,0.258595)(0.800000,0.296503)(0.850000,0.339433)(0.900000,0.384739)(0.950000,0.432905)(1.000000,0.478739)	      
      };
      \addplot[black] plot coordinates{
	      (0.010000,0.010587)(0.020000,0.010118)(0.030000,0.009698)(0.040000,0.009321)(0.050000,0.008982)(0.060000,0.008680)(0.070000,0.008409)(0.080000,0.008168)(0.090000,0.007954)(0.100000,0.007764)(0.110000,0.007598)(0.120000,0.007453)(0.130000,0.007328)(0.140000,0.007221)(0.150000,0.007132)(0.160000,0.007060)(0.170000,0.007003)(0.180000,0.006962)(0.190000,0.006935)(0.200000,0.006922)(0.210000,0.006924)(0.220000,0.006938)(0.230000,0.006966)(0.240000,0.007007)(0.250000,0.007061)(0.260000,0.007128)(0.270000,0.007208)(0.280000,0.007302)(0.290000,0.007409)(0.300000,0.007530)(0.310000,0.007665)(0.320000,0.007814)(0.330000,0.007979)(0.340000,0.008158)(0.350000,0.008354)(0.360000,0.008566)(0.370000,0.008796)(0.380000,0.009043)(0.390000,0.009310)(0.400000,0.009596)(0.410000,0.009903)(0.420000,0.010232)(0.430000,0.010584)(0.440000,0.010960)(0.450000,0.011362)(0.460000,0.011790)(0.470000,0.012247)(0.480000,0.012734)(0.490000,0.013253)(0.500000,0.013805)(0.510000,0.014392)(0.520000,0.015017)(0.530000,0.015681)(0.540000,0.016388)(0.550000,0.017138)(0.560000,0.017936)(0.570000,0.018783)(0.580000,0.019682)(0.590000,0.020636)(0.600000,0.021649)(0.610000,0.022724)(0.620000,0.023863)(0.630000,0.025072)(0.640000,0.026352)(0.650000,0.027710)(0.660000,0.029147)(0.670000,0.030669)(0.680000,0.032279)(0.690000,0.033983)(0.700000,0.035785)(0.710000,0.037690)(0.720000,0.039702)(0.730000,0.041826)(0.740000,0.044068)(0.750000,0.046432)(0.760000,0.048924)(0.770000,0.051549)(0.780000,0.054312)(0.790000,0.057218)(0.800000,0.060272)(0.810000,0.063479)(0.820000,0.066845)(0.830000,0.070374)(0.840000,0.074071)(0.850000,0.077939)(0.860000,0.081985)(0.870000,0.086210)(0.880000,0.090619)(0.890000,0.095215)(0.900000,0.100002)(0.910000,0.104981)(0.920000,0.110156)(0.930000,0.115527)(0.940000,0.121096)(0.950000,0.126865)(0.960000,0.132834)(0.970000,0.139003)(0.980000,0.145372)(0.990000,0.151942)(1.000000,0.158710)
      };
      \addplot[black,only marks,mark=o,mark options={scale=0.75},error bars/.cd,y dir=both,y explicit, error bar style={mark size=1.5pt}] plot coordinates{
%	      (0.050000,0.008955)+-(0.000744,0.000744)(0.100000,0.007788)+-(0.001036,0.001036)(0.150000,0.007211)+-(0.001166,0.001166)(0.200000,0.007049)+-(0.001253,0.001253)(0.250000,0.007219)+-(0.001322,0.001322)(0.300000,0.007696)+-(0.001416,0.001416)(0.350000,0.008505)+-(0.001560,0.001560)(0.400000,0.009694)+-(0.001765,0.001765)(0.450000,0.011366)+-(0.002074,0.002074)(0.500000,0.013604)+-(0.002532,0.002532)(0.550000,0.016631)+-(0.003204,0.003204)(0.600000,0.020640)+-(0.004123,0.004123)(0.650000,0.025944)+-(0.005337,0.005337)(0.700000,0.032896)+-(0.007014,0.007014)(0.750000,0.041991)+-(0.009163,0.009163)(0.800000,0.053714)+-(0.011908,0.011908)(0.850000,0.068489)+-(0.015399,0.015399)(0.900000,0.086577)+-(0.019741,0.019741)(0.950000,0.108023)+-(0.025032,0.025032)
	      (0.050000,0.008954)+-(0.000743,0.000743)(0.100000,0.007786)+-(0.001034,0.001034)(0.150000,0.007211)+-(0.001170,0.001170)(0.200000,0.007047)+-(0.001250,0.001250)(0.250000,0.007217)+-(0.001323,0.001323)(0.300000,0.007694)+-(0.001418,0.001418)(0.350000,0.008505)+-(0.001558,0.001558)(0.400000,0.009696)+-(0.001762,0.001762)(0.450000,0.011363)+-(0.002074,0.002074)(0.500000,0.013608)+-(0.002542,0.002542)(0.550000,0.016628)+-(0.003196,0.003196)(0.600000,0.020636)+-(0.004120,0.004120)(0.650000,0.025945)+-(0.005360,0.005360)(0.700000,0.032905)+-(0.006995,0.006995)(0.750000,0.042004)+-(0.009156,0.009156)(0.800000,0.053704)+-(0.011903,0.011903)(0.850000,0.068460)+-(0.015448,0.015448)(0.900000,0.086600)+-(0.019812,0.019812)(0.950000,0.108034)+-(0.025015,0.025015)
      };
      \addplot[black,only marks,mark=*,mark options={scale=0.75}] plot coordinates{
%	      (0.050000,0.001500)(0.100000,0.001410)(0.150000,0.001290)(0.200000,0.001140)(0.250000,0.001440)(0.300000,0.001270)(0.350000,0.001570)(0.400000,0.002130)(0.450000,0.002620)(0.500000,0.003230)(0.550000,0.004290)(0.600000,0.006360)(0.650000,0.008680)(0.700000,0.012760)(0.750000,0.018760)(0.800000,0.027100)(0.850000,0.040770)(0.900000,0.057480)(0.950000,0.079240)(1.000000,0.107680)
	      (0.050000,0.001664)(0.100000,0.001371)(0.150000,0.001260)(0.200000,0.001244)(0.250000,0.001266)(0.300000,0.001392)(0.350000,0.001612)(0.400000,0.001974)(0.450000,0.002550)(0.500000,0.003313)(0.550000,0.004514)(0.600000,0.006323)(0.650000,0.009090)(0.700000,0.013042)(0.750000,0.019099)(0.800000,0.027746)(0.850000,0.040506)(0.900000,0.057496)(0.950000,0.080889)(1.000000,0.108458)	      
      };
      \node at (axis cs:0.4,0.3) {$\gamma=2$};
      \draw (axis cs:0.4,0.12) ellipse [black,x radius=2,y radius=0.5];
      \node at (axis cs:0.4,0.03) {$\gamma=3$};
      \draw (axis cs:0.4,0.005) ellipse [black,x radius=2,y radius=1.5];
      \legend{ {Limiting theory},{Empirical estimator},{Detector} }
    \end{semilogyaxis}
  \end{tikzpicture}
  &
  \begin{tikzpicture}[font=\footnotesize,scale=.7]
    \renewcommand{\axisdefaulttryminticks}{4} 
    \tikzstyle{every major grid}+=[style=densely dashed]       
    \tikzstyle{every axis y label}+=[yshift=-10pt] 
    \tikzstyle{every axis x label}+=[yshift=5pt]
    \tikzstyle{every axis legend}+=[cells={anchor=west},fill=white,
        at={(0.98,0.02)}, anchor=south east, font=\scriptsize ]
    \begin{semilogyaxis}[
      %ybar,
      xmin=0,
      xmax=1,
      ymax=1,
      grid=major,
      ymajorgrids=false,
      scaled ticks=true,
      %scale ticks above={4},
      xlabel={$\rho$},
      ylabel={$P(\sqrt{N}T_N(\rho)>\gamma)$}
      ]
      \addplot[black] plot coordinates{
	      (0.010000,0.132556)(0.020000,0.130002)(0.030000,0.127659)(0.040000,0.125515)(0.050000,0.123558)(0.060000,0.121777)(0.070000,0.120163)(0.080000,0.118708)(0.090000,0.117404)(0.100000,0.116244)(0.110000,0.115221)(0.120000,0.114330)(0.130000,0.113565)(0.140000,0.112922)(0.150000,0.112395)(0.160000,0.111983)(0.170000,0.111680)(0.180000,0.111483)(0.190000,0.111391)(0.200000,0.111401)(0.210000,0.111510)(0.220000,0.111717)(0.230000,0.112020)(0.240000,0.112418)(0.250000,0.112910)(0.260000,0.113495)(0.270000,0.114171)(0.280000,0.114940)(0.290000,0.115800)(0.300000,0.116752)(0.310000,0.117795)(0.320000,0.118930)(0.330000,0.120157)(0.340000,0.121477)(0.350000,0.122890)(0.360000,0.124398)(0.370000,0.126000)(0.380000,0.127699)(0.390000,0.129495)(0.400000,0.131391)(0.410000,0.133386)(0.420000,0.135483)(0.430000,0.137683)(0.440000,0.139989)(0.450000,0.142401)(0.460000,0.144923)(0.470000,0.147555)(0.480000,0.150300)(0.490000,0.153159)(0.500000,0.156136)(0.510000,0.159233)(0.520000,0.162451)(0.530000,0.165793)(0.540000,0.169262)(0.550000,0.172859)(0.560000,0.176587)(0.570000,0.180449)(0.580000,0.184447)(0.590000,0.188584)(0.600000,0.192861)(0.610000,0.197281)(0.620000,0.201846)(0.630000,0.206559)(0.640000,0.211421)(0.650000,0.216435)(0.660000,0.221602)(0.670000,0.226924)(0.680000,0.232402)(0.690000,0.238039)(0.700000,0.243834)(0.710000,0.249789)(0.720000,0.255905)(0.730000,0.262182)(0.740000,0.268620)(0.750000,0.275218)(0.760000,0.281977)(0.770000,0.288895)(0.780000,0.295972)(0.790000,0.303204)(0.800000,0.310590)(0.810000,0.318129)(0.820000,0.325815)(0.830000,0.333647)(0.840000,0.341621)(0.850000,0.349732)(0.860000,0.357975)(0.870000,0.366345)(0.880000,0.374838)(0.890000,0.383447)(0.900000,0.392165)(0.910000,0.400988)(0.920000,0.409907)(0.930000,0.418917)(0.940000,0.428010)(0.950000,0.437180)(0.960000,0.446419)(0.970000,0.455720)(0.980000,0.465078)(0.990000,0.474484)(1.000000,0.483934)
      };
      \addplot[black,only marks,mark=o,mark options={scale=0.75},error bars/.cd,y dir=both,y explicit, error bar style={mark size=1.5pt}] plot coordinates{
	      (0.050000,0.123520)+-(0.001961,0.001961)(0.100000,0.116258)+-(0.003004,0.003004)(0.150000,0.112500)+-(0.003619,0.003619)(0.200000,0.111525)+-(0.003995,0.003995)(0.250000,0.113046)+-(0.004201,0.004201)(0.300000,0.116926)+-(0.004334,0.004334)(0.350000,0.122988)+-(0.004459,0.004459)(0.400000,0.131404)+-(0.004595,0.004595)(0.450000,0.142259)+-(0.004695,0.004695)(0.500000,0.155536)+-(0.005049,0.005049)(0.550000,0.171997)+-(0.005367,0.005367)(0.600000,0.191571)+-(0.005818,0.005818)(0.650000,0.214513)+-(0.006541,0.006541)(0.700000,0.241244)+-(0.007450,0.007450)(0.750000,0.271917)+-(0.008836,0.008836)(0.800000,0.306142)+-(0.010523,0.010523)(0.850000,0.344681)+-(0.012411,0.012411)(0.900000,0.386030)+-(0.015349,0.015349)(0.950000,0.429722)+-(0.018687,0.018687)
      };
      \addplot[black,only marks,mark=*,mark options={scale=0.75}] plot coordinates{
	   %   (0.050000,0.124600)(0.100000,0.113600)(0.150000,0.105300)(0.200000,0.109000)(0.250000,0.103500)(0.300000,0.111900)(0.350000,0.121100)(0.400000,0.136300)(0.450000,0.136500)(0.500000,0.152000)(0.550000,0.178200)(0.600000,0.189500)(0.650000,0.211400)(0.700000,0.246600)(0.750000,0.282400)(0.800000,0.306200)(0.850000,0.363100)(0.900000,0.394300)(0.950000,0.447900)(1.000000,0.495000)
	      (0.050000,0.123990)(0.100000,0.115930)(0.150000,0.111060)(0.200000,0.109690)(0.250000,0.111230)(0.300000,0.114650)(0.350000,0.120460)(0.400000,0.129600)(0.450000,0.141420)(0.500000,0.154690)(0.550000,0.174360)(0.600000,0.194010)(0.650000,0.218770)(0.700000,0.246950)(0.750000,0.279720)(0.800000,0.318310)(0.850000,0.355070)(0.900000,0.398220)(0.950000,0.447030)(1.000000,0.492040)
      };
      \addplot[black] plot coordinates{
	      (0.010000,0.010602)(0.020000,0.010148)(0.030000,0.009741)(0.040000,0.009377)(0.050000,0.009051)(0.060000,0.008760)(0.070000,0.008501)(0.080000,0.008271)(0.090000,0.008068)(0.100000,0.007890)(0.110000,0.007735)(0.120000,0.007601)(0.130000,0.007487)(0.140000,0.007392)(0.150000,0.007314)(0.160000,0.007254)(0.170000,0.007210)(0.180000,0.007182)(0.190000,0.007168)(0.200000,0.007170)(0.210000,0.007186)(0.220000,0.007216)(0.230000,0.007260)(0.240000,0.007318)(0.250000,0.007390)(0.260000,0.007476)(0.270000,0.007577)(0.280000,0.007692)(0.290000,0.007823)(0.300000,0.007968)(0.310000,0.008129)(0.320000,0.008306)(0.330000,0.008500)(0.340000,0.008712)(0.350000,0.008942)(0.360000,0.009190)(0.370000,0.009459)(0.380000,0.009748)(0.390000,0.010059)(0.400000,0.010394)(0.410000,0.010752)(0.420000,0.011136)(0.430000,0.011547)(0.440000,0.011987)(0.450000,0.012457)(0.460000,0.012959)(0.470000,0.013494)(0.480000,0.014065)(0.490000,0.014675)(0.500000,0.015324)(0.510000,0.016017)(0.520000,0.016754)(0.530000,0.017540)(0.540000,0.018376)(0.550000,0.019267)(0.560000,0.020214)(0.570000,0.021223)(0.580000,0.022295)(0.590000,0.023436)(0.600000,0.024649)(0.610000,0.025938)(0.620000,0.027308)(0.630000,0.028764)(0.640000,0.030310)(0.650000,0.031951)(0.660000,0.033693)(0.670000,0.035541)(0.680000,0.037501)(0.690000,0.039578)(0.700000,0.041779)(0.710000,0.044110)(0.720000,0.046578)(0.730000,0.049188)(0.740000,0.051947)(0.750000,0.054862)(0.760000,0.057940)(0.770000,0.061188)(0.780000,0.064612)(0.790000,0.068219)(0.800000,0.072015)(0.810000,0.076007)(0.820000,0.080202)(0.830000,0.084605)(0.840000,0.089223)(0.850000,0.094060)(0.860000,0.099121)(0.870000,0.104413)(0.880000,0.109938)(0.890000,0.115701)(0.900000,0.121704)(0.910000,0.127951)(0.920000,0.134445)(0.930000,0.141185)(0.940000,0.148174)(0.950000,0.155412)(0.960000,0.162900)(0.970000,0.170636)(0.980000,0.178621)(0.990000,0.186853)(1.000000,0.195330)
      };
      \addplot[black,only marks,mark=o,mark options={scale=0.75},error bars/.cd,y dir=both,y explicit, error bar style={mark size=1.5pt}] plot coordinates{
	      (0.050000,0.009048)+-(0.000322,0.000322)(0.100000,0.007900)+-(0.000456,0.000456)(0.150000,0.007341)+-(0.000528,0.000528)(0.200000,0.007201)+-(0.000576,0.000576)(0.250000,0.007425)+-(0.000617,0.000617)(0.300000,0.008010)+-(0.000663,0.000663)(0.350000,0.008974)+-(0.000728,0.000728)(0.400000,0.010414)+-(0.000816,0.000816)(0.450000,0.012448)+-(0.000921,0.000921)(0.500000,0.015215)+-(0.001107,0.001107)(0.550000,0.019077)+-(0.001335,0.001335)(0.600000,0.024311)+-(0.001655,0.001655)(0.650000,0.031357)+-(0.002145,0.002145)(0.700000,0.040842)+-(0.002831,0.002831)(0.750000,0.053472)+-(0.003899,0.003899)(0.800000,0.069831)+-(0.005378,0.005378)(0.850000,0.091197)+-(0.007359,0.007359)(0.900000,0.117723)+-(0.010475,0.010475)(0.950000,0.149908)+-(0.014589,0.014589)
      };
      \addplot[black,only marks,mark=*,mark options={scale=0.75}] plot coordinates{
%	      (0.050000,0.006500)(0.100000,0.004300)(0.150000,0.005700)(0.200000,0.005600)(0.250000,0.006100)(0.300000,0.006200)(0.350000,0.007400)(0.400000,0.008700)(0.450000,0.011600)(0.500000,0.011600)(0.550000,0.017900)(0.600000,0.020500)(0.650000,0.025500)(0.700000,0.040300)(0.750000,0.054900)(0.800000,0.066200)(0.850000,0.085700)(0.900000,0.118300)(0.950000,0.152200)(1.000000,0.187900)
	      (0.050000,0.008520)(0.100000,0.007550)(0.150000,0.006520)(0.200000,0.006560)(0.250000,0.006420)(0.300000,0.006850)(0.350000,0.008040)(0.400000,0.009110)(0.450000,0.011040)(0.500000,0.014580)(0.550000,0.017780)(0.600000,0.023820)(0.650000,0.030230)(0.700000,0.041990)(0.750000,0.054280)(0.800000,0.071530)(0.850000,0.092110)(0.900000,0.121480)(0.950000,0.157840)(1.000000,0.197780)
      };
      \node at (axis cs:0.4,0.3) {$\gamma=2$};
      \draw (axis cs:0.4,0.12) ellipse [black,x radius=2,y radius=0.5];
      \node at (axis cs:0.4,0.025) {$\gamma=3$};
      \draw (axis cs:0.4,0.01) ellipse [black,x radius=2,y radius=0.5];
      \legend{ {Limiting theory},{Empirical estimator},{Detector} }
    \end{semilogyaxis}
  \end{tikzpicture}
  \end{tabular}
  \caption{False alarm rate $P(\sqrt{N}T_N(\rho)>\gamma)$, $N=20$ (left), $N=100$ (right), $p=N^{-\frac12}[1,\ldots,1]^\trans$, $[C_N]_{ij}=0.7^{|i-j|}$, $c_N=1/2$.}
  \label{fig:FAR}
\end{figure}
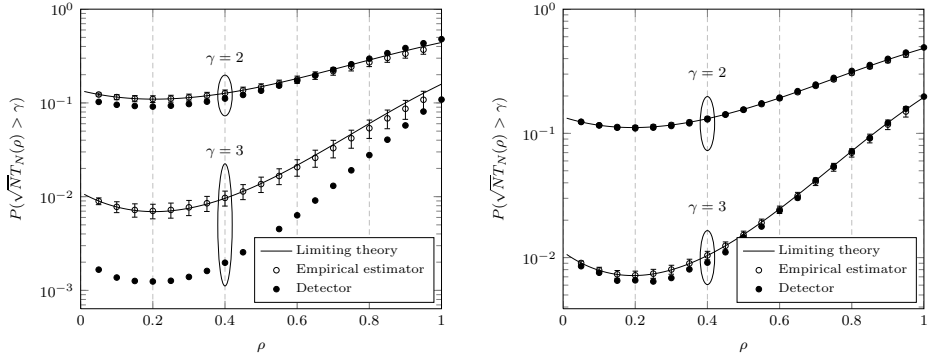

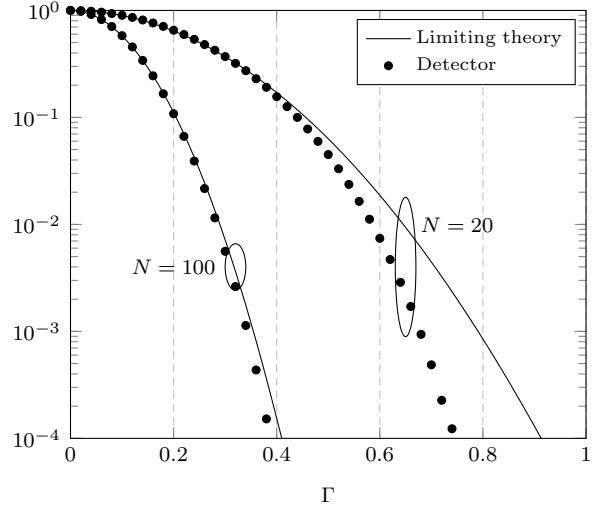
\begin{figure}[h!]
  \centering
  \begin{tikzpicture}[font=\footnotesize]
    \renewcommand{\axisdefaulttryminticks}{4} 
    \tikzstyle{every major grid}+=[style=densely dashed]       
    \tikzstyle{every axis y label}+=[yshift=-10pt] 
    \tikzstyle{every axis x label}+=[yshift=5pt]
    \tikzstyle{every axis legend}+=[cells={anchor=west},fill=white,
        at={(0.98,0.98)}, anchor=north east, font=\scriptsize ]
    \begin{semilogyaxis}[
      %ybar,
      xmin=0,
      ymin=1e-4,
      xmax=1,
      ymax=1,
      grid=major,
      ymajorgrids=false,
      scaled ticks=true,
      %scale ticks above={4},
      xlabel={$\Gamma$},
      mark repeat=2;
      ylabel={$P(T_N(\rho)>\Gamma)$}
      ]
      \addplot[black] plot coordinates{
	      (0.,1.)(0.01,0.998896)(0.02,0.995589)(0.03,0.990103)(0.04,0.982474)(0.05,0.97275)(0.06,0.960997)(0.07,0.94729)(0.08,0.931716)(0.09,0.914376)(0.1,0.895377)(0.11,0.874837)(0.12,0.852881)(0.13,0.82964)(0.14,0.805251)(0.15,0.779853)(0.16,0.753589)(0.17,0.726602)(0.18,0.699035)(0.19,0.671028)(0.2,0.642722)(0.21,0.614251)(0.22,0.585744)(0.23,0.557328)(0.24,0.529119)(0.25,0.501229)(0.26,0.473761)(0.27,0.446809)(0.28,0.420461)(0.29,0.394792)(0.3,0.369873)(0.31,0.345761)(0.32,0.322507)(0.33,0.300153)(0.34,0.278732)(0.35,0.258268)(0.36,0.238778)(0.37,0.220272)(0.38,0.202751)(0.39,0.186212)(0.4,0.170645)(0.41,0.156033)(0.42,0.142358)(0.43,0.129595)(0.44,0.117715)(0.45,0.106688)(0.46,0.096481)(0.47,0.087058)(0.48,0.078381)(0.49,0.070414)(0.5,0.063117)(0.51,0.056451)(0.52,0.050377)(0.53,0.044858)(0.54,0.039856)(0.55,0.035333)(0.56,0.031254)(0.57,0.027585)(0.58,0.024293)(0.59,0.021346)(0.6,0.018716)(0.61,0.016373)(0.62,0.014292)(0.63,0.012448)(0.64,0.010818)(0.65,0.009381)(0.66,0.008117)(0.67,0.007007)(0.68,0.006036)(0.69,0.005188)(0.7,0.004449)(0.71,0.003807)(0.72,0.003251)(0.73,0.002769)(0.74,0.002354)(0.75,0.001997)(0.76,0.00169)(0.77,0.001427)(0.78,0.001202)(0.79,0.001011)(0.8,0.000848)(0.81,0.00071)(0.82,0.000593)(0.83,0.000494)(0.84,0.000411)(0.85,0.000341)(0.86,0.000282)(0.87,0.000233)(0.88,0.000192)(0.89,0.000158)(0.9,0.00013)(0.91,0.000106)(0.92,0.000087)(0.93,0.000071)(0.94,0.000057)(0.95,0.000047)(0.96,0.000038)(0.97,0.00003)(0.98,0.000025)(0.99,0.00002)(1.,0.000016)
      };
      \addplot[black,only marks,mark=*,mark options={scale=0.75}] plot coordinates{
	      (0.,1.)(0.01,0.99896)(0.02,0.995809)(0.03,0.990564)(0.04,0.983483)(0.05,0.974368)(0.06,0.963289)(0.07,0.950217)(0.08,0.935453)(0.09,0.919055)(0.1,0.900819)(0.11,0.881349)(0.12,0.860069)(0.13,0.837701)(0.14,0.813929)(0.15,0.789225)(0.16,0.763256)(0.17,0.736593)(0.18,0.709161)(0.19,0.681532)(0.2,0.652937)(0.21,0.624299)(0.22,0.595142)(0.23,0.565768)(0.24,0.537037)(0.25,0.508517)(0.26,0.479995)(0.27,0.45221)(0.28,0.424418)(0.29,0.397055)(0.3,0.370642)(0.31,0.344963)(0.32,0.320174)(0.33,0.296156)(0.34,0.273264)(0.35,0.251097)(0.36,0.230089)(0.37,0.210095)(0.38,0.191335)(0.39,0.173287)(0.4,0.156354)(0.41,0.140724)(0.42,0.126194)(0.43,0.112588)(0.44,0.099988)(0.45,0.088543)(0.46,0.077981)(0.47,0.068355)(0.48,0.059736)(0.49,0.051963)(0.5,0.044954)(0.51,0.038702)(0.52,0.033099)(0.53,0.028047)(0.54,0.023645)(0.55,0.019754)(0.56,0.016417)(0.57,0.013608)(0.58,0.011183)(0.59,0.009115)(0.6,0.00741)(0.61,0.005953)(0.62,0.004697)(0.63,0.003682)(0.64,0.002875)(0.65,0.002209)(0.66,0.001709)(0.67,0.001276)(0.68,0.000938)(0.69,0.000692)(0.7,0.000487)(0.71,0.000338)(0.72,0.000227)(0.73,0.000163)(0.74,0.000123)(0.75,0.000076)(0.76,0.00005)(0.77,0.00003)(0.78,0.000014)(0.79,0.000006)(0.8,0.000002)(0.81,0.000001)(0.82,0.)(0.83,0.)(0.84,0.)(0.85,0.)(0.86,0.)(0.87,0.)(0.88,0.)(0.89,0.)(0.9,0.)(0.91,0.)(0.92,0.)(0.93,0.)(0.94,0.)(0.95,0.)(0.96,0.)(0.97,0.)(0.98,0.)(0.99,0.)(1.,0.)
      };
%      \addplot[black,dashed] plot coordinates{
%(0.,1.)(0.01,0.999)(0.02,0.996008)(0.03,0.991039)(0.04,0.984126)(0.05,0.975307)(0.06,0.964636)(0.07,0.952176)(0.08,0.937998)(0.09,0.922185)(0.1,0.904827)(0.11,0.886022)(0.12,0.865874)(0.13,0.844493)(0.14,0.821994)(0.15,0.798496)(0.16,0.77412)(0.17,0.748988)(0.18,0.723224)(0.19,0.696951)(0.2,0.67029)(0.21,0.643361)(0.22,0.616279)(0.23,0.589159)(0.24,0.562106)(0.25,0.535224)(0.26,0.508609)(0.27,0.482351)(0.28,0.456536)(0.29,0.431238)(0.3,0.406528)(0.31,0.382469)(0.32,0.359114)(0.33,0.336511)(0.34,0.314701)(0.35,0.293717)(0.36,0.273584)(0.37,0.254322)(0.38,0.235943)(0.39,0.218456)(0.4,0.20186)(0.41,0.186152)(0.42,0.171324)(0.43,0.157362)(0.44,0.144248)(0.45,0.131964)(0.46,0.120484)(0.47,0.109783)(0.48,0.099833)(0.49,0.090603)(0.5,0.082062)(0.51,0.074178)(0.52,0.066917)(0.53,0.060246)(0.54,0.054132)(0.55,0.048541)(0.56,0.043441)(0.57,0.038799)(0.58,0.034583)(0.59,0.030764)(0.6,0.027313)(0.61,0.0242)(0.62,0.021398)(0.63,0.018884)(0.64,0.016631)(0.65,0.014618)(0.66,0.012823)(0.67,0.011226)(0.68,0.009808)(0.69,0.008552)(0.7,0.007442)(0.71,0.006464)(0.72,0.005602)(0.73,0.004846)(0.74,0.004183)(0.75,0.003604)(0.76,0.003099)(0.77,0.002659)(0.78,0.002277)(0.79,0.001947)(0.8,0.00166)(0.81,0.001413)(0.82,0.001201)(0.83,0.001018)(0.84,0.000862)(0.85,0.000728)(0.86,0.000613)(0.87,0.000516)(0.88,0.000433)(0.89,0.000363)(0.9,0.000303)(0.91,0.000253)(0.92,0.000211)(0.93,0.000175)(0.94,0.000145)(0.95,0.00012)(0.96,0.000099)(0.97,0.000082)(0.98,0.000067)(0.99,0.000055)(1.,0.000045)
%      };
      \addplot[black] plot coordinates{
	      (0.,1.)(0.01,0.994528)(0.02,0.978292)(0.03,0.951819)(0.04,0.915955)(0.05,0.871823)(0.06,0.820761)(0.07,0.764257)(0.08,0.703876)(0.09,0.641191)(0.1,0.577714)(0.11,0.51484)(0.12,0.453802)(0.13,0.395635)(0.14,0.341159)(0.15,0.290974)(0.16,0.245462)(0.17,0.20481)(0.18,0.169025)(0.19,0.13797)(0.2,0.111391)(0.21,0.088952)(0.22,0.070257)(0.23,0.054886)(0.24,0.04241)(0.25,0.032412)(0.26,0.024501)(0.27,0.018318)(0.28,0.013547)(0.29,0.009908)(0.3,0.007168)(0.31,0.005129)(0.32,0.00363)(0.33,0.002541)(0.34,0.00176)(0.35,0.001205)(0.36,0.000816)(0.37,0.000547)(0.38,0.000362)(0.39,0.000237)(0.4,0.000154)(0.41,0.000099)(0.42,0.000063)(0.43,0.000039)(0.44,0.000024)(0.45,0.000015)(0.46,0.000009)(0.47,0.000005)(0.48,0.000003)(0.49,0.000002)(0.5,0.000001)(0.51,0.000001)(0.52,0.)(0.53,0.)(0.54,0.)(0.55,0.)(0.56,0.)(0.57,0.)(0.58,0.)(0.59,0.)(0.6,0.)(0.61,0.)(0.62,0.)(0.63,0.)(0.64,0.)(0.65,0.)(0.66,0.)(0.67,0.)(0.68,0.)(0.69,0.)(0.7,0.)(0.71,0.)(0.72,0.)(0.73,0.)(0.74,0.)(0.75,0.)(0.76,0.)(0.77,0.)(0.78,0.)(0.79,0.)(0.8,0.)(0.81,0.)(0.82,0.)(0.83,0.)(0.84,0.)(0.85,0.)(0.86,0.)(0.87,0.)(0.88,0.)(0.89,0.)(0.9,0.)(0.91,0.)(0.92,0.)(0.93,0.)(0.94,0.)(0.95,0.)(0.96,0.)(0.97,0.)(0.98,0.)(0.99,0.)(1.,0.)
      };
      \addplot[black,only marks,mark=*,mark options={scale=0.75}] plot coordinates{
%	      (0.,1.)(0.01,0.9942)(0.02,0.9754)(0.03,0.949)(0.04,0.9125)(0.05,0.8698)(0.06,0.8159)(0.07,0.7604)(0.08,0.7029)(0.09,0.6417)(0.1,0.5767)(0.11,0.5156)(0.12,0.4521)(0.13,0.393)(0.14,0.3354)(0.15,0.2849)(0.16,0.2389)(0.17,0.195)(0.18,0.16)(0.19,0.1297)(0.2,0.1025)(0.21,0.0814)(0.22,0.0646)(0.23,0.0494)(0.24,0.0384)(0.25,0.029)(0.26,0.0224)(0.27,0.017)(0.28,0.0123)(0.29,0.0087)(0.3,0.0066)(0.31,0.0041)(0.32,0.0031)(0.33,0.0019)(0.34,0.001)(0.35,0.0009)(0.36,0.0006)(0.37,0.0005)(0.38,0.0004)(0.39,0.0003)(0.4,0.0001)(0.41,0.)(0.42,0.)(0.43,0.)(0.44,0.)(0.45,0.)(0.46,0.)(0.47,0.)(0.48,0.)(0.49,0.)(0.5,0.)(0.51,0.)(0.52,0.)(0.53,0.)(0.54,0.)(0.55,0.)(0.56,0.)(0.57,0.)(0.58,0.)(0.59,0.)(0.6,0.)(0.61,0.)(0.62,0.)(0.63,0.)(0.64,0.)(0.65,0.)(0.66,0.)(0.67,0.)(0.68,0.)(0.69,0.)(0.7,0.)(0.71,0.)(0.72,0.)(0.73,0.)(0.74,0.)(0.75,0.)(0.76,0.)(0.77,0.)(0.78,0.)(0.79,0.)(0.8,0.)(0.81,0.)(0.82,0.)(0.83,0.)(0.84,0.)(0.85,0.)(0.86,0.)(0.87,0.)(0.88,0.)(0.89,0.)(0.9,0.)(0.91,0.)(0.92,0.)(0.93,0.)(0.94,0.)(0.95,0.)(0.96,0.)(0.97,0.)(0.98,0.)(0.99,0.)(1.,0.)
	      (0.000000,1.000000)(0.010000,0.994368)(0.020000,0.978268)(0.030000,0.952258)(0.040000,0.917284)(0.050000,0.872770)(0.060000,0.821488)(0.070000,0.766126)(0.080000,0.706967)(0.090000,0.643454)(0.100000,0.580648)(0.110000,0.517506)(0.120000,0.455490)(0.130000,0.396165)(0.140000,0.341170)(0.150000,0.290833)(0.160000,0.244757)(0.170000,0.202599)(0.180000,0.166316)(0.190000,0.135318)(0.200000,0.108255)(0.210000,0.085529)(0.220000,0.066516)(0.230000,0.051258)(0.240000,0.039086)(0.250000,0.029266)(0.260000,0.021645)(0.270000,0.015920)(0.280000,0.011507)(0.290000,0.008071)(0.300000,0.005601)(0.310000,0.003907)(0.320000,0.002626)(0.330000,0.001749)(0.340000,0.001136)(0.350000,0.000720)(0.360000,0.000436)(0.370000,0.000256)(0.380000,0.000152)(0.390000,0.000100)(0.400000,0.000045)(0.410000,0.000024)(0.420000,0.000014)(0.430000,0.000010)(0.440000,0.000003)(0.450000,0.000003)(0.460000,0.000000)(0.470000,0.000000)(0.480000,0.000000)(0.490000,0.000000)(0.500000,0.000000)(0.510000,0.000000)(0.520000,0.000000)(0.530000,0.000000)(0.540000,0.000000)(0.550000,0.000000)(0.560000,0.000000)(0.570000,0.000000)(0.580000,0.000000)(0.590000,0.000000)(0.600000,0.000000)(0.610000,0.000000)(0.620000,0.000000)(0.630000,0.000000)(0.640000,0.000000)(0.650000,0.000000)(0.660000,0.000000)(0.670000,0.000000)(0.680000,0.000000)(0.690000,0.000000)(0.700000,0.000000)(0.710000,0.000000)(0.720000,0.000000)(0.730000,0.000000)(0.740000,0.000000)(0.750000,0.000000)(0.760000,0.000000)(0.770000,0.000000)(0.780000,0.000000)(0.790000,0.000000)(0.800000,0.000000)(0.810000,0.000000)(0.820000,0.000000)(0.830000,0.000000)(0.840000,0.000000)(0.850000,0.000000)(0.860000,0.000000)(0.870000,0.000000)(0.880000,0.000000)(0.890000,0.000000)(0.900000,0.000000)(0.910000,0.000000)(0.920000,0.000000)(0.930000,0.000000)(0.940000,0.000000)(0.950000,0.000000)(0.960000,0.000000)(0.970000,0.000000)(0.980000,0.000000)(0.990000,0.000000)(1.000000,0.000000)
      };
%      \addplot[black,dashed] plot coordinates{
%	      (0.,1.)(0.01,0.995012)(0.02,0.980197)(0.03,0.955993)(0.04,0.923108)(0.05,0.882485)(0.06,0.835254)(0.07,0.782683)(0.08,0.726123)(0.09,0.666947)(0.1,0.606497)(0.11,0.546038)(0.12,0.486714)(0.13,0.429517)(0.14,0.375271)(0.15,0.324612)(0.16,0.277998)(0.17,0.235709)(0.18,0.197863)(0.19,0.164442)(0.2,0.135306)(0.21,0.110224)(0.22,0.088898)(0.23,0.070985)(0.24,0.056117)(0.25,0.043922)(0.26,0.034035)(0.27,0.026111)(0.28,0.019833)(0.29,0.014914)(0.3,0.011104)(0.31,0.008184)(0.32,0.005973)(0.33,0.004315)(0.34,0.003087)(0.35,0.002186)(0.36,0.001533)(0.37,0.001064)(0.38,0.000731)(0.39,0.000498)(0.4,0.000335)(0.41,0.000224)(0.42,0.000148)(0.43,0.000096)(0.44,0.000062)(0.45,0.00004)(0.46,0.000025)(0.47,0.000016)(0.48,0.00001)(0.49,0.000006)(0.5,0.000004)(0.51,0.000002)(0.52,0.000001)(0.53,0.000001)(0.54,0.)(0.55,0.)(0.56,0.)(0.57,0.)(0.58,0.)(0.59,0.)(0.6,0.)(0.61,0.)(0.62,0.)(0.63,0.)(0.64,0.)(0.65,0.)(0.66,0.)(0.67,0.)(0.68,0.)(0.69,0.)(0.7,0.)(0.71,0.)(0.72,0.)(0.73,0.)(0.74,0.)(0.75,0.)(0.76,0.)(0.77,0.)(0.78,0.)(0.79,0.)(0.8,0.)(0.81,0.)(0.82,0.)(0.83,0.)(0.84,0.)(0.85,0.)(0.86,0.)(0.87,0.)(0.88,0.)(0.89,0.)(0.9,0.)(0.91,0.)(0.92,0.)(0.93,0.)(0.94,0.)(0.95,0.)(0.96,0.)(0.97,0.)(0.98,0.)(0.99,0.)(1.,0.)
%      };
      \node at (axis cs:0.2,0.004) {$N=100$};
      \draw (axis cs:0.32,0.004) ellipse [black,x radius=2,y radius=0.5];
      \node at (axis cs:0.75,0.01) {$N=20$};
      \draw (axis cs:0.65,0.004) ellipse [black,x radius=2,y radius=1.5];
      \legend{ {Limiting theory},{Detector}} %,{No shrinkage} }
    \end{semilogyaxis}
  \end{tikzpicture}
  \caption{False alarm rate $P(T_N(\rho_N^*)>\Gamma)$ for $N=20$ and $N=100$, $p=N^{-\frac12}[1,\ldots,1]^\trans$, $[C_N]_{ij}=0.7^{|i-j|}$, $c_N=1/2$.}
  \label{fig:FAR2}
\end{figure}

\section{Proof}
\label{sec:proof}
In this section, we shall successively prove Theorem~\ref{th:bilin}, Theorem~\ref{th:T}, Proposition~\ref{prop:1}, and Corollary~\ref{co:1}. Of utmost interest is the proof of Theorem~\ref{th:bilin} which shall be the concern of most of the section and of Appendix~\ref{app:key_lemma} for the proof of a key lemma.

Before delving into the core of the proofs, let us introduce a few notations that shall be used throughout the section. 
First recall from \citep{COU14} that we can write, for each $\rho\in(\max\{0,1-c_N^{-1}\},1]$,
\begin{align*}
	\hat{C}_N(\rho)=\frac{1-\rho}{1-(1-\rho)c_N} \frac1n\sum_{i=1}^n \frac{z_iz_i^*}{\frac1Nz_i^*\hat{C}_{(i)}^{-1}(\rho)z_i} + \rho I_N
\end{align*}
where $\hat{C}_{(i)}(\rho)=\hat{C}_N(\rho)-(1-\rho)\frac1n\frac{z_iz_i^*}{\frac1Nz_i^*\hat{C}_N^{-1}(\rho)z_i}$.

Now, we define
\begin{align*}
	\alpha(\rho) &= \frac{1-\rho}{1-(1-\rho)c_N} \\
	d_i(\rho) &=  \frac1Nz_i^* \hat{C}_{(i)}^{-1}(\rho)z_i = \frac1Nz_i^* \left( \alpha(\rho) \frac1n\sum_{j\neq i} \frac{z_jz_j^*}{d_j(\rho)} + \rho I_N \right)^{-1}z_i \\
	\tilde d_i(\rho) &= \frac1Nz_i^* \hat{S}_{(i)}^{-1}(\rho)z_i =  \frac1Nz_i^* \left( \alpha(\rho) \frac1n\sum_{j\neq i}^n \frac{z_jz_j^*}{\gamma_N(\rho)} + \rho I_N \right)^{-1}z_i% \\
%		Z &= [z_1,\ldots,z_n].
\end{align*}
Clearly by uniqueness of $\hat{C}_N$ and by the relation to $\hat{C}_{(i)}$ above, $d_1(\rho),\ldots,d_n(\rho)$ are uniquely defined by their $n$ implicit equations.
We shall also discard the parameter $\rho$ for readability whenever not needed.

\subsection{Bilinear form equivalence}

In this section, we prove Theorem~\ref{th:bilin}. As shall become clear, the proof unfolds similarly for each $k\in\ZZ\setminus\{0\}$ and we can therefore restrict ourselves to a single value for $k$. As Theorem~\ref{th:T} relies on $k=-1$, for consistency, we take $k=-1$ from now on. Thus, our objective is to prove that, for $a,b\in\CC^N$ with $\Vert a\Vert=\Vert b\Vert=1$, and for any $\varepsilon>0$,
\begin{align*}
	\sup_{\rho\in \mathcal R_\kappa} N^{1-\varepsilon}\left| a^*\hat{C}_N^{-1}(\rho)b - a^*\hat{S}_N^{-1}(\rho)b \right| \asto 0.
\end{align*}

For this, forgetting for some time the index $\rho$, first write 
\begin{align}
	\label{eq:firsteq}
	a^*\hat{C}_N^{-1}b - a^*\hat{S}_N^{-1}b &= a^*\hat{C}_N^{-1} \left( \frac{\alpha}n \sum_{i=1}^n \left[ \frac1{\gamma_N} - \frac1{d_i} \right] z_iz_i^* \right) \hat{S}_N^{-1}b \\
	\label{eq:secondeq}
&= \frac{\alpha}n \sum_{i=1}^n a^*\hat{C}_N^{-1} z_i\frac{d_i-\gamma_N}{\gamma_Nd_i}z_i^*\hat{S}_N^{-1}b.
\end{align}
In \citep{COU14}, where it is shown that $\Vert \hat{C}_N-\hat{S}_N\Vert\asto 0$ (that is the spectral norm of the inner parenthesis in \eqref{eq:firsteq} vanishes), the core of the proof was to show that $\max_{1\leq i\leq n}|d_i-\gamma_N|\asto 0$ which, along with the convergence of $\gamma_N$ away from zero and the almost sure boundedness of $\Vert\frac1n\sum_{i=1}^nz_iz_i^*\Vert$ for all large $N$ (from e.g.\@ \citep{SIL98}), gives the result. A thorough inspection of the proof in \citep{COU14} reveals that $\max_{1\leq i\leq n}|d_i-\gamma_N|\asto 0$ may be improved into $\max_{1\leq i\leq n}N^{\frac12-\varepsilon}|d_i-\gamma_N|\asto 0$ for any $\varepsilon>0$ but that this speed cannot be further improved beyond $N^{\frac12}$. The latter statement is rather intuitive since $\gamma_N$ is essentially a sharp deterministic approximation for $\frac1N\tr \hat{C}_N^{-1}$ while $d_i$ is a quadratic form on $\hat{C}_{(i)}^{-1}$; classical random matrix results involving fluctuations of such quadratic forms, see e.g.\@ \citep{KAM09}, indeed show that these fluctuations are of order $N^{-\frac12}$. As a consequence, $\max_{1\leq i\leq n}N^{1-\varepsilon}|d_i-\gamma_N|$ and thus $N^{1-\varepsilon}\Vert \hat{C}_N-\hat{S}_N\Vert$ are not expected to vanish for small $\varepsilon$.

This being said, when it comes to bilinear forms, for which we shall naturally have $N^{\frac12-\varepsilon}|a^*\hat{C}_N^{-1}b - a^*\hat{S}_N^{-1}b|\asto 0$, seeing the difference in absolute values as the $n$-term average \eqref{eq:secondeq}, one may expect that the fluctuations of $d_i-\gamma_N$ are sufficiently loosely dependent across $i$ to further increase the speed of convergence from $N^{\frac12-\varepsilon}$ to $N^{1-\varepsilon}$ (which is the best one could expect from a law of large numbers aspect if the $d_i-\gamma_N$ were truly independent). It turns out that this intuition is correct.

Nonetheless, to proceed with the proof, it shall be quite involved to work directly with \eqref{eq:secondeq} which involves the rather intractable terms $d_i$ (as the random solutions to an implicit equation). As in \citep{COU14}, our approach will consist in first approximating $d_i$ by a much more tractable quantity. Letting $\gamma_N$ be this approximation is however not good enough this time since $\gamma_N-d_i$ is a non-obvious quantity of amplitude $O(N^{-\frac12})$ which, due to intractability, we shall not be able to average across $i$ into a $O(N^{-1})$ quantity. Thus, we need a refined approximation of $d_i$ which we shall take to be $\tilde{d}_i$ defined above. Intuitively, since $\tilde{d}_i$ is also a quadratic form closely related to $d_i$, we expect $d_i-\tilde{d}_i$ to be of order $O(N^{-1})$, which we shall indeed observe. With this approximation in place, $d_i$ can be replaced by $\tilde{d}_i$ in \eqref{eq:secondeq}, which now becomes a more tractable random variable (as it involves no implicit equation) that fluctuates around $\gamma_N$ at the expected $O(N^{-1})$ speed.

Let us then introduce the variable $\tilde{d}_i$ in \eqref{eq:firsteq} to obtain
\begin{align*}
a^*\hat{C}_N^{-1}b - a^*\hat{S}_N^{-1}b
	&= a^*\hat{C}_N^{-1} \left( \frac{\alpha}n \sum_{i=1}^n \left[ \frac1{\gamma_N} - \frac1{\tilde{d}_i} \right] z_iz_i^* \right) \hat{S}_N^{-1}b \nonumber \\
	&+ a^*\hat{C}_N^{-1} \left( \frac{\alpha}n \sum_{i=1}^n \left[ \frac1{\tilde{d}_i} - \frac1{d_i} \right] z_iz_i^* \right) \hat{S}_N^{-1}b \\
	&\triangleq \xi_1+\xi_2.
\end{align*}
We will now show that $\xi_1=\xi_1(\rho)$ and $\xi_2=\xi_2(\rho)$ vanish at the appropriate speed and uniformly so on $\mathcal R_\kappa$.

Let us first progress in the derivation of $\xi_1(\rho)$ from which we wish to discard the explicit dependence on $\hat{C}_N$. We have
\begin{align*}
	\xi_1 &= a^*\hat{C}_N^{-1} \left( \frac{\alpha}n \sum_{i=1}^n \left[ \frac1{\gamma_N} - \frac1{\tilde{d}_i} \right] z_iz_i^* \right) \hat{S}_N^{-1}b \\
	&= a^*\hat{S}_N^{-1} \left( \frac{\alpha}n \sum_{i=1}^n \left[ \frac1{\gamma_N} - \frac1{\tilde{d}_i} \right] z_iz_i^* \right) \hat{S}_N^{-1}b \nonumber \\ 
	&+ a^*(\hat{C}_N^{-1}-\hat{S}_N^{-1}) \left( \frac{\alpha}n \sum_{i=1}^n \left[ \frac1{\gamma_N} - \frac1{\tilde{d}_i} \right] z_iz_i^* \right) \hat{S}_N^{-1}b \\
	&= a^*\hat{S}_N^{-1} \left( \frac{\alpha}n \sum_{i=1}^n \frac{\tilde{d}_i-\gamma_N}{\gamma_N^2} z_iz_i^* \right) \hat{S}_N^{-1}b \nonumber \\
	&- a^*\hat{S}_N^{-1} \left( \frac{\alpha}n \sum_{i=1}^n \frac{(\tilde{d}_i-\gamma_N)^2}{\gamma_N^2\tilde{d}_i} z_iz_i^* \right) \hat{S}_N^{-1}b \nonumber \\
	&+ a^*(\hat{C}_N^{-1}-\hat{S}_N^{-1}) \left( \frac{\alpha}n \sum_{i=1}^n \left[ \frac1{\gamma_N} - \frac1{\tilde{d}_i} \right] z_iz_i^* \right) \hat{S}_N^{-1}b \\
	&\triangleq \xi_{11} + \xi_{12} + \xi_{13}.
\end{align*}
The terms $\xi_{12}$ and $\xi_{13}$ exhibit products of two terms that are expected to be of order $O(N^{-\frac12})$ and which are thus easily handled. As for $\xi_{11}$, it no longer depends on $\hat{C}_N$ and is therefore a standard random variable which, although involved, is technically tractable via standard random matrix methods. In order to show that $N^{1-\varepsilon}\max\{|\xi_{12}|,|\xi_{13}|\}\asto 0$ uniformly in $\rho$, we use the following lemma.
\begin{lemma}
	\label{le:1}
	For any $\varepsilon>0$,
	\begin{align*}
		\max_{1\leq i\leq n}\sup_{\rho\in\mathcal R_\kappa}N^{\frac12-\varepsilon}|\tilde{d}_i(\rho)-\gamma_N(\rho)| &\asto 0 \\
		\max_{1\leq i\leq n}\sup_{\rho\in\mathcal R_\kappa}N^{\frac12-\varepsilon}|d_i(\rho)-\gamma_N(\rho)| &\asto 0.
	\end{align*}
\end{lemma}
Note that, while the first result is a standard, easily established, random matrix result, the second result is the aforementioned refinement of the core result in the proof of \citep[Theorem~1]{COU14}.
\begin{proof}[Proof of Lemma~\ref{le:1}]
	We start by proving the first identity.
	From \cite[p.~17]{COU14} (taking $w=-\gamma_N\rho \alpha^{-1}$), we have, for each $p\geq 2$ and for each $1\leq k\leq n$,
	\begin{align*}
		\EE\left[ \left| \tilde{d}_k(\rho) - \gamma_N(\rho) \right|^p\right] &= O(N^{-\frac{p}2})
	\end{align*}
	where the bound does not depend on $\rho>\max\{0,1-1/c\}+\kappa$.
	Let now $\max\{0,1-1/c\}+\kappa=\rho_0<\ldots<\rho_{\lceil\sqrt{n}\rceil}=1$ be a regular sampling of $\mathcal R_\kappa$ in $\lceil\sqrt{n}\rceil$ intervals. We then have, from Markov inequality and the union bound on $n(\lceil\sqrt{n}\rceil+1)$ events, for $C>0$ given,
	\begin{align*}
		P \left( \max_{1\leq k\leq n,0\leq i\leq \lceil\sqrt{n}\rceil}  \left| N^{\frac12-\varepsilon}( \tilde{d}_k(\rho_i) - \gamma_N(\rho_i)) \right| > C \right) &\leq KN^{-p\varepsilon+\frac32}
	\end{align*}
	for some $K>0$ only dependent on $p$ and $C$. From the Borel Cantelli lemma, we then have $\max_{k,i} | N^{\frac12-\varepsilon}( \tilde{d}_k(\rho_i) - \gamma_N(\rho_i))|\asto 0$ as long as $-p\varepsilon+3/2<-1$, which is obtained for $p>5/(2\varepsilon)$. Using $|\gamma_N(\rho)-\gamma_N(\rho')|\leq K|\rho-\rho'|$ for some constant $K$ and each $\rho,\rho'\in\mathcal R_\kappa$ (see \citep[top of Section~5.1]{COU14}) and similarly $\max_{1\leq k\leq n}|\tilde{d}_k(\rho)-\tilde{d}_k(\rho')|\leq K|\rho-\rho'|$ for all large $n$ a.s.\@ (obtained by explicitly writing the difference and using the fact that $\Vert z_k\Vert^2/N$ is asymptotically bounded almost surely), we get
	\begin{align*}
		\max_{1\leq k\leq n}\sup_{\rho\in\mathcal R_\kappa}N^{\frac12-\varepsilon}|\tilde{d}_k(\rho)-\gamma_N(\rho)| &\leq \max_{k,i} N^{\frac12-\varepsilon} | \tilde{d}_k(\rho_i) - \gamma_N(\rho_i) | + KN^{-\varepsilon} \\ &\asto 0.
	\end{align*}

	The second result relies on revisiting the proof of \cite[Theorem~1]{COU14} incorporating the convergence speed on $\tilde{d}_k-\gamma_N$.
	For convenience and compatibility with similar derivations that appear later in the proof, we slightly modify the original proof of \cite[Theorem~1]{COU14}.
	We first define $f_i(\rho)=d_i(\rho)/\gamma_N(\rho)$ and relabel the $d_i(\rho)$ in such a way that $f_1(\rho)\leq \ldots\leq f_n(\rho)$ (the ordering may then depend on $\rho$). Then, we have by definition of $d_n(\rho)=\gamma_N(\rho) f_n(\rho)$
\begin{align*}
	\gamma_N(\rho) f_n(\rho) &= \frac1Nz_n^*\left( \alpha(\rho) \frac1n\sum_{i<n} \frac{z_iz_i^*}{\gamma_N(\rho) f_i(\rho)} + \rho I_N \right)^{-1}z_n \\
	&\leq \frac1Nz_n^*\left( \alpha(\rho) \frac1{f_n(\rho)} \frac1n\sum_{i<n} \frac{z_iz_i^*}{\gamma_N(\rho)} + \rho I_N \right)^{-1}z_n
\end{align*}
where we used $f_n(\rho)\geq f_i(\rho)$ for each $i$. The above is now equivalent to
\begin{align*}
	\gamma_N(\rho) &\leq \frac1Nz_n^* \left( \alpha(\rho) \frac1n\sum_{i<n} \frac{z_iz_i^*}{\gamma_N(\rho)} + f_n(\rho) \rho I_N \right)^{-1}z_n.
\end{align*}
We now make the assumption that there exists $\eta>0$ and a sequence $\{\rho^{(n)}\}\in\mathcal R_\kappa$ such that $f_n(\rho^{(n)})>1+N^{\eta-\frac12}$ infinitely often, which is equivalent to saying $d_n(\rho^{(n)})>\gamma_N(\rho^{(n)})(1+N^{\eta-\frac12})$ infinitely often (i.o.). Then, from these assumptions and the above first convergence result
\begin{align}
	\label{eq:gammaineq}
	\gamma_N(\rho^{(n)}) &\leq \frac1Nz_n^*\left( \alpha(\rho^{(n)}) \frac1n\sum_{i<n} \frac{z_iz_i^*}{\gamma_N(\rho^{(n)})} + \rho^{(n)} (1+N^{\eta-\frac12}) I_N \right)^{-1}z_n \nonumber \\
			%&= \frac1Nz_n^*\left( \alpha(\rho^{(n)}) \frac1n\sum_{i<n} \frac{z_iz_i^*}{\gamma_N(\rho^{(n)})} + \rho^{(n)} I_N \right)^{-1}z_n \nonumber \\
			&= \tilde{d}_n(\rho^{(n)}) - N^{\eta-\frac12} \frac1Nz_n^*\left( \frac1n\sum_{i<n} \frac{\alpha(\rho^{(n)}) z_iz_i^*}{\rho^{(n)}\gamma_N(\rho^{(n)})} + (1+N^{\eta-\frac12}) I_N \right)^{-1} \nonumber \\
			&\times\left( \frac1n\sum_{i<n} \frac{\alpha(\rho^{(n)}) z_iz_i^*}{\gamma_N(\rho^{(n)})} + \rho^{(n)} I_N \right)^{-1}z_n.
\end{align}
Now, by the first result of the lemma, letting $0<\varepsilon<\eta$, we have
\begin{align*}
	\left| \tilde{d}_n(\rho^{(n)}) - \gamma_N(\rho^{(n)}) \right| &\leq \max_{\rho\in\mathcal R_\kappa} \left| \tilde{d}_n(\rho) - \gamma_N(\rho) \right| \leq N^{\varepsilon-\frac12}
\end{align*}
for all large $n$ a.s., so that, for these large $n$, $\tilde{d}_n(\rho^{(n)}) \leq \gamma_N(\rho^{(n)}) + N^{\varepsilon-\frac12}$. Applying this inequality to the first right-end side term of \eqref{eq:gammaineq} and using the almost sure boundedness of the rightmost right-end side term entails
\begin{align*}
	0 \leq N^{\varepsilon-\frac12} - KN^{\eta-\frac12}
\end{align*}
for some $K>0$ for all large $n$ a.s. But, $N^{\varepsilon/2-1/2} - KN^{\eta/2-1/2}<0$ for all large $N$, which contradicts the inequality. Thus, our initial assumption is wrong and therefore, for each $\eta>0$, we have for all large $n$ a.s., $d_n(\rho)<\gamma_N(\rho)+N^{\eta-\frac12}$ uniformly on $\rho\in\mathcal R_\kappa$. The same calculus can be performed for $d_1(\rho)$ by assuming that $f_1(\rho^{\prime(n)})<1-N^{\eta-\frac12}$ i.o.\@ over some sequence $\rho^{\prime(n)}$; by reverting all inequalities in the derivation above, we similarly conclude by contradiction that $d_1(\rho)>\gamma_N(\rho)-N^{\eta-\frac12}$ for all large $n$, uniformly so in $\mathcal R_\kappa$. Together, both results finally lead, for each $\varepsilon>0$, to
\begin{align*}
	\max_{1\leq k\leq n} \sup_{\rho\in\mathcal R_\kappa} \left| N^{\frac12-\varepsilon} \left( d_k(\rho) - \gamma_N(\rho) \right) \right| &\asto 0
\end{align*}
obtained by fixing $\varepsilon$, taking $\eta$ such that $0<\eta<\varepsilon$, and using $\max_k \sup_{\rho} |d_k(\rho)-\gamma_N(\rho)|<N^{\eta-\frac12}$ for all large $n$ a.s.
\end{proof}

Thanks to Lemma~\ref{le:1}, expressing $\hat{C}_N^{-1}(\rho)-\hat{S}_N^{-1}(\rho)$ as a function of $d_i(\rho)-\gamma_N(\rho)$ and using the (almost sure) boundedness of the various terms involved, we finally get $N^{1-\varepsilon}\xi_{12}\asto 0$ and $N^{1-\varepsilon}\xi_{13}\asto 0$ uniformly on $\rho$. 

\bigskip

It then remains to handle the more delicate term $\xi_{11}$, which can be further expressed as
\begin{align*}
	\xi_{11} &= \frac{\alpha}{\gamma_N^2}a^*\hat{S}_N^{-1}\left( \frac1n \sum_{i=1}^n (\tilde{d}_i-\gamma_N) z_iz_i^* \right)\hat{S}_N^{-1}b \nonumber \\
	&= \frac{\alpha}{\gamma_N^2} \frac1n\sum_{i=1}^n a^*\hat{S}_N^{-1}z_iz_i^*\hat{S}_N^{-1}b \left( \tilde{d}_i - \gamma_N \right).
\end{align*}
For that, we will resort to the following lemma, whose proof is postponed to Appendix~\ref{app:key_lemma}.
\begin{lemma}
	Let $\first$ and $\second$ be random or deterministic vectors, independent of $z_1,\cdots,z_n$, such that $\max\left(\EE[\|\first\|^{k}], \EE[\|\second\|^k]\right) \leq K$ for some $K>0$ and all integer $k$. Then, for each integer $p$,
%For every $a,b \in \mathbb{C}^{N}$ with $\|a\|=\|b\| =1$, and for every $p \geq 1$,
\begin{align*}
\EE\left[\left|\frac{1}{n}\sum_{i=1}^n \first^*\SN z_iz_i^* \SN \second \left(\frac{1}{N}z_i^* \Si z_i-\gammanrho\right) \right|^{2p}\right] =O\left(N^{-2p}\right)
\end{align*}
\label{le:keylemma1}
\end{lemma}

By the Markov inequality and the union bound, similar to the proof of Lemma~\ref{le:1}, we get from Lemma~\ref{le:keylemma1} (with $a=c$ and $d=b$) that, for each $\eta>0$ and for each integer $p\geq 1$,
\begin{align*}
	P \left( \sup_{\rho\in \{\rho_0<\ldots<\rho_{\lceil\sqrt{n}\rceil}\}} N^{1-\varepsilon} |\xi_{11}| > \eta \right) &\leq K N^{-p\varepsilon+\frac12}
\end{align*}
with $K$ only function of $\eta$ and $\rho_0<\ldots<\rho_{\lceil\sqrt{n}\rceil}$ a regular sampling of $\mathcal R_\kappa$.
Taking $p>3/(2\varepsilon)$, we finally get from the Borel Cantelli lemma that
\begin{align*}
	N^{1-\varepsilon} \xi_{11} &\asto 0
\end{align*}
uniformly on $\{\rho_0,\ldots,\rho_{\lceil\sqrt{n}\rceil}\}$ and finally, using Lipschitz arguments as in the proof of Lemma~\ref{le:1}, uniformly on $\mathcal R_\kappa$. Putting all results together, we finally have
\begin{align*}
	\sup_{\rho\in\mathcal R_\kappa} N^{1-\varepsilon} |\xi_1(\rho)| &\asto 0
\end{align*}
which concludes the first part of the proof.

\bigskip 

We now continue with $\xi_2(\rho)$. In order to prove $N^{1-\varepsilon}\xi_2(\rho)\asto 0$ uniformly on $\rho\in \mathcal R_\kappa$, it is sufficient (thanks to the boundedness of the various terms involved) to prove that
\begin{align*}
	\max_{1\leq i\leq n}\sup_{\rho\in\mathcal R_\kappa}  \left| N^{1-\varepsilon} \left(\tilde{d}_i(\rho) - d_i(\rho)\right)\right| &\asto 0.
\end{align*}

To obtain this result, we first need the following fundamental proposition.
\begin{proposition}
	\label{prop:2}
	For any $\varepsilon>0$,
\begin{align*}
	\max_{1\leq k\leq n}\sup_{\rho \in\mathcal R_\kappa} \left| N^{1-\varepsilon} \left( \tilde{d}_k(\rho) - \frac1N z_k^*\left( \alpha(\rho) \frac1n \sum_{i\neq k} \frac{z_iz_i^*}{\tilde{d}_i(\rho)} + \rho I_N \right)^{-1}z_k \right) \right| \asto 0.
\end{align*}
\end{proposition}
\begin{proof}
	By expanding the definition of $\tilde{d}_k$, first observe that
	\begin{align*}
		&\tilde{d}_k - \frac1N z_k^*\left( \alpha \frac1n \sum_{i\neq k} \frac{z_iz_i^*}{\tilde{d}_i} + \rho I_N \right)^{-1}z_k \\
		&=\alpha \frac1n\sum_{i\neq k} \frac1N z_k^* \hat{S}_{(k)}^{-1} z_iz_i^* \frac{\gamma_N - \tilde{d}_i}{\gamma_N\tilde{d}_i} \left( \alpha \frac1n \sum_{j\neq k} \frac{z_jz_j^*}{\tilde{d}_j} + \rho I_N \right)^{-1}z_k.
	\end{align*}
	Similar to the derivation of $\xi_1$, we now proceed to approximating $\tilde{d}_i$ in the central denominator and each $\tilde{d}_j$ in the rightmost inverse matrix by the non-random $\gamma_N$. We obtain (from Lemma~\ref{le:1})
	\begin{align*}
	&\tilde{d}_k - \frac1N z_k^*\left( \alpha \frac1n \sum_{i\neq k} \frac{z_iz_i^*}{\tilde{d}_i} + \rho I_N \right)^{-1}z_k \\
	&= \frac{\alpha}{\gamma_N^2} \frac1n\sum_{i\neq k} \frac1N z_k^* \hat{S}_{(k)}^{-1} z_iz_i^* (\gamma_N - \tilde{d}_i) \hat{S}_{(k)}^{-1} z_k + o(N^{\varepsilon-1})
	\end{align*}
	almost surely, for $\varepsilon>0$ and uniformly so on $\rho$. %Again, similar to the derivation of $\xi_1$, we continue by extracting the dependence of each $\hat{S}_{(k)}$ matrix on $z_i$ by writing
	The objective is then to show that the first right-hand side term is  $o(N^{\varepsilon-1})$ almost surely and that this holds uniformly on $k$ and $\rho$. This is achieved by applying Lemma~\ref{le:keylemma1} with $c=d=z_k$. Indeed, Lemma~\ref{le:keylemma1} ensures that, for each integer $p$,\footnote{Note that Lemma~\ref{le:keylemma1} can strictly be applied here for $n-1$ instead of $n$; but since $1/n-1/(n-1)=O(n^{-2})$, this does not affect the result.}
\begin{align*}
	\EE\left[\left|\frac{1}{n}\sum_{i\neq k}\frac{1}{N}z_k^*S_{(k)}^{-1}(\rho) z_iz_i^* S_{(k)}^{-1}(\rho)z_k\left(\frac{1}{N}z_i^*{S}_{(i,k)}^{-1}(\rho)z_i-\gamma_N(\rho)\right)\right|^p\right]=O(N^{-p})
\end{align*}
%the following lemma, which is quite similar to the already established Lemma~\ref{le:keylemma1}.	

%\begin{lemma}
%	\label{le:keylemma2}
%	For each $\rho\in\mathcal R_\kappa$, each $1\leq k\leq n$ and each integer $p\geq 1$,
%\begin{align*}
%	\EE \left[ \left| \frac1n\sum_{i\neq k} \frac1Nz_k^* \hat{S}_{(i,k)}^{-1}(\rho)z_iz_i^*\hat{S}_{(i,k)}^{-1}(\rho)z_k  \left( \frac1Nz_i^*\hat{S}_{(i,k)}^{-1}(\rho)z_i - \gamma_N(\rho) \right) \right|^p \right] &= O(N^{-p}) 
%\end{align*}
%where the right-hand side term does not depend on $k$ or $\rho$.
%\end{lemma}
%Similar to Lemma~\ref{le:keylemma1}, the proof of Lemma~\ref{le:keylemma2} is deferred to the appendix.

From this lemma,  applying Markov's inequality, we have for each $k$,
\begin{align*}
	P \left( N^{1-\varepsilon} \left| \frac1n\sum_{i\neq k} \frac1Nz_k^* \hat{S}_{(k)}^{-1}z_iz_i^*\hat{S}_{(k)}^{-1}z_k  \left( \frac1Nz_i^*\hat{S}_{(i,k)}^{-1}z_i - \gamma_N \right) \right| > \eta \right) \leq K N^{-p\varepsilon}
\end{align*}
for some $K>0$ only dependent on $\eta>0$. Applying the union bound on the $n(n+1)$ events for $k=1,\ldots,n$ and for $\rho\in\{\rho_0,\ldots,\rho_n\}$, regular $n$-discretization of $\mathcal R_\kappa$, we then have 
\begin{align*}
	&P \left( \max_{k,j} N^{1-\varepsilon} \left| \frac1n\sum_{i\neq k} \frac1Nz_k^* \hat{S}_{(k)}^{-1}z_iz_i^*\hat{S}_{(k)}^{-1}z_k  \left( \frac1Nz_i^*\hat{S}_{(i,k)}^{-1}z_i - \gamma_N(\rho_j) \right) \right| > \eta \right) \nonumber \\
	&\leq K N^{-p\varepsilon+2}.
\end{align*}
Taking $p>3/\varepsilon$, by the Borel Cantelli lemma the above convergence holds almost surely, we finally get
\begin{align*}
	\max_{k,j} \left| N^{1-\varepsilon} \left( \tilde{d}_k(\rho_j) - \frac1N z_k^*\left( \alpha(\rho_j) \frac1n \sum_{i\neq k} \frac{z_iz_i^*}{\tilde{d}_i(\rho_j)} + \rho_j I_N \right)^{-1}z_k \right) \right| \asto 0.
\end{align*}
Using the $\rho$-Lipschitz property (which holds almost surely so for all large $n$ a.s.) on both terms in the above difference concludes the proof of the proposition. 
\end{proof}

The crux of the proof for the convergence of $\xi_2$ starts now.
In a similar manner as in the proof of Lemma~\ref{le:1}, we define $\tilde{f}_i(\rho)=d_i(\rho)/\tilde{d}_i(\rho)$ and reorder the indexes in such a way that $\tilde{f}_1(\rho)\leq \ldots\leq \tilde{f}_n(\rho)$ (this ordering depending on $\rho$). Then, by definition of $d_n(\rho)=\tilde{f}_i(\rho)\tilde{d}_i(\rho)$,
\begin{align*}
	\tilde{d}_n(\rho) \tilde{f}_n(\rho) &= \frac1Nz_n^*\left( \alpha(\rho) \frac1n\sum_{i<n} \frac{z_iz_i^*}{\tilde{d}_i(\rho)\tilde{f}_i(\rho)} + \rho I_n \right)^{-1}z_n \\
	&\leq \frac1Nz_n^*\left( \alpha(\rho) \frac1{\tilde{f}_n(\rho)} \frac1n\sum_{i<n} \frac{z_iz_i^*}{\tilde{d}_i(\rho)} + \rho I_n \right)^{-1}z_n
\end{align*}
where we used $\tilde{f}_n(\rho)\geq \tilde{f}_i(\rho)$ for each $i$. This inequality is equivalent to
\begin{align*}
	\tilde{d}_n(\rho) &\leq \frac1Nz_n^* \left( \alpha(\rho) \frac1n\sum_{i<n} \frac{z_iz_i^*}{\tilde{d}_i(\rho)} + \tilde{f}_n(\rho) \rho I_n \right)^{-1}z_n.
\end{align*}
Assume now that, over some sequence $\{\rho^{(n)}\}\in\mathcal R_\kappa$, $\tilde{f}_n(\rho^{(n)})>1+N^{\eta-1}$ infinitely often for some $\eta>0$ (or equivalently, $d_n(\rho^{(n)})>\tilde{d}_n(\rho^{(n)})+N^{\eta-1}$ i.o.). Then we would have
\begin{align*}
	\tilde{d}_n(\rho^{(n)}) &\leq \frac1Nz_n^*\left( \alpha(\rho^{(n)}) \frac1n\sum_{i<n} \frac{z_iz_i^*}{\tilde{d}_i(\rho^{(n)})} + \rho^{(n)} (1+N^{\eta-1}) I_N \right)^{-1}z_n \\
	%&= \frac1Nz_n^*\left( \alpha(\rho) \frac1n\sum_{i<n} \frac{z_iz_i^*}{\tilde{d}_i(\rho^{(n)})} + \rho I_n \right)^{-1}z_n \nonumber \\
	&=\tilde{d}_n(\rho^{(n)}) - N^{\eta-1} \frac1Nz_n^*\left( \frac1n\sum_{i<n} \frac{\alpha(\rho^{(n)}) z_iz_i^*}{\rho^{(n)}\tilde{d}_i(\rho^{(n)})} + (1+N^{\eta-1}) I_N \right)^{-1}\nonumber \\
	&\times\left( \frac1n\sum_{i<n} \frac{\alpha(\rho^{(n)}) z_iz_i^*}{\tilde{d}_i(\rho^{(n)})} + \rho I_N \right)^{-1}z_n.
\end{align*}
But, by Proposition~\ref{prop:2}, letting $0<\varepsilon<\eta$, we have, for all large $n$ a.s.,
\begin{align*}
	\frac1Nz_n^*\left( \alpha(\rho^{(n)}) \frac1n\sum_{i<n} \frac{z_iz_i^*}{\tilde{d}_i(\rho^{(n)})} + \rho^{(n)} I_n \right)^{-1}z_n \leq \tilde{d}_n(\rho^{(n)}) + N^{\varepsilon-1}
\end{align*}
which, along with the uniform boundedness of the $\tilde{d}_i$ away from zero, leads to
\begin{align*}
	\tilde{d}_n(\rho^{(n)}) &\leq \tilde{d}_n(\rho^{(n)}) + N^{\varepsilon-1} - KN^{\eta-1}
\end{align*}
for some $K>0$. But, as $N^{\varepsilon-1} - KN^{\eta-1}<0$ for all large $N$, we obtain a contradiction. Hence, for each $\eta>0$, we have for all large $n$ a.s., $d_n(\rho)<\tilde{d}_n(\rho)+N^{\eta-1}$ uniformly on $\rho\in\mathcal R_\kappa$. Proceeding similarly with $d_1(\rho)$, and exploiting $\limsup_n \sup_\rho \max_i |\tilde{d}_i(\rho)|=O(1)$ a.s.\@, we finally have, for each $0<\varepsilon<\frac12$, that
\begin{align*}
	\max_{1\leq k\leq n} \sup_{\rho\in\mathcal R_\kappa}\left| N^{1-\varepsilon} \left( d_k(\rho) - \tilde{d}_k(\rho) \right) \right| &\asto 0
\end{align*}
(for this, take an $\eta$ such that $0<\eta<\varepsilon$ and use $\max_k \sup_\rho |d_k(\rho)-\tilde{d}_k(\rho)|<N^{\eta-1}$ for all large $n$ a.s.).

Getting back to $\xi_2$, we now have
\begin{align*}
	N^{1-\varepsilon} |\xi_2(\rho)| &= N^{1-\varepsilon} \left|a^*\hat{C}_N^{-1}(\rho) \left( \frac{\alpha(\rho)}n \sum_{i=1}^n \frac{d_i(\rho)-\tilde{d}_i(\rho)}{d_i(\rho)\tilde{d}_i(\rho)} z_iz_i^* \right)\hat{S}_N^{-1}(\rho)b\right|.
\end{align*}
But, from the above result,
\begin{align*}
	N^{1-\varepsilon} \left\Vert \frac{\alpha(\rho)}n \sum_{i=1}^n \frac{d_i(\rho)-\tilde{d}_i(\rho)}{d_i(\rho)\tilde{d}_i(\rho)} z_iz_i^* \right\Vert &\leq N^{1-\varepsilon} \max_{1\leq k\leq n} \left| \frac{d_k(\rho)-\tilde{d}_k(\rho)}{d_k(\rho)\tilde{d}_k(\rho)} \right| \left\Vert \frac{\alpha(\rho)}n \sum_{i=1}^n z_iz_i^* \right\Vert \\
	&\asto 0
\end{align*}
uniformly so on $\rho\in\mathcal R_\kappa$ which, along with the boundedness of $\Vert \hat{C}_N^{-1}\Vert$, $\Vert \hat{S}_N^{-1}\Vert$, $\Vert a\Vert$, and $\Vert b\Vert$, finally gives $N^{1-\varepsilon} \xi_2\asto 0$ uniformly on $\rho\in\mathcal R_\kappa$ as desired.

\bigskip

We have then proved that for each $\varepsilon>0$,
\begin{align*}
	\sup_{\rho\in\mathcal R_\kappa}\left| N^{1-\varepsilon} \left( a^*\hat{C}_N^{-1}(\rho)b - a^*\hat{S}_N^{-1}(\rho)b \right)\right| \asto 0
\end{align*}
which proves Theorem~\ref{th:bilin} for $k=-1$. The generalization to arbitrary $k$ is rather immediate. Writing recursively $\hat{C}_N^k-\hat{S}_N^k=\hat{C}_N^{k-1}(\hat{C}_N-\hat{S}_N)+(\hat{C}_N^{k-1}-\hat{S}_N^{k-1})\hat{S}_N$, for positive $k$ or $\hat{C}_N^k-\hat{S}_N^k=\hat{C}_N^k(\hat{S}_N-\hat{C}_N)\hat{S}_N^{-1}+(\hat{C}_N^{k-1}-\hat{S}_N^{k-1})\hat{S}_N^{-1}$, \eqref{eq:firsteq} becomes a finite sum of terms that can be treated almost exactly as in the proof. This concludes the proof of Theorem~\ref{th:bilin}.

\subsection{Fluctuations of the GLRT detector}
This section is devoted to the proof of Theorem~\ref{th:T}, which shall fundamentally rely on Theorem~\ref{th:bilin}. The proof will be established in two steps. First, we shall prove the convergence for each $\rho\in\mathcal R_\kappa$, which we then generalize to the uniform statement of the theorem.

Let us then fix $\rho\in\mathcal R_\kappa$ for the moment.
In anticipation of the eventual replacement of $\hat{C}_N(\rho)$ by $\underline{\hat{S}}_N(\underline\rho)$, we start by studying the fluctuations of the bilinear forms involved in $T_N(\rho)$ but with $\hat{C}_N(\rho)$ replaced by $\underline{\hat{S}}_N(\underline\rho)$ (note that $T_N(\rho)$ remains constant when scaling $\hat{C}_N(\rho)$ by any constant, so that replacing $\hat{C}_N(\rho)$ by $\underline{\hat{S}}_N(\underline\rho)$ instead of by $\underline{\hat{S}}_N(\underline\rho)\cdot\frac1N\tr \hat{S}_N(\rho)$ as one would expect comes with no effect). 

Our first goal is to show that the vector $\sqrt{N}(\Re[y^*\underline{\hat S}^{-1}_N(\underline\rho) p],\Im[y^*\underline{\hat S}^{-1}_N(\underline\rho) p])$ is asymptotically well approximated by a zero mean Gaussian vector with given covariance matrix. To this end, let us denote $A=[y~p]\in\CC^{N\times 2}$ and $Q_N=Q_N(\underline\rho)=(I_N+(1-\underline{\rho}) m(-\underline{\rho})C_N)^{-1}$. Then, from \cite[Lemma~5.3]{CHA12} (adapted to our current notations and normalizations), for any Hermitian $B\in\CC^{2\times 2}$ and for any $u\in\RR$,
\begin{align}
	&\EE\left[\exp\left( \imath \sqrt{N} u \tr BA^*\left[\underline{\hat{S}}_N(\underline\rho)^{-1}-\frac1{\underline\rho}Q_N(\underline\rho) \right]A \right) ~\Big|~y\right] \nonumber \\
	&= \exp\left(-\frac12 u^2 \Delta_N^2(B;y;p) \right) + O(N^{-\frac12}) \label{eq:characteristic_fun}
\end{align}
where we denote by $\EE[\cdot |y]$ the conditional expectation with respect to the random vector $y$ and where
\begin{align*}
	\Delta_N^2(B;y;p) &\triangleq \frac{cm(-\underline{\rho})^2(1-\underline{\rho})^2\tr \left( A B A^* C_NQ_N^2(\underline\rho)\right)^2}{\underline{\rho}^2 \left(1-c m(-\underline{\rho})^2(1-\underline{\rho})^2 \frac1N\tr C_N^2Q_N^2(\underline\rho)\right) } .
\end{align*}

Also, we have from classical central limit results on Gaussian random variables
\begin{align*}
	\EE\left[\exp\left( \imath \sqrt{N} u \tr B \left[A^*Q_N(\underline\rho)A - \Gamma_N \right]\right)\right] &= \exp \left( -\frac12u^2 \Delta_N^{\prime 2}(B;p) \right) + O(N^{-\frac12})
\end{align*}
where
\begin{align*}
	\Gamma_N &\triangleq \frac1{\underline\rho}\begin{bmatrix} \frac1N\tr C_NQ_N(\underline\rho)& 0 \\ 0 & p^*Q_N(\underline\rho)p \end{bmatrix} \\
	\Delta_N^{\prime 2}(B;p) &\triangleq \frac{B_{11}^2}{\underline\rho^2} \frac1N\tr C_N^2Q_N^2(\underline\rho) + \frac{2|B_{12}|^2}{\underline\rho^2} p^*C_NQ_N^2(\underline\rho)p.
\end{align*}
Besides, the $O(N^{-\frac12})$ terms in the right-hand side of \eqref{eq:characteristic_fun} remains $O(N^{-\frac12})$ under expectation over $y$ (for this, see the proof of \cite[Lemma~5.3]{CHA12}).

Altogether, we then have
\begin{align*}
	& \EE\left[\exp\left( \imath \sqrt{N} u \tr B\left[ A^*\underline{\hat{S}}_N^{-1}(\underline\rho)A - \Gamma_N\right] \right)\right] \nonumber \\
	&= \EE \left[\exp\left(-\frac12 u^2 \Delta_N^2(B;y;p)\right) \right] \exp\left(-\frac12 u^2 \Delta_N^{\prime 2}(B;p) \right) + O(N^{-\frac12}).
\end{align*}

Note now that
\begin{align*}
	A^* C_NQ_N^2(\underline\rho)A - \Upsilon_N &\asto 0
\end{align*}
where
\begin{align*}
	\Upsilon_N &\triangleq \begin{bmatrix} \frac1N\tr C_N^2 Q_N^2(\underline\rho) & 0 \\ 0 & p^*C_NQ_N^2(\underline\rho) p \end{bmatrix}
\end{align*}
so that, by dominated convergence, we obtain
\begin{align*}
	&\EE\left[\exp\left( \imath \sqrt{N} u \tr B\left[ A^*\underline{\hat{S}}_N^{-1}(\underline\rho)A - \Gamma_N\right] \right)\right] \nonumber \\ 
	&= \exp\left(-\frac12 u^2 \left[ \Delta_N^2(B;p) + \Delta_N^{\prime 2}(B;p) \right]\right) + o(1)
\end{align*}
where we defined
\begin{align*}
	\Delta_N^2(B;p) &\triangleq \frac{cm(-\underline{\rho})^2(1-\underline{\rho})^2 \tr \left( B \Upsilon_N\right)^2}{\underline{\rho}^2 \left(1-c m(-\underline{\rho})^2(1-\underline{\rho})^2 \frac1N\tr C_N^2Q_N^2(\underline\rho)\right)}.
\end{align*}

By a generalized L\'evy's continuity theorem argument (see e.g.\@ \cite[Proposition~6]{HAC06}) and the Cramer-Wold device, we conclude that
\begin{align*}
\sqrt{N}\left(y^*\underline{\hat{S}}_N^{-1}(\underline{\rho})y,\Re[y^*\underline{\hat{S}}_N^{-1}(\underline{\rho})p],\Im[y^*\underline{\hat{S}}_N^{-1}(\underline{\rho})p]\right) - Z_N &= o_P(1)
\end{align*}
where $Z_N$ is a Gaussian random vector with mean and covariance matrix prescribed by the above approximation of $\sqrt{N} \tr BA^*\underline{\hat{S}}_N^{-1}A$ for each Hermitian $B$. In particular, taking $B_1\in\left\{ \left[\begin{smallmatrix} 0 & \frac12 \\ \frac12 & 0\end{smallmatrix}\right],\left[\begin{smallmatrix} 0 & \frac{\imath}2 \\ -\frac{\imath}2 & 0\end{smallmatrix}\right]\right\}$ to retrieve the asymptotic variances of $\sqrt{N}\Re[y^*\underline{\hat{S}}_N^{-1}(\underline{\rho})p]$ and $\sqrt{N}\Im[y^*\underline{\hat{S}}_N^{-1}(\underline{\rho})p]$, respectively, gives
\begin{align*}
	\Delta_N^2(B_1;p) &=\frac1{2 \underline{\rho}^2}  p^*C_NQ_N^2(\underline\rho)p \frac{c m(-\underline{\rho})^2(1-\underline{\rho})^2 \frac1N\tr C_N^2Q_N^2(\underline\rho) }{1-c m(-\underline{\rho})^2(1-\underline{\rho})^2 \frac1N\tr C_N^2Q_N^2(\underline\rho)} \\
	\Delta_N^{\prime 2}(B_1;p) &=\frac1{2\underline{\rho}^2} p^*C_NQ_N^2(\underline\rho)p
\end{align*}
and thus $\sqrt{N}(\Re[y^*\underline{\hat{S}}_N^{-1}(\underline{\rho})p],\Im[y^*\underline{\hat{S}}_N^{-1}(\underline{\rho})p])$ is asymptotically equivalent to a Gaussian vector with zero mean and covariance matrix
\begin{align*}
	(\Delta_N^2(B_1;p)+\Delta_N^{\prime 2}(B_1;p))I_2 &= \frac1{2\underline{\rho}^2} \frac{p^*C_NQ_N^2(\underline\rho)p}{1-c m(-\underline{\rho})^2(1-\underline{\rho})^2 \frac1N\tr C_N^2Q_N^2(\underline\rho)} I_2.
\end{align*}

We are now in position to apply Theorem~\ref{th:bilin}. Reminding that $\hat{S}_N^{-1}(\rho) (\rho + \frac1{\gamma_N(\rho)} \frac{1-\rho}{1-(1-\rho)c}) =\underline{\hat{S}}_N^{-1}(\underline\rho)$, we have by Theorem~\ref{th:bilin} for $k=-1$
\begin{align*}
	\sqrt{N} A^*\left[\hat{C}_N^{-1}(\rho)-\frac{\underline{\hat{S}}_N(\underline\rho)^{-1}}{\rho + \frac1{\gamma_N(\rho)} \frac{1-\rho}{1-(1-\rho)c}}\right] A &\asto 0.
\end{align*}
Since almost sure convergence implies weak convergence, $\sqrt{N} A^*\hat{C}_N^{-1}(\rho)A$ has the same asymptotic fluctuations as $\sqrt{N} A^*\underline{\hat{S}}_N^{-1}(\underline\rho)A/(\frac1N\tr \hat{S}_N(\rho))$.
Also, as $T_N(\rho)$ remains identical when scaling $\hat{C}_N^{-1}(\rho)$ by $\frac1N\tr \hat{S}_N(\rho)$, only the fluctuations of $\sqrt{N} A^*\underline{\hat{S}}_N^{-1}(\underline\rho)A$ are of interest, which were previously derived. We then finally conclude by the delta method (or more directly by Slutsky's lemma) that
\begin{align*}
	\sqrt{\frac{N}{y^*\hat{C}_N^{-1}(\rho)y p^*\hat{C}_N^{-1}(\rho)p}} \begin{bmatrix} \Re\left[ y^*\hat{C}_N^{-1}(\rho)p \right] \\ \Im\left[ y^*\hat{C}_N^{-1}(\rho)p \right] \end{bmatrix} - \sigma_N(\underline{\rho}) Z' = o_P(1)
\end{align*}
for some $Z'\sim \mathcal N(0,I_2)$ and
\begin{align*}
	\sigma_N^2(\underline\rho) &\triangleq \frac12 \frac{ p^*C_NQ_N^2(\underline\rho)p}{  p^*Q_N(\underline\rho)p\cdot \frac1N\tr C_NQ_N(\underline\rho)\cdot \left(1-c m(-\underline{\rho})^2(1-\underline{\rho})^2 \frac1N\tr C_N^2Q_N^2(\underline\rho)\right)  }.
\end{align*}
It unfolds that, for $\gamma>0$,
\begin{align}
	\label{eq:conv_proba_rho}
	P\left( T_N(\rho) > \frac{\gamma}{\sqrt{N}} \right) - \exp \left( - \frac{\gamma^2}{2\sigma_N^2(\underline\rho)} \right) \to 0
\end{align}
as desired.

\bigskip

The second step of the proof is to generalize \eqref{eq:conv_proba_rho} to uniform convergence across $\rho\in\mathcal R_\kappa$. To this end, somewhat similar to above, we shall transfer the distribution $P(\sqrt{N}T_N(\rho) > \gamma)$ to $P(\sqrt{N}\underline T_N(\rho) > \gamma)$ by exploiting the uniform convergence of Theorem~\ref{th:bilin}, where we defined\begin{align*}
	\underline T_N(\rho) &\triangleq \frac{\left|y^*\underline{\hat S}_N(\underline\rho) p\right|}{\sqrt{y^*\underline{\hat S}_N(\underline\rho) y}\sqrt{p^*\underline{\hat S}_N(\underline\rho) p}}
\end{align*}
and exploit a $\rho$-Lipschitz property of $\sqrt{N}\underline T_N(\rho)$ to reduce the uniform convergence over $\mathcal R_\kappa$ to a uniform convergence over finitely many values of $\rho$. 

The $\rho$-Lipschitz property we shall need is as follows: for each $\varepsilon>0$
\begin{align}
	\label{eq:tightness_condition}
	\lim_{\delta\to 0} \lim_{N\to\infty} P\left( \sup_{\substack{\rho,\rho'\in\mathcal R_\kappa \\ |\rho-\rho'|<\delta} } \sqrt{N}\left|T_N(\rho)-T_N(\rho')\right| > \varepsilon \right) &= 0.
\end{align}
Let us prove this result. By Theorem~\ref{th:bilin}, since almost sure convergence implies convergence in distribution, we have 
\begin{align*}
	P\left( \sup_{\rho\in\mathcal R_\kappa} \sqrt{N}\left|T_N(\rho)-\underline T_N(\rho)\right| > \varepsilon \right) &\to 0.
\end{align*}
Applying this result to \eqref{eq:tightness_condition} induces that it is sufficient to prove \eqref{eq:tightness_condition} for $\underline T_N(\rho)$ in place of $T_N(\rho)$.
Let $\eta>0$ small and $\mathcal A_N^\eta\triangleq \{\exists \underline\rho \in \mathcal R_\kappa,y^*\underline{\hat{S}}_N^{-1}(\underline\rho)yp^*\underline{\hat{S}}_N^{-1}(\underline\rho)p<\eta\}$. Developing the difference $\underline T_N(\rho)-\underline T_N(\rho')$ and isolating the denominator according to its belonging to $\mathcal A_N^\eta$ or not, we may write
\begin{align*}
	& P\left( \sup_{\substack{\rho,\rho'\in\mathcal R_\kappa \\ |\rho-\rho'|<\delta} } \sqrt{N}\left|\underline T_N(\rho)-\underline T_N(\rho')\right| > \varepsilon \right) \nonumber \\
	&\leq P\left( \mathcal A_N^\eta \right) + P\left( \sup_{\substack{\rho,\rho'\in\mathcal R_\kappa \\ |\rho-\rho'|<\delta} } \sqrt{N} V_N(\rho,\rho')>\varepsilon\eta \right)
\end{align*}
where
\begin{align*}
V_N(\rho,\rho') &\triangleq \left| y^*\underline{\hat{S}}_N^{-1}(\underline\rho)p \right| \sqrt{y^*\underline{\hat{S}}_N^{-1}(\underline\rho')y}\sqrt{p^*\underline{\hat{S}}_N^{-1}(\underline\rho')p} \nonumber \\
&- \left| y^*\underline{\hat{S}}_N^{-1}(\underline\rho')p \right| \sqrt{y^*\underline{\hat{S}}_N^{-1}(\underline\rho)y}\sqrt{p^*\underline{\hat{S}}_N^{-1}(\underline\rho)p}.
\end{align*}

From classical random matrix results, $P( \mathcal A_N^\eta )\to 0$ for a sufficiently small choice of $\eta$. To prove that $\lim_\delta \limsup_n P(\sup_{|\rho-\rho'|<\delta} \sqrt{N} V_N(\rho,\rho')>\varepsilon\eta)=0$, it is then sufficient to show that
\begin{align}
	\label{eq:rho-rho'}
	\lim_{\delta\to 0} \limsup_n P\left( \sup_{\substack{\rho,\rho'\in\mathcal R_\kappa \\ |\rho-\rho'|<\delta}} \sqrt{N}|y^*\underline{\hat S}_N(\underline\rho)^{-1}p-y^*\underline{\hat S}_N(\underline\rho')^{-1}p| > \varepsilon' \right) = 0
\end{align}
for any $\varepsilon'>0$ and similarly for $y^*\underline{\hat S}_N(\underline\rho)^{-1}y-y^*\underline{\hat S}_N(\underline\rho')^{-1}y$ and $p^*\underline{\hat S}_N(\underline\rho)^{-1}p-p^*\underline{\hat S}_N(\underline\rho')^{-1}p$.
Let us prove \eqref{eq:rho-rho'}, the other two results following essentially the same line of arguments.
For this, by \cite[Corollary~16.9]{KAL02} (see also \cite[Theorem~12.3]{BIL68}), it is sufficient to prove, say
\begin{align*}
	\sup_{\substack{\rho,\rho'\in \mathcal R_\kappa\\ \rho\neq \rho'}} \sup_n  \frac{\EE \left[\sqrt{N}\left|y^*\underline{\hat S}_N(\underline\rho)^{-1}p-y^*\underline{\hat S}_N(\underline\rho')^{-1}p \right|^2\right]}{|\rho-\rho'|^2} < \infty.
\end{align*}
But then, remarking that
\begin{align*}
	&\sqrt{N} y^*\underline{\hat S}_N(\underline\rho)^{-1}p-y^*\underline{\hat S}_N(\underline\rho')^{-1}p \nonumber \\
	&= (\underline\rho'-\underline\rho) \sqrt{N} y^*\underline{\hat S}_N(\underline\rho)^{-1}\left( I_N - \frac1n\sum_{i=1}^n z_iz_i^* \right)\underline{\hat S}_N(\underline\rho')^{-1}p
\end{align*}
this reduces to showing that
\begin{align*}
	\sup_{\rho,\rho'\in \mathcal R_\kappa} \sup_n \EE \left[ N \left| y^*\underline{\hat S}_N(\underline\rho)^{-1}\left( I_N - \frac1n\sum_{i=1}^n z_iz_i^* \right)\underline{\hat S}_N(\underline\rho')^{-1}p \right|^2\right] <\infty.
\end{align*}
Conditioning first on $z_1,\ldots,z_n$, this further reduces to showing
\begin{align*}
	\sup_{\rho,\rho'\in \mathcal R_\kappa} \sup_n \EE \left[ \left\Vert \underline{\hat S}_N(\underline\rho)^{-1}\left( I_N - \frac1n\sum_{i=1}^n z_iz_i^* \right)\underline{\hat S}_N(\underline\rho')^{-1}p \right\Vert^2\right] <\infty.
\end{align*}
But this is yet another standard random matrix result, obtained e.g., by noticing that
\begin{align*}
	\left\Vert \underline{\hat S}_N(\underline\rho)^{-1}\left( I_N - \frac1n\sum_{i=1}^n z_iz_i^* \right)\underline{\hat S}_N(\underline\rho')^{-1}p \right\Vert^2 &\leq \frac1{\kappa^4}\left\Vert I_N - \frac1n\sum_{i=1}^n z_iz_i^* \right\Vert^2
\end{align*}
which remains of uniformly finite expectation (left norm is vector Euclidean norm, right norm is matrix spectral norm). This completes the proof of \eqref{eq:tightness_condition}.
%We similarly to prove that
%\begin{align*}
%	\lim_{\delta\to 0} \lim_{n\to\infty} P\left( \sup_{\substack{\rho,\rho'\in\mathcal R_\kappa \\ |\rho-\rho'|<\delta}} \sqrt{N}\left|\sqrt{y^*\underline{\hat S}_N(\underline{\rho})^{-1}y}-\sqrt{y^*\underline{\hat S}_N(\underline{\rho}')^{-1}y}\right| > \varepsilon \right) &= 0
%\end{align*}
%(for this term, multiplying each side of the inner inequality by $\sqrt{y^*\underline{\hat S}_N(\underline{\rho})^{-1}y}+\sqrt{y^*\underline{\hat S}_N(\underline{\rho}')^{-1}y}$ and exploiting the almost sure boundedness of the latter, it is sufficient to evaluate the same term without the square roots) and for the same term with $y$ replaced by $p$. 

\smallskip

Getting back to our original problem, let us now take $\varepsilon>0$ arbitrary, $\rho_1<\ldots<\rho_K$ be a regular sampling of $\mathcal R_\kappa$, and $\delta=1/K$. Then by \eqref{eq:conv_proba_rho}, $K$ being fixed, for all $n>n_0(\varepsilon)$,
\begin{align}
	\label{eq:Kfoldconv}
	\max_{1\leq k\leq K} \left| P\left( T_N(\rho_i) > \frac{\gamma}{\sqrt{N}} \right) - \exp\left( -\frac{\gamma^2}{2\sigma_N^2(\rho_i)} \right) \right| &< \varepsilon.
\end{align}
Also, from \eqref{eq:tightness_condition}, for small enough $\delta$,
\begin{align*}
	&\max_{1\leq k\leq K} P\left( \sup_{\substack{\rho\in\mathcal R_\kappa \\ |\rho-\rho_k|<\delta}} \sqrt{N}|T_N(\rho)-T_N(\rho_k)| > \gamma \zeta \right) \nonumber \\
	&\leq P\left( \sup_{\substack{\rho,\rho'\in\mathcal R_\kappa \\ |\rho-\rho'|<\delta}} \sqrt{N}|T_N(\rho)-T_N(\rho')| > \gamma \zeta \right) \\
	&< \varepsilon
\end{align*}
for all large $n>n_0'(\varepsilon,\zeta)>n_0(\varepsilon)$ where $\zeta>0$ is also taken arbitrarily small. Thus we have, for each $\rho\in \mathcal R_\kappa$ and for $n>n_0'(\varepsilon,\zeta)$
\begin{align*}
	P\left( T_N(\rho) > \frac{\gamma}{\sqrt{N}} \right) &\leq P\left( T_N(\rho_i) > \frac{\gamma(1-\zeta)}{\sqrt{N}}  \right)+ P\left( \sqrt{N} |T_N(\rho)-T_N(\rho_i)|>\gamma\zeta\right) \\
	&\leq P\left( T_N(\rho_i) > \frac{\gamma(1-\zeta)}{\sqrt{N}} \right) + \varepsilon
\end{align*}
for $i\leq K$ the unique index such that $|\rho-\rho_i|<\delta$ and where the inequality holds uniformly on $\rho\in \mathcal R_\kappa$. Similarly, reversing the roles of $\rho$ and $\rho_i$,
\begin{align*}
	P\left(  T_N(\rho) > \frac{\gamma}{\sqrt{N}} \right) &\geq P\left(  T_N(\rho_i) > \frac{\gamma(1+\zeta)}{\sqrt{N}} \right) - \varepsilon.
\end{align*}
As a consequence, by \eqref{eq:Kfoldconv}, for $n>n_0'(\varepsilon,\zeta)$, uniformly on $\rho\in\mathcal R_\kappa$,
\begin{align*}
	P\left( T_N(\rho) > \frac{\gamma}{\sqrt{N}} \right) &\leq \exp\left( -\frac{\gamma^2(1-\zeta)^2}{2\sigma_N^2(\rho_i)} \right) + 2\varepsilon \\
	P\left( T_N(\rho) > \frac{\gamma}{\sqrt{N}} \right) &\geq \exp\left( -\frac{\gamma^2(1+\zeta)^2}{2\sigma_N^2(\rho_i)} \right) - 2\varepsilon
\end{align*}
which, by continuity of the exponential and of $\rho\mapsto \sigma_N(\rho)$,\footnote{Note that it is unnecessary to ensure $\liminf_N \sigma_N(\rho)>0$ as the exponential would tend to zero anyhow in this scenario.} letting $\zeta$ and $\delta$ small enough (up to growing $n_0'(\varepsilon,\zeta)$), leads to
\begin{align*}
	\sup_{\rho\in\mathcal R_\kappa}\left| P\left( \sqrt{N} T_N(\rho) > \gamma \right) - \exp\left( -\frac{\gamma^2}{2\sigma_N^2(\rho)} \right)\right| &\leq 3\varepsilon
\end{align*}
for all $n>n_0'(\varepsilon,\zeta)$, which completes the proof.

\subsection{Around empirical estimates}

This section is dedicated to the proof of Proposition~\ref{prop:1} and Corollary~\ref{co:1}.

We start by showing that $\hat{\sigma}^2_N(1)$ is well defined. It is easy to observe that the ratio defining $\hat{\sigma}^2_N(\underline\rho)$ converges to an undetermined form (zero over zero) as $\underline\rho\uparrow 1$. Applying l'Hospital's rule to the ratio, using the differentiation $\frac{d}{d\underline\rho} \underline{\hat{S}}^{-1}_N(\underline\rho)=-\underline{\hat{S}}^{-2}_N(\underline\rho)(I_N-\frac1n\sum_i z_iz_i^*)$ and the limit $\underline{\hat{S}}^{-1}_N(\underline\rho)\to I_N$ as $\underline\rho\uparrow 1$, we end up with
\begin{align*}
	\hat{\sigma}^2_N(\underline\rho) \to \frac12 \frac{p^*\left( \frac1n\sum_{i=1}^n z_iz_i^* \right)p}{\frac1N\tr \left( \frac1n\sum_{i=1}^n z_iz_i^* \right)}.
\end{align*}
Letting $\varepsilon>0$ arbitrary, since $p^*\frac1n\sum_i z_iz_i^*p - p^*C_Np\asto 0$, $\frac1N\tr \frac1n\sum_i z_iz_i^*\asto 1$ as $n\to\infty$, we immediately have, by continuity of both $\sigma^2_N(\underline\rho)$ and $\hat\sigma^2_N(\underline\rho)$,
\begin{align*}
	\sup_{\rho \in (1-\kappa,1]}\left|\hat{\sigma}^2_N(\underline\rho) - \sigma^2_N(\underline\rho) \right| &\leq \varepsilon
\end{align*}
for all large $n$ almost surely. From now on, it then suffices to prove Proposition~\ref{prop:1} on the complementary set $\mathcal R_\kappa'\triangleq [\kappa+\min\{0,1-c^{-1}\},1-\kappa]$.
For this, we first recall the following results borrowed from \citep{COU14}:
\begin{align*}
	\sup_{\rho\in\mathcal R_\kappa} \left\Vert \frac{\hat{C}_N(\rho)}{\frac1N\tr \hat{C}_N(\rho)} - \underline{\hat{S}}_N(\underline\rho) \right\Vert &\asto 0.
\end{align*}
Also, for $z\in\CC\setminus \RR^+$, defining
\begin{align*}
	\underline{\underline{\hat S}}_N(z)&\triangleq (1-\underline\rho)\frac1n\sum_{i=1}^n z_iz_i^* -z I_N
\end{align*}
(so in particular $\underline{\underline{\hat S}}_N(-\underline\rho)=\underline{\hat S}_N(\underline\rho)$, for all $\underline\rho\in \mathcal R_\kappa$), we have, with $\mathcal C$ a compact set of $\CC\setminus\RR^+$ and any integer $k$,
\begin{align*}
	\sup_{\bar z\in\mathcal C} \left| \frac{d^k}{dz^k} \left\{ \frac1N\tr \underline{\underline{\hat S}}_N^{-1}(z) - \frac1N\tr \left(-z \left[I_N + (1-\underline\rho)m_N(z)C_N\right]\right)^{-1} \right\}_{z=\bar z} \right| &\asto 0 \\
	\sup_{\bar z\in\mathcal C} \left| \frac{d^k}{dz^k} \left\{ p^*\underline{\underline{\hat S}}_N^{-1}(z)p - p^*\left(-z \left[I_N + (1-\underline\rho)m_N(z)C_N\right]\right)^{-1}p \right\}_{z=\bar z} \right| &\asto 0
\end{align*}
where $m_N(z)$ is defined as the unique solution with positive (resp.\@ negative) imaginary part if $\Im[z]>0$ (resp.\@ $\Im[z]<0$) or unique positive solution if $z<0$ of
\begin{align*}
	m_N(z) &= \left( -z + c \int \frac{(1-\underline\rho)t}{1+(1-\underline\rho)tm_N(z)}\nu_N(dt) \right)^{-1}
\end{align*}
(this follows directly from \citep{SIL95}).

This expression of $m_N(z)$ can be more rewritten under the more convenient form
\begin{align*}
	m_N(z) &= -\frac{1-c}z + c \int \frac{\nu_N(dt)}{-z-z(1-\underline\rho)tm_N(z)} \\
	&= -\frac{1-c}z + c \frac1N\tr \left(-z \left[I_N + (1-\underline\rho)m_N(z)C_N\right]\right)^{-1}
\end{align*}
so that, from the above relations
\begin{align*}
	\sup_{\rho \in\mathcal R_\kappa'} \left| m_N(-\underline\rho) - \left( \frac{1-c_N}{\underline\rho} + c_N \frac1N\tr \hat{C}_N^{-1}(\rho)\cdot\frac1N\tr \hat{C}_N(\rho) \right) \right| &\asto 0 \\
	\sup_{\rho \in\mathcal R_\kappa'} \left| \int \frac{t\nu_N(dt)}{1+(1-\underline\rho)m_N(-\underline\rho)t} - \frac{ 1 - \underline\rho \frac1N\tr \hat{C}_N^{-1}(\rho)\cdot\frac1N\tr \hat{C}_N(\rho)}{(1-\underline\rho)m_N(-\underline\rho)} \right| &\asto 0.
\end{align*}
Differentiating along $z$ the first defining identity of $m_N(z)$, we also recall that
\begin{align*}
	m_N'(z) &= \frac{m_N^2(z)}{1-c \int \frac{m_N(z)^2(1-\underline\rho)^2t^2\nu_N(dt)}{(1-(1-\underline\rho)tm_N(-\underline\rho))^2}}.
\end{align*}
Now, remark that
\begin{align*}
	p^*\underline{\underline{\hat S}}_N(\underline\rho)^{-2}p &= \frac{d}{dz} \left[ p^*\underline{\underline{\hat S}}_N(z)^{-1}p \right]_{z=-\underline\rho}
\end{align*}
which (by analyticity) is uniformly well approximated by
\begin{align*}
	& \frac{d}{dz} \left[ p^*\left( -z \left[ I_N + (1-\underline\rho)m_N(z) C_N\right]\right)^{-1}p \right]_{z=-\underline\rho} \nonumber \\
	&= \frac1{\underline\rho^2} p^*Q_N(\underline\rho)p - \frac1{\underline\rho} (1-\underline\rho) m_N'(-\underline\rho)p^*C_NQ^2_N(\underline\rho)p \\
	&= \frac1{\underline\rho^2} p^*Q_N(\underline\rho)p - \frac1{\underline\rho} (1-\underline\rho) \frac{m_N^2(-\underline\rho)p^*C_NQ^2_N(\underline\rho)p}{1-c m_N(-\underline\rho)^2(1-\underline\rho)^2 \frac1N\tr Q^2_N(\underline\rho)}. 
\end{align*}
(recall that $Q_N(\underline\rho)=\left(I_N+(1-\underline\rho)m_N(-\underline\rho)C_N \right)^{-1}$).
We then conclude
\begin{align*}
	&\sup_{\rho\in\mathcal R_\kappa'} \left| \frac{p^*C_NQ^2_N(\underline\rho)p}{1-c m_N(-\underline\rho)^2(1-\underline\rho)^2 \frac1N\tr Q^2_N(\underline\rho)} \right. \nonumber \\
	& \left. - \frac{p^*\hat{C}_N^{-1}(\rho)p\cdot \frac1N\tr \hat{C}_N(\rho)-\underline\rho p^*\hat{C}_N^{-2}(\rho)p\cdot \left( \frac1N\tr \hat{C}_N(\rho)\right)^2}{(1-\underline\rho)m_N(-\underline\rho)^2} \right| &\asto 0.
\end{align*}
Putting all results together, we obtain the expected result.

\bigskip

It now remains to prove Corollary~\ref{co:1}. This is easily performed thanks to Theorem~\ref{th:T} and Proposition~\ref{prop:1}.
From these, we indeed have the three relations
\begin{align*}
	P\left( \sqrt{N} T_N(\hat\rho_N^*) > \gamma \right) - \exp\left(- \frac{\gamma^2}{2 \sigma_N^2(\underline{\hat\rho}_N^*)} \right) &\asto 0 \\
	P\left( \sqrt{N} T_N(\rho_N^*) > \gamma \right) - \exp\left(- \frac{\gamma^2}{2 \sigma_N^2(\underline\rho_N^*)} \right) &\to 0 \\
	\exp\left(- \frac{\gamma^2}{2 \sigma_N^2(\underline{\hat\rho}_N^*)} \right) - \exp\left(- \frac{\gamma^2}{2 \sigma_N^{*2}} \right) &\asto 0
\end{align*}
where we denoted $\rho_N^*$ any element in the argmin over $\rho$ of $P(\sqrt{N}T_N(\rho)>\gamma)$ (and $\underline\rho_N^*$ its associated value through the mapping $\rho\mapsto\underline\rho$) and $\sigma_N^{*2}$ the minimum of $\sigma_N(\underline\rho)$ (i.e.\@ the minimizer for $\exp(- \frac{\gamma^2}{2 \sigma_N^2(\underline\rho)})$). Note that the first two relations rely fundamentally on the uniform convergence $\sup_{\rho\in\mathcal R_\kappa}|P\left( \sqrt{N} T_N(\rho) > \gamma \right)-\exp\left(-\gamma^2/(2\sigma_N^2(\underline\rho)) \right)|\asto 0$. By definition of $\rho_N^*$ and $\sigma_N^{*2}$, we also have
\begin{align*}
	\exp\left(- \frac{\gamma^2}{2 \sigma_N^{*2}} \right) &\leq \min\left\{ \exp\left(- \frac{\gamma^2}{2 \sigma_N^2(\underline{\hat\rho}_N^*)} \right) , \exp\left(- \frac{\gamma^2}{2 \sigma_N^2(\underline{\rho}_N^*)} \right) \right\} \\
	P\left( \sqrt{N} T_N(\rho_N^*) > \gamma \right) &\leq P\left( \sqrt{N} T(\hat\rho_N^*) > \gamma \right).
\end{align*}
Putting things together then gives
\begin{align*}
	P\left( \sqrt{N} T(\hat\rho_N^*) > \gamma \right) - P\left( \sqrt{N} T_N(\rho_N^*) > \gamma \right) \asto 0
\end{align*}
which is the expected result.

\appendix

\section{Proof of Lemma~\ref{le:keylemma1}}
\label{app:key_lemma}
This section is devoted to the proof of the key Lemma~\ref{le:keylemma1}.
The proof relies on an appropriate decomposition of the quantity under study as a sum of martingale differences. Before delving into the core of the proofs, we introduce some notations along with some of the key-lemmas that will be extensively used in this section.

In this section, $\EE_j$ will denote the conditional expectation with respect to the $\sigma-$ field $\mathcal{F}_j$ generated by the vectors $\left(z_\ell,1\leq \ell \leq j\right)$. By convention, $\EE_0=\EE$.
\begin{paragraph}
{\it Useful lemmas}
 
We shall review two key lemmas that will be extensively used, namely the generalized H\"older inequality as well as an instance of Jensen's inequality.
\begin{lemma}[Jensen Inequality, \citep{boyd}]
	Let $\mathcal{I}$ be a discrete set of elements of $\{1,\ldots,n\}$ with finite cardinality denoted by $\left|\mathcal{I}\right|$. Let $\left(\theta_i\right)_{i\in \mathcal{I}}$ be a sequence of complex scalars indexed by the set $\mathcal{I}$. Then, for any $p\geq 1$,
\begin{align*}
\left|\sum_{i\in\mathcal{I}}\theta_i\right|^p \leq \left|\mathcal{I}\right|^{p-1} \sum_{i=1}^n \left|\theta_i\right|^p
\end{align*}
\label{lemma:jensen}
\end{lemma}
\begin{lemma}[Generalized H\"older inequality,\citep{elkaroui}]
\label{lemma:holder}
Let $X_1,\cdots,X_k$ be $k$ complex random variables with finite moments of order $k$. Then,
\begin{align*}
\left|\EE\left[\prod_{i=1}^k X_i\right] \right|\leq \prod_{i=1}^k \left(\E \left[\left|X_i\right|^k\right]\right)^{\frac{1}{k}}.
\end{align*}
\end{lemma}
It remains to introduce the Burkh\"older inequalities on which the proof relies.
\begin{lemma}[Burkh\"older inequality \citep{burkholder}]
Let $\left(X_k\right)_{k=1}^n$ be a sequence of complex martingale differences sequence. For every $p\geq 1$, there exists $K_p$ dependent only on $p$ such that:
\begin{align*}
	\EE\left[\left|\sum_{k=1}^n X_k\right|^{2p}\right] \leq K_p n^{p}\max_k \EE\left[\left|X_k\right|^{2p}\right].
\end{align*}
\label{lemma:burkholder}
\end{lemma}
Letting $X_k=\left(\E_k-\E_{k-1}\right)z_k^*A_kz_k$ where $A_k$ is independent of $z_k$ and noting that $\EE\left[\left|X_k\right|^{2p}\right] \leq \EE\left[\left\|A_k\right\|_{\rm Fro}^{2p}\right]$, with $\left\|A\right\|_{\rm Fro}\triangleq\sqrt{\tr AA^*}$, we get in particular.
\begin{lemma}[Burkh\"older inequality for quadratic forms]
Let $z_1,\cdots,z_n \in \mathbb{C}^{N\times 1}$ be $n$ independent random vectors with mean $0$ and covariance $C_N$. Let $\left(A_j\right)_{j=1}^n$ be a sequence of $N\times N$ random matrices where for all $j$, $A_j$ is independent of $z_j$. Define $X_j$ as
\begin{align*}
X_j=\left(\EE_j-\EE_{j-1}\right) z_j^*A_jz_j=z_j^*\E_j A_j z_j-\tr \E_{j-1} C_NA_j.
\end{align*}
 Then,
\begin{align*}
	\EE\left[\left|\sum_{j=1}^n X_j\right|^{2p}\right] \leq K_p \left\|C_N\right\|_{\rm Fro}^{2p}n^{p}\max_j \EE\left[\left\|A_jC_N\right\|_{\rm Fro}^{2p}\right].
\end{align*}
\label{lemma:burkholder_quadratic}
\end{lemma}

\end{paragraph}

\begin{paragraph}
{\it Preliminaries.}

We start the proof by some preliminary results.
\begin{lemma}
Let $z_1,\cdots,z_n$ be as in Assumption \ref{ass:x}. Let $\first\in \mathbb{C}^{N\times 1}$ be  independent of $z_1,\cdots,z_n$ and such that $\E\|\first\|^k$ is bounded uniformly in $N$ for all order $k$. Then,  for any integer $p$, there exists $K_p$ such that
\begin{align*}
\EE\left[\left|z_i^*\SN \first\right|^p\right] \leq \EE\left[\left|z_i^*\Si \first\right|^p\right] \leq K_p.
\end{align*}
\label{lemma:control_quadratic}
\end{lemma}
\begin{proof}
The first inequality can be obtained from the following decomposition:
\begin{align*}
\SN z_i =\frac{\Si z_i}{1+\qnrho \frac{1}{n}z_i^*\Si z_i}
\end{align*}
while the second follows by noticing that $\EE\left|z_i^* \first\right|^p \leq \E\left(\first^*C_N\first\right)^{\frac{p}{2}}$. 
\end{proof}
Using the same kind of calculations, we can also control the order of magnitude of some interesting quantities.
\begin{lemma}
\label{lemma:many_results}
The following statements hold true:
\begin{enumerate}
\item Denote by $\Delta_{i,j}$ the quantity:
\begin{align*}
\Delta_{i,j}=\frac{1}{n}z_j^*\Sij z_j -\frac{1}{n}\tr {C}_N\Sij.
\end{align*}
Then, for any $p\geq 2$.
\begin{align*}
\EE\left|\Delta_{i,j}\right|^{p}=O(n^{-\frac{p}{2}}).
\end{align*}
\item Let $i$ and $j$ be two distinct integers from $\left\{1,\cdots,n\right\}$. Then,
\begin{align*}
\E\left|z_i^*\Sij z_j\right|^p=O(n^{\frac{p}{2}}).
\end{align*}
\item Let $z_i\in \mathbb{C}^{N\times 1}$ be as in Assumption~\ref{ass:x} and $A$ be a $N\times N$ random matrix independent of $z_i$ and having a bounded spectral norm.
 Then,
\begin{align*}
\E\left|z_i^*A z_i\right|^p=O(n^{p}).
\end{align*}

\item Let $j\in \left\{1,\cdots,n\right\}$ and $i$ and $k$ two distinct integers different from $j$. Then:
\begin{align*}
\E\left|z_i^*\Sij \Sjk z_k\right|^p =O(n^{\frac{p}{2}}).
\end{align*}
\end{enumerate}
\end{lemma}
\begin{proof}
Item 1) and 3) are standard results that are a by-product of \citep[Lemma~B.26]{SIL06}, while Item~2) can be easily obtained from Lemma~\ref{lemma:control_quadratic}. As for item 4), it follows by first decomposing $\Sij$ and $\Sjk$ as:
\begin{align*}
\Sij=\Sijk-\frac{1}{n}\qnrho \frac{\Sijk z_kz_k^*\Sijk}{1+\frac{1}{n}\qnrho z_k^*\Sijk z_k} \\
\Sjk=\Sijk -\frac{1}{n}\qnrho \frac{\Sijk z_iz_i^*\Sijk}{1+\frac{1}{n}\qnrho z_i^*\Sijk z_i}
\end{align*}
The above relations serve to better control the dependencies of $\Sij$ and $\Sjk$ on $z_k$ and $z_i$. Plugging the above decompositions on $z_i^*\Sij \Sjk z_k$, we obtain
\begin{align*}
z_i^*\Sij\Sjk z_k&=z_i^*\hat{S}_{(i,j,k)}^{-2}z_k-\frac{1}{n}\frac{\alpha(\rho)}{\gamma_N(\rho)} \frac{z_i^*\Sijk z_kz_k^*\hat{S}_{(i,j,k)}^{-2}z_k}{1+\frac{1}{n}\frac{\alpha(\rho)}{\gamma_N(\rho)}z_k^*\Sijk z_k}\\
&-\frac{1}{n}\qnrho \frac{z_i^*\hat{S}_{(i,j,k)}^{-2}z_iz_i^*\Sijk z_k}{1+\frac{1}{n}\qnrho z_i^*\Sijk z_i} \\
&+\frac{1}{n^2}\left(\qnrho\right)^2 \frac{z_i^*\Sijk z_kz_k^*\hat{S}_{(i,j,k)}^{-2}z_iz_i^*\Sijk z_k}{\left(1+\frac{1}{n}\frac{\alpha(\rho)}{\gamma_N(\rho)}z_k^*\Sijk z_k\right)\left(1+\frac{1}{n}\qnrho z_i^*\Sijk z_i\right)}.
\end{align*}
The control of these four terms follows from a direct application of item 2) and 3) along with possibly the use of the generalized H\"older inequality in Lemma~\ref{lemma:holder}.
\end{proof}

\end{paragraph}

\begin{paragraph}
	{\it Core of the proof.}

With these preliminaries results at hand, we are now in position to get into the core of the proof. Let $\beta_N$ be given by
\begin{align*}
\beta_N=\frac{1}{n}\sum_{i=1}^n \first^* \SN z_iz_i^* \SN \second\left(\frac{1}{N}z_i^*\Si z_i -\gammanrho\right).
\end{align*}
Decompose $\beta_N$ as
\begin{align*}
\beta_N&=\frac{1}{n}\sum_{i=1}^n \first^* \SN z_iz_i^* \SN \second\left(\frac{1}{N}z_i^*\Si z_i -\frac{1}{N}\tr C_N\Si\right)\\
&+\frac{1}{n}\sum_{i=1}^n \first^* \SN z_iz_i^* \SN \second\left(\frac{1}{N}\tr C_N\Si - \gammanrho\right)\\
&\triangleq\beta_{N,1}+\beta_{N,2}.
\end{align*} 
The control of $\beta_{N,2}$ follows from a direct application of Lemma~\ref{lemma:jensen} and Lemma~\ref{lemma:holder}, that is
\begin{align*}
\E \left[\left|\beta_{N,2}\right|^{2p}\right] &\leq \frac{n^{2p-1}}{n^{2p}}\sum_{i=1}^n \E\left|\first^* \SN z_i\right|^{2p} \left|z_i^*\SN \second\right|^{2p} \left|\frac{1}{N}\tr C_N\Si -\gammanrho\right|^{2p}\\
&\leq \frac{n^{2p-1}}{n^{2p}} \sum_{i=1}^n \left(\E\left|\first^* \SN z_i\right|^{6p}\right)^{\frac{1}{3}} \left(\E\left|z_i^*\SN \second\right|^{6p}\right)^{\frac{1}{3}} \left(\E\left|\frac{1}{N}\tr C_N\Si -\gammanrho\right|^{6p}\right)^{\frac{1}{3}}
\end{align*}
By standard results from random matrix theory (e.g.\@ \cite[Prop.~7.1]{NAJ13}), we know that
\begin{align*}
\EE\left|\frac{1}{N}\tr C_N\Si -\gammanrho\right|^{6p} =O(n^{-6p})
\end{align*}
Hence, by Lemma \ref{lemma:control_quadratic}, we finally get:
\begin{align*}
\EE\left|\beta_{N,2}\right|^{2p} =O(n^{-2p}).
\end{align*} 
While the control of $\beta_{N,2}$ requires only the manipulation of conventional moment bounds due to the rapid convergence of $\frac{1}{N}\tr C_N\Si - \gammanrho$, the analysis of $\beta_{N,1}$ is more intricate since
\begin{align*}
\E\left|\frac{1}{N}z_i^{*}\Si z_i -\frac{1}{N}\tr C_N\Si \right|^{p} =O(n^{-\frac{p}{2}})
\end{align*}
a convergence rate which seems insufficient at the onset. The averaging occurring in $\beta_{N,2}$ shall play the role of improving this rate. To control $\beta_{N,1}$, one needs to resort to advanced tools based on Burkh\"older inequalities. First, decompose $\beta_{N,1}$ as
\begin{align*}
\beta_{N,1}=\stackrel{o}{\beta}_{N,1} +\EE\left[\beta_{N,1}\right].
\end{align*}
As in Lemma~\ref{lemma:many_results}, define $\Delta_i\triangleq\frac{1}{n}z_i^*\Si z_i-\frac{1}{n}\tr C_N\Si$. Using the relation
\begin{align*}
\SN z_i =\frac{\Si z_i}{1+\frac{1}{n}\qnrho z_i^*\Si z_i}
\end{align*}
we get
\begin{align*}
\E \left[\beta_{N,1}\right] &=\E\left[\frac{1}{N}\sum_{i=1}^n \frac{\first^* \Si z_iz_i^*\Si \second}{\left(1+\frac{1}{n}\qnrho z_i^*\Si z_i\right)^2} \Delta_i\right]\\
&=\E\left[\frac{1}{N}\sum_{i=1}^n \frac{\first^* \Si z_iz_i^*\Si \second}{\left(1+\frac{1}{n}\qnrho \tr C_N\Si\right)^2} \Delta_i\right] \\
&-\qnrho\EE\left[ \frac{1}{N}\sum_{i=1}^n\frac{\first^*\Si z_iz_i^*\Si\second \Delta_i^2 \left(2+\left(\qnrho\right)\left(\frac{1}{n}z_i^*\Si z_i+\frac{1}{n}\tr C_N\Si\right)\right)}{\left(1+\frac{1}{n}\qnrho \tr C_N\Si\right)^2\left(1+\qnrho\frac{1}{n}z_i^*\Si z_i\right)^2}\right]\\
&\triangleq\beta_{N,1,1} +\beta_{N,1,2}
\end{align*}
Since $\EE\left[w^*Aw\left(w^*Bw-\tr B\right)\right]=\EE\tr AB$ when $w$ is standard complex Gaussian vector and $A,B$ random matrices independent of $w$, we have
\begin{align*}
\EE\left[\beta_{N,1,1}\right]&=\frac{1}{Nn}\EE\left[\tr\frac{C_N\Si C_N \Si d c^* \Si}{\left(1+\qnrho \frac{1}{n}\tr C_N\Si\right)^2}\right] =O(n^{-1}).
\end{align*}
As for $\beta_{N,1,2}$, we have for some $K>0$, again by Lemma~\ref{lemma:many_results}
\begin{align*}
\left|\beta_{N,1,2}\right|& \leq \frac{K}{n}\sum_{i=1}^n \left(\EE\left|\first^*\Si z_i\right|^4\right)^{\frac{1}{4}}\left(\EE\left|z_i^*\Si \second\right|^4\right)^{\frac{1}{4}} \left(\EE\left|\Delta_i\right|^8\right)^{\frac{1}{4}}\\
&\times \left(\EE\left|2+\qnrho\left(\frac{1}{n}z_i^*\Si z_i +\frac{1}{n}\tr C_N\Si\right)\right|^4\right)^{\frac{1}{4}} =O(\frac{1}{n}).
\end{align*}
We therefore have
\begin{align*}
\left|\EE\left[\beta_{N,1}\right]\right|^{2p} =O(n^{-2p}).
\end{align*}
Let's turn to the control of $\stackrel{o}{\beta}_{N,1}$.  For that, we decompose  $\stackrel{o}{\beta}_{N,1}$ as a sum of martingale differences as
\begin{align*}
\stackrel{o}{\beta}_{N,1}=\sum_{j=1}^n \left(\E_j-\E_{j-1}\right)\beta_{N,1}
\end{align*}
The control of $\EE\left[\left|\stackrel{o}{\beta}_{N,1}\right|^p\right]$ requires the convergence rate of two kinds of martingale differences:
%The problem thus unfolds to analyzing the behaviour of $\stackrel{o}{\beta}_{N,1}$. For that, we decompose $\stackrel{o}{\beta}_{N,1}$ as a sum of martingale differences. Indeed, $\stackrel{o}{\beta}_{N,1}$ can be written as: 
\begin{itemize}
\item Sum of martingale differences with a quadratic form representation of the form %involve terms with only one incidence of $z_j$ and $z_j^*$. They can be written as:
\begin{align*}
\sum_{j=1}^n \left(\E_j-\E_{j-1}\right)z_j^* A_j z_j.
\end{align*}
For these terms, from  Lemma~\ref{lemma:burkholder_quadratic}, it will be sufficient to show that  $\max_j\EE\|A_j\|_{\rm Fro}^{2p}=O(n^{-3p})$ in order to obtain the required convergence rate. 
\item Sum of martingale differences with more than one occurrence of $z_j$ and $z_j^*$. In this case, this sum is given by:
\begin{align*}
\sum_{j=1}^n (\E_j-\E_{j-1})\sum_{i=1,i\neq j}^n \varepsilon_i
\end{align*}
where $\varepsilon_j$ are small random quantities depending on $z_1,\cdots,z_n$. According to Lemma~\ref{lemma:burkholder}, we have
\begin{align*}
\left|\sum_{j=1}^n (\E_j-\E_{j-1}) \sum_{i=1,i\neq j }^n \varepsilon_i\right|^{2p}=O(n^{-2p})
\end{align*}
provided that
\begin{align*}
\E\left|\sum_{i=1,i\neq j}\varepsilon_i\right|^{2p}=O(n^{-3p}).
\end{align*}
The control of the above sum will rely on successively using Lemma~\ref{lemma:jensen} to get
\begin{align*}
\E\left|\sum_{i=1,i\neq j}\varepsilon_i\right|^{2p} \leq n^{2p-1}\sum_{i=1}^n \EE\left|\varepsilon_i\right|^{2p}
\end{align*}
and controlling $\max_i \EE\left|\varepsilon_i\right|^{2p}$.
\end{itemize}

With this explanation at hand, we will now get into the core of the proofs. We first have
\begin{align*}
\stackrel{o}{\beta}_{N,1}&=\sum_{j=1}^n(\E_j-\E_{j-1})\frac{1}{N}\sum_{i=1}^n \first^*\SN z_iz_i^*\second \Delta_i \\
&=\sum_{j=1}^n (\E_j-\E_{j-1})\first^*\SN z_jz_j^*\second \Delta_j \\
&+\sum_{j=1}^n(\E_j-\E_{j-1})\frac{1}{N}\sum_{i=1,i\neq j}^n \first^*\SN z_iz_i^*\second \Delta_i\\
&\triangleq\sum_{j=1}^n W_{j,1}+\sum_{j=1}^n W_{j,2}.
\end{align*}
In order to prove that $\EE\left|\sum_{j=1}^n W_{j,1}\right|=O(n^{-2p})$, it is sufficient to show 
\begin{align*}
\EE\left|W_{j,1}\right|=O(n^{-3p})
\end{align*}
a statement which holds true since, by Lemma~\ref{lemma:holder}
\begin{align*}
\EE\left|W_{j,1}\right|^{2p} &\leq \frac{K}{n^{2p}}\EE\left|\first^*\SN z_j\right|^{2p}\left|z_j^*\SN \second\right|^{2p} \Delta_j^{2p} \\
&\leq \frac{K}{n^{2p}}\left(\EE\left|\first^*\SN z_j\right|^{6p}\right)^{\frac{1}{3}}\left(\EE\left|z_j^*\SN \second\right|^{6p}\right)^{\frac{1}{3}} \left(\EE\Delta_j^{6p}\right)^{\frac{1}{3}}\\
&=O(n^{-3p}).
\end{align*}
We now consider the more involved term $\sum_{j=1}^n W_{j,2}$. Using the relation
%To this end, we proceed to handle the dependence of $\SN$ on $z_j$ by using the relation:
\begin{align*}
\SN=\Sj-\qnrho \frac{1}{n}\frac{\Sj z_jz_j^*\Sj}{1+\qnrho \frac{1}{n}z_j^*\Sj z_j}
\end{align*}
to let the independent $\Sj$ and $z_j$ variables appear, we write
\begin{align*}
&\sum_{j=1}^n W_{j,2}=\sum_{j=1}^n (\E_j-\E_{j-1})\frac{1}{n}\sum_{i=1,i\neq j}^n \first^*\Sj z_iz_i^*\Sj \second \left(\frac{1}{N}z_i^*\Si z_i -\frac{1}{N}\tr C_N\Si\right) \\
&-\sum_{j=1}^n (\E_j-\E_{j-1})\qnrho \frac{1}{n^2}\sum_{i=1,i\neq j}^n \frac{\first^*\Sj z_jz_j^*\Sj z_iz_i^*\Sj \second}{1+\frac{1}{n}\qnrho z_j^*\Sj z_j} \left(\frac{1}{N}z_i^*\Si z_i -\frac{1}{N}\tr C_N\Si\right) \\
&-\sum_{j=1}^n (\E_j-\E_{j-1})\qnrho \frac{1}{n^2}\sum_{i=1,i\neq j}^n \frac{\first^*\Sj z_iz_i^*\Sj z_jz_j^*\Sj \second}{1+\frac{1}{n}\qnrho z_j^*\Sj z_j} \left(\frac{1}{N}z_i^*\Si z_i -\frac{1}{N}\tr C_N \Si\right)\\
&+\sum_{j=1}^n(\E_j-\E_{j-1})\left(\qnrho\right)^2 \frac{1}{n^3}\sum_{i=1,i\neq j}^n \frac{\first^*\Sj z_jz_j^*\Sj z_iz_i^*\Sj z_jz_j^*\Sj \second}{\left(1+\frac{1}{n}\qnrho z_j^*\Sj z_j\right)^2} \left(\frac{1}{N}z_i^*\Si z_i -\frac{1}{N}\tr C_N\Si\right)\\
&\triangleq\chi_1+\chi_2+\chi_3+\chi_4.
\end{align*}
Next, we will sequentially control $\chi_i,i=1,\cdots,4$. 

\begin{paragraph}
{Control of $\chi_1$.}
Using the relation
\begin{align*}
\Si =\Sij -\qnrho\frac{1}{n}\frac{\Sij z_jz_j^*\Sij}{1+\frac{1}{n}\qnrho z_j^*\Sij z_j}
\end{align*}
the quantity $\chi_1$ can be decomposed as
\begin{align*}
\chi_1&=\sum_{j=1}^n -(\E_j-\E_{j-1}) \qnrho \frac{1}{n^2N}\sum_{i=1,i\neq j}^n \frac{\first^*\Sj z_iz_i^*\Sj \second\left|z_i^*\Sij z_j\right|^2}{1+\frac{1}{n}\qnrho z_j^*\Sij z_j} \\
&+\sum_{j=1}^n \qnrho (\E_j-\E_{j-1})\frac{1}{n^2N}\sum_{i=1,i\neq j}^n \frac{\first^*\Sj z_iz_i^*\Sj \second z_j^*\Sij C_N \Sij z_j}{1+\frac{1}{n}\qnrho z_j^*\Sij z_j}\\
&\triangleq\chi_{1,1}+\chi_{1,2}.
\end{align*}
where we used the fact that for $r_j$ random quantity independent of $z_j$, $(\E_j-\E_{j-1})(r_j)=0$.
We will begin by controlling $\chi_{1,1}$. To handle the quadratic forms in the denominator, we further develop $\chi_{1,1}$ as
\begin{align*}
\chi_{1,1}&=-\sum_{j=1}^n (\E_j-\E_{j-1}) \qnrho \frac{1}{n^2N}\sum_{i=1,i\neq j}^n \frac{\first^*\Sj z_iz_i^*\Sj \second \left|z_i^*\Sij z_j\right|^2}{1+\frac{1}{n}\qnrho \tr C_N \Sij} \\
&+\sum_{j=1}^n (\E_j-\E_{j-1})\left(\qnrho\right)^2 \frac{1}{n^2 N} \sum_{i=1,i\neq j}^n \frac{\first^*\Sj z_iz_i^*\Sj \left|z_i^*\Sij z_j\right|^2\Delta_{i,j}}{\left(1+\frac{1}{n}\qnrho \tr C_N\Sij\right)\left(1+\frac{1}{n}\qnrho z_j^*\Sij z_j \right)}\\
&=\sum_{j=1}^n X_{j,1}+\sum_{j=1}^n X_{j,2}.
\end{align*}
To control $\sum_{j=1}^n X_{j,1}$, we resort to Lemma \ref{lemma:burkholder_quadratic}. Indeed, $X_{j,1}$ can be written as
\begin{align*}
X_{j,1}=-\qnrho (\E_j-\E_{j-1}) z_j^*A_jz_j
\end{align*}
where $A_j$ is given by
\begin{align*}
A_j=\frac{1}{n^2N} \sum_{i=1,i\neq j}^n \frac{\first^*\Sj z_iz_i^*\Sj d}{1+\frac{1}{n}\qnrho \tr C_N\Sij}\Sij z_iz_i^*\Sij.
\end{align*}
According to Lemma~\ref{lemma:burkholder_quadratic}, it is sufficient to prove that $\EE\left\|A_j\right\|_{\rm Fro}^{2p} =O(n^{-3p})$.
Expanding $\EE\left\|A_j\right\|_{\rm Fro}^{2p}$, we indeed get
\begin{align*}
\EE\left\|A_j\right\|_{\rm Fro}^{2p}&\leq \frac{K}{n^{6p}} \EE\left|\sum_{i\neq j}\sum_{k\neq j} \frac{\left|z_k^*\Sjk \Sij z_i\right|^2 \first^*\Sj z_iz_i^*\Sj \second \second^*\Sj z_kz_k^*\Sj \first }{\left(1+\qnrho\frac{1}{n}\tr C_N\Sij\right)\left(1+\frac{1}{n}\qnrho \tr C_N\Sjk\right)}\right|^p \\
&\leq \frac{K}{n^{6p}} \EE\left|\sum_{i\neq j} \left|z_i^*\Sijsquare z_i\right|^2\left|\first^*\Sj z_iz_i^*\Sj \second \right|^2 \right|^p \\
&+\frac{K}{n^{6p}}\EE\left|\sum_{i\neq j}\sum_{\substack{k\neq j\\k\neq i}}\frac{\left|z_k^*\Sjk \Sij z_i\right|^2 \first^*\Sj z_iz_i^*\Sj \second \second^*\Sj z_kz_k^*\Sj \first }{\left(1+\qnrho\frac{1}{n}\tr C_N\Sij\right)\left(1+\frac{1}{n}\qnrho \tr C_N\Sjk\right)}\right|^p \\
&\leq \frac{Kn^{p-1}}{n^{6p}} \EE\left|z_i^*\Sijsquare z_i\right|^{2p} \left|\first^*\Sj z_iz_i^*\Sj \second\right|^{2p} \\
&+\frac{Kn^{2(p-1)}}{n^{6p}} \sum_{i\neq j} \sum_{\substack{k\neq j \\ k\neq i}} \EE\left|z_k^*\Sjk  \Sij z_i\right|^{2p} \left|\first^*\Sj z_iz_i^*\Sj\second\right|^{p} \left|\second^*\Sjk z_kz_k^*\Sjk \first\right|^p \\
&\leq \frac{K n^{p-1}}{n^{6p}} \sum_{i\neq j}\left(\EE\left|z_i^*\Sijsquare z_i\right|^{6p}\right)^{\frac{1}{3}} \left(\left|\first^*\Sj z_i\right|^{6p}\right)^{\frac{1}{3}}\left(\left|z_i^*\Sj \second\right|^{6p}\right)^{\frac{1}{3}}\\
&+\frac{Kn^{2(p-1)}}{n^{6p}} \sum_{i\neq j} \sum_{\substack{k\neq j \\ k\neq i}}\left(\EE\left|z_k^*\Sjk  \Sij z_i\right|^{10p}\right)^{\frac{1}{5}} \left(\EE\left|\first^*\Sij z_i\right|^{5p}\right)^{\frac{1}{5}}\\
&\times \left(\EE\left|z_i^*\Sj\second\right|^{5p}\right)^{\frac{1}{5}} 
\left(\EE\left|\second^*\Sjk z_k\right|^{5p}\right)^{\frac{1}{5}} \left(\EE\left|z_k^*\Sjk \first\right|^{5p}\right)^{\frac{1}{5}}\\
&=O(n^{-3p}).
\end{align*}
As for $X_{j,1}$, we can show that $\EE\left|X_{j,1}\right|^{2p} =O(n^{-3p})$. Indeed, we have
\begin{align*}
\EE\left|X_{j,2}\right|^{2p}&\leq \frac{Kn^{2p-1}}{n^{6p}}\sum_{i\neq j}\left(\EE\left|\first^*\Sj z_i\right|^{8p}\right)^{\frac{1}{4}} \left(\EE\left|z_i^*\Sj \second\right|^{8p}\right)^{\frac{1}{4}} \left(\EE\left|z_i^*\Sj z_j\right|^{16p}\right)^{\frac{1}{4}} \left(\EE\left|\Delta_{i,j}\right|^{8p}\right)^{\frac{1}{4}} \\
&=O(n^{-3p}).
\end{align*}
The Burkh\"older inequality shows that this rate of convergence of the moment of $X_{j,1}$ and $X_{j,2}$ is sufficient to finally ensure that $\EE\left|\chi_{1,1}\right|^{2p} =O(n^{-2p})$.

We study next $\chi_{1,2}$. First, decompose $\chi_{1,2}$ as
\begin{align*}
\chi_{1,2}&=\sum_{j=1}^n (\E_j-\E_{j-1})\frac{1}{n^2 N}\sum_{i\neq j} \qnrho\frac{\first^*\Sj z_iz_i^*\Sj \second z_j^*\Sij C_N \Sij z_j}{1+\frac{1}{n}\qnrho \tr C_N\Sij} \\
&-\sum_{j=1}^n (\E_j-\E_{j-1})\frac{1}{n^2N}\sum_{i\neq j} \qnrho \frac{\first^*\Sj z_iz_i^*\Sj \second \Delta_{i,j}z_j^*\Sij C_N \Sij z_j}{\left(1+\qnrho \frac{1}{n}z_j^*\Sij z_j\right)\left(1+\frac{1}{n}\qnrho \tr C_N\Sij\right)} \\
&\triangleq \sum_{j=1}^n Y_{j,1}+\sum_{j=1}^n Y_{j,2}.
\end{align*}
The quantities $\sum_{j=1}^n Y_{j,1}$ and $\sum_{j=1}^n Y_{j,2}$ are   differences of martingales whose controls follow the same procedure as above. While $\sum_{j=1}^n Y_{j,1}$ can be controlled using Lemma~\ref{lemma:burkholder_quadratic}, the convergence of $\sum_{j=1}^n Y_{j,2}$ is faster due to the term $\Delta_{i,j}$. Details are thus omitted.
\end{paragraph}
\begin{paragraph}
{Control of $\chi_2$.} The control of $\chi_2$ cannot be exactly dealt with using the same procedure. As for $\chi_1$, one works out $\chi_2$ by substituting  $\frac{1}{n}z_j^*\Sj z_j$ by its approximate $\frac{1}{n}\tr C_N\Sj$ and using the decomposition of $\Si$ as a function of $\Sij$ to get
\begin{align*}
\chi_2=-\qnrho \sum_{j=1}^n (\E_j-\E_{j-1})\frac{1}{n^2}\sum_{\substack{i=1\\ i\neq j}}\frac{\first^*\Sj z_jz_j^*\Sj z_iz_i^*\Sj \second\left(\frac{1}{N}z_i^*\Sij z_i- \frac{1}{N}\tr C_N\Sij\right)}{1+\frac{1}{n}\qnrho \tr C_N\Sj} +\varepsilon
\end{align*}
where we easily obtain that $\EE[\left|\varepsilon\right|^{2p}]=O(n^{-2p})$. We omit the details of this step, since the calculations are the same as those used for the control of $\chi_1$. The control of the Frobenius norm of the underlying matrices using the same techniques as above does not yield the required convergence rate. We will thus pursue a different approach. Precisely, we write $\chi_2$ as
% The above martingale difference have a quadratic form representation. However, the control of the Frobenuis norm of the underlying matrix using the same techniques as above, does not yield the required convergence rate. We will thus pursue a different approach. Indeed, $\chi_2$ can be written as:
\begin{align*}
\chi_2=-\qnrho \sum_{j=1}^n  (\E_j-\E_{j-1}) T_j +\varepsilon
\end{align*}
with 
\begin{align*}
T_j=\frac{1}{n^2} \frac{\first^*\Sj z_jz_j^*\Sj Z_jD_j Z_j^*\Sj \second}{1+\frac{1}{n}\qnrho \tr \Sj}
\end{align*}
where $Z_j=\left[z_1,\cdots,z_{j-1},z_{j+1},\cdots,z_n\right]$ and $D_j$ is a diagonal matrix with diagonal elements: $\left[D_j\right]_{i,i}=\frac{n}{N}\Delta_{j,i}$. Hence, by Lemma~\ref{lemma:burkholder}
\begin{align*}
\EE\left|T_j\right|^{2p} &\leq \frac{1}{n^{4p}} \EE\left|\first^*\Sj z_j\right|^{2p}\left|z_j^*\Sj Z_j D_j Z_j^* \Sj \second\right|^{2p} \\
&\leq \frac{1}{n^{4p}}\left(\EE\left|\first^*\Sj z_j\right|^{4p}\right)^{\frac{1}{2}}\left(\EE\left|z_j^*\Sj Z_jD_jZ_j^*\second\right|^{4p}\right)^{\frac{1}{2}}
\end{align*}
Since $D_j$ is independent of $z_j$, applying the inequality $\EE\left|z_j^* u\right|^p\leq \EE\left(u^*C_N u\right)^{\frac{p}{2}}$, we finally get
\begin{align*}
\EE\left|T_j\right|^{2p}&\leq \frac{K}{n^{4p}} \left(\EE\left|\second^*\Sj Z_j D_jZ_j^* \Sj C_N \Sj Z_jD_jZ_j^*\Sj \second\right|^{2p}\right)^{\frac{1}{2}} \\
&=\frac{K}{n^{3p}} \left(\EE\left|\second^*\Sj Z_jD_j\frac{Z_j^*\Sj C_N \Sj Z_j}{n}D_jZ_j^*\Sj \second\right|^{2p}\right)^{\frac{1}{2}} \\
&\stackrel{(a)}{\leq} \frac{K}{n^{3p}} \left(\EE\left\|D_j Z_j^*\Sj \second\right\|^{4p}\right)^{\frac{1}{2}}
\end{align*}
where $(a)$ follows since $\left\|\frac{Z_j^*\Sj C_N \Sj Z_j}{n}\right\|$ is bounded. In order to prove that $\EE[\left|T_j\right|^{2p}]=O(n^{-3p})$, it suffices to check that $\EE[\left\|D_j Z_j^*\Sj \second\right\|^{4p}]$ is uniformly bounded in $N$. Expanding this quantity, we indeed get
\begin{align*}
\EE\left\|D_j Z_j^*\Sj \second\right\|^{4p}&=\EE\left|\sum_{\substack{i=1\\i\neq j}}^n \left(\frac{1}{N}z_i^*\Sij z_i -\frac{1}{N}\tr C_N\Sij\right)^2\left|z_i^*\Sj \second\right|^2\right|^{2p} \\
&\leq n^{2p-1}\sum_{i=1}^n \EE  \left(\frac{1}{N}z_i^*\Sij z_i -\frac{1}{N}\tr C_N\Sij\right)^{4p} \left|z_i^*\Sj \second\right|^{4p} \\
&\leq n^{2p-1}\sum_{i=1}^n \left(\EE  \left(\frac{1}{N}z_i^*\Sij z_i -\frac{1}{N}\tr C_N\Sij\right)^{8p}\right)^{\frac{1}{2}} \left(\EE\left|z_i^*\Sj \second\right|^{8p}\right)^{\frac{1}{2}} \\
&=O(1).
\end{align*}
The control of $\chi_3$ is similar to that of $\chi_2$, while that of $\chi_4$ follows immediately by using sequentially Lemma~\ref{lemma:jensen} along with the generalized H\"older inequality in Lemma~\ref{lemma:holder}. This completes the proof.
\end{paragraph}

\end{paragraph}

\bibliographystyle{elsarticle-harv}
\bibliography{/home/romano/Documents/PhD/phd-group/papers/rcouillet/tutorial_RMT/book_final/IEEEabrv.bib,/home/romano/Documents/PhD/phd-group/papers/rcouillet/tutorial_RMT/book_final/IEEEconf.bib,/home/romano/Documents/PhD/phd-group/papers/rcouillet/tutorial_RMT/book_final/tutorial_RMT.bib,./bib_file.bib}
\end{document}